\documentclass[aos,preprint]{imsart}

\RequirePackage{amsthm,amsmath,amsfonts,amssymb}
\RequirePackage[numbers,sort&compress]{natbib}
\RequirePackage[colorlinks,citecolor=blue,urlcolor=blue]{hyperref}
\RequirePackage{graphicx}

\usepackage{lmodern} 
\usepackage{babel} 
\usepackage[utf8]{inputenc} 
\usepackage[T1]{fontenc} 
\usepackage{accents,pifont}

\usepackage[shortlabels]{enumitem} 
\usepackage{dsfont,tikz-cd,hyperref,soul}

\newcommand{\cmark}{\ding{51}}
\newcommand{\xmark}{\ding{55}}

\newtheorem{theorem}{Theorem}

\newtheorem{lemma}{Lemma} 
\newtheorem{proposition}{Proposition}
\newtheorem{definition}{Definition}

\newcommand{\E}{\mathbb{E}}
\newcommand{\R}{\mathbb{R}}
\newcommand{\Var}{\mathrm{Var}}
\newcommand{\Cov}{\mathrm{Cov}}
\newcommand{\dd}{\mathrm{d}}
\newcommand{\bX}{\mathbf{X}}
\newcommand{\bY}{\mathbf{Y}}
\newcommand{\bZ}{\mathbf{Z}}
\renewcommand{\L}{\mathcal{L}}

\newcommand{\interior}{\operatorname{int}}

\newcommand{\cent}[1]{\accentset{\circ}{#1}}

\begin{document}

\begin{frontmatter}
\title{Critical point processes obtained from a Gaussian random field with a view toward statistics}
\runtitle{Critical point processes}

\begin{aug}
\author[A]{\fnms{Julien}~\snm{Chevallier} \ead[label=e1]{Julien.Chevallier1@univ-grenoble-alpes.fr}},
\author[A]{\fnms{Jean-François}~\snm{Coeurjolly} \ead[label=e2]{Jean-Francois.Coeurjolly@univ-grenoble-alpes.fr}
\thanks{[\textbf{Corresponding author}]}}
\and
\author[B]{\fnms{Rasmus}~\snm{Waagepetersen}\ead[label=e3]{rw@math.aau.dk}}

\address[A]{Univ. Grenoble Alpes, CNRS, LJK, 38000 Grenoble, France \printead[presep={,\ }]{e1,e2}}
\address[B]{Department of Mathematics, Aalborg University, Denmark \printead[presep={,\ }]{e3}}

\end{aug}

\begin{abstract}
This paper establishes the theoretical foundation for
statistical applications of an intriguing new type of spatial point
processes called critical point processes. These point processes,
residing in Euclidean space,  consist of the critical points of 
latent smooth Gaussian random fields or of subsets of critical points like minima, saddle
points etc.  Despite of the simplicity of their definition, the
mathematical analysis of critical point processes is non-trivial
involving for example deep results on the geometry of random fields,
Sobolev space theory, chaos expansions, and multiple Wiener-It{\^ o}
integrals.  We provide explicit expressions for fundamental moment
characteristics used in spatial point process statistics like the intensity
parameter, the pair correlation function, and higher order intensity
functions. The crucial dependence structure
(attraction or repulsiveness) of
a critical point process is discussed in depth and is in particular
related to the dimension of the points and the type of critical points
(extrema, saddle points, or all of the critical points). We propose
simulation strategies based on spectral methods or
smoothing of grid-based simulations and
show that resulting approximate critical point process simulations
asymptotically converge to  the exact critical point process
distribution. Finally, under the increasing domain framework, we
obtain asymptotic results for linear and bilinear statistics of a  critical point process. In particular, we obtain a multivariate central limit theorem for the intensity parameter estimate and a modified version of Ripley's $K$-function.
\end{abstract}


\end{frontmatter}

\tableofcontents


\section{Introduction}

 Point patterns emerge in a huge variety of applications, e.g.\ in cognitive sciences to model eye movement data~\cite{barthelme2013modeling},
in economics to model locations of industries~\cite{sweeney2016localization}, in
forestry to model locations of
trees~\cite{rajala2018detecting,choiruddin2020regularized} or in
environmental sciences to model impacts of lightning
strikes~\cite{coeurjolly2024spatio}. Such datasets are modelled by
point processes, covered from theoretical and practical points of
view by~\cite{daley2006introductionI,daley2007introductionII,moller2003statistical,illian2008statistical,baddeley2016spatial}.

During the last forty years, diverse types of spatial point process
models have been developed to model dependence between points. These
include (a) log Gaussian and shot noise Cox processes defined in terms
of  a random intensity function and typically used for clustered point
patterns \cite{moller1998log,moller2003statistical}; (b) Gibbs point
processes \cite[see e.g.][]{dereudre2019introduction} defined via a
density with respect to a Poisson point process; (c) Determinantal
point processes \cite{lavancier2015determinantal}, a  class
of repulsive models with intensity functions defined as determinants of a
positive definite kernel function; and (d) Perturbed lattice point processes
\cite{dereudre2024non} that can exhibit a certain hyperuniformity
property.

The aim of this paper is to investigate probabilistic properties of a
new type of point process with points given as critical points of a
smooth, stationary and isotropic latent Gaussian random field. Although log
Gaussian Cox processes and critical point processes are both defined
in terms of a latent Gaussian field, the point processes are very
different. Given the Gaussian random field, the
critical point process is deterministic, while the log Gaussian Cox process is
specified as an inhomogeneous Poisson point process with log intensity
function given by the Gaussian field.
The full distribution of a spatial point process (including critical
point processes and the other aforementioned models) is often intractable
and statistical methods for spatial point processes therefore often
focus on moment properties summarized by the intensity, the pair
correlation function and higher order intensity functions, or
cumulative summary statistics like Ripley's $K$-function~\cite{ripley1979tests}.

The study of critical points and level sets of Gaussian processes and
random fields has a long history beginning with the pioneering works
of Kac \cite{kac1943average} and Rice
\cite{rice1944mathematical}. Since then it has generated a huge
literature in probability theory
\cite{adler2007random,azais2009level} which exploits the celebrated
Kac-Rice formula providing (essentially) expected values for number of critical points or areas of level sets. We refer to \cite{berzin2022kac} for a contemporary review of the Kac-Rice formula.
Computing the expected number and correlation structure of critical
points as well as heights of peaks has also received a lot of attention in many domains such as astronomy \cite{bardeen1986statistics,larson2004hot,jow2019taller}, oceanography \cite[e.g.][]{lindgren1982wave}, neuroimaging \cite{worsley1996searching,taylor2007detecting}, etc.


The intensity parameter of a critical point process is expressed in
terms of the correlation function of the Gaussian field in
\cite{cheng2018expected} and \cite{azais2022mean}. Next,
\cite{kratz2006second} and \cite{estrade2016number} give conditions ensuring
the finiteness of second order moments while
\cite{armentano2023finiteness} and \cite{gass2024number} consider higher order
moments. Dimension-dependent repulsion and clustering properties of
critical point processes are studied by
\cite{beliaev2020no,azais2022mean}, and \cite{ladgham2023local} via the
asymptotic characterization of the pair correlation function $g(r)$ as
$r\to 0$. Finally,
\cite{estrade2016central,nicolaescu2017clt}, and \cite{azais2024multivariate}
prove central limit theorems, in any dimension, for the number of
critical points and related quantities. Most assumptions of the
aforementioned papers are related to regularity and non-degeneracy of
the Gaussian field.

The first main contribution of the current paper is a complete
characterization of the moment characteristics of any order for a
critical point process. In particular we provide an explicit
expression of the pair correlation function at any distance and
discuss how a Monte Carlo approximation can be efficiently
implemented. This enables us to discuss the overall shape of the pair
correlation function for: i) different dimensions; ii) different
examples of Gaussian random fields, namely the Gaussian Matérn random
field, the Bargmann-Fock model and the Gaussian random wave models
\cite[see e.g.\ ][]{berry1977regular,minasny2005matern}; and iii) different
types of critical points (local minima, maxima, extrema, saddlepoints,
etc). We also illustrate the asymptotic behavior of the pair correlation
at small distances as established by \cite{beliaev2020no,ladgham2023local}, and \cite{azais2022mean}.

Next, we consider the important question of simulating a critical
point process on a bounded domain. We believe that only
approximate simulation is possible for this problem. First
an approximate simulation of the Gaussian field is generated over
the entire simulation domain. In the second step, which we do not
discuss, critical points are found by
applying an appropriate standard root-finding procedure. Our
Theorem~\ref{thm:convcrit} shows that if a sequence of approximate simulations of the
Gaussian field converges in distribution to the original Gaussian field, then
the sequence of critical points of the approximate random field simulations converges to the
critical points of the original Gaussian field. In turn, Theorem~\ref{thm:Xab} applies this
general result to two procedures for generating approximate random
field simulations: i) smoothing of  an exact simulation on a lattice;
ii) averaging of continuous random fields generated from the spectral
approach~\cite[e.g.\ ][]{lantuejoul2013geostatistical}.

Our last contribution deals with asymptotic results for functionals of
critical point processes observed in a sequence of increasing
observation domains. More precisely we consider general linear and
bilinear statistics. The linear statistics include as a special case the standard estimate of
the intensity which is closely related to the statistics considered
 by~\cite{nicolaescu2017clt,estrade2016number}, and \cite{azais2024multivariate}. The
 previous literature does not cover more complex linear statistics and
 does not consider any type of bilinear statistic. A particular
 choice of bilinear statistic yields a non-parametric estimate of a
 slightly modified version of Ripley's $K$-function. As a step towards
 the asymptotic results, Proposition~\ref{prop:chaos} provides
 Hermite expansions of the linear and bilinear statistics. Next,
 Theorem~\ref{thm:variance} shows that the variance and covariances of the appropriately
 normalized statistics tend to explicit constants. Finally,
 Theorem~\ref{thm:clt} establishes a multivariate central limit
 theorem using a multiple Wiener-Itô integral representation of the
 normalized and centered statistics and a general result by
 \cite{peccati2011wiener}. These general results are applied to
 obtain the joint asymptotic Gaussian distribution of the intensity
 parameter estimate and the (slightly modified) $K$-function evaluated
 at a finite number of distances in space.

The rest of the paper is organized
as follows. Section~\ref{sec:background} provides a background on
Gaussian random fields and spatial point processes and provides the
definition of a critical point process. Section~\ref{sec:intensity} investigates
intensity functions and discusses the dependence  structure of a
critical point process. In Section~\ref{sec:simulation} we consider
simulation of a critical point process using spectral or grid based
methods to obtain approximate simulations of the Gaussian random
field. Asymptotic results for linear and bilinear functionals of
critical point processes under the increasing domain framework are
derived and discussed in detail in Section~\ref{sec:statistics}. This
paper is a foundation for future statistical applications of critical
point processes and interesting further theoretical and practical
research questions are discussed in Section~\ref{sec:conclusion}.


\section{Background on random fields and point processes}
\label{sec:background}

\subsection{Gaussian random fields}

In this paper, we use the order $[\{(\dots)\}]$ for multiple brackets.
Let $\bX=\{X(t),{t\in \R^d}\}$ denote a real-valued Gaussian random
field on $\R^d$ ($d \ge 1$). We assume throughout the paper that $\bX$
is centered (i.e.\ $\E X(t)=0$, $t \in \R^d$), stationary and
isotropic, and we denote by $c:\R^d\times \R^d \to \R$ its translation
and rotation invariant covariance function. With an abuse of notation we write $c(h)=c(t,t+h)$, $s, h \in \R^d$, and define functions $c_1$ and $c_2$ on $\R$ by the following equalities,
\begin{equation}
\label{eq:cov}
c_2(\|s-t\|^2) = c_1( \|s-t \|) = c(s-t) = c(s,t) = \E\left\{X(s)X(t)\right\}, \quad s,t \in \R^d
\end{equation}
where $\|\cdot\|$ is Euclidean norm. For instance, in case of the
Gaussian covariance function $c(s,t)=\exp(- \|s-t\|^2/\varphi^2)$ (for
some parameter $\varphi>0$), $c_1(r)=\exp(-r^2/\varphi^2)$ and
$c_2(r)=\exp(-r/\varphi^2)$. The notation $c_2$ may appear useless or
unusual but some  results, in particular from~\cite{azais2022mean},
are better expressed in terms of $c_2$ than $c_1$. We are only
interested in the locations of the critical points of $\bX$. We
therefore assume without loss of generality that $\bX$ has unit variance. Hence, $c,c_1$ and $c_2$ are actually correlation functions. 

We let $F$ denote the spectral (probability) measure  of $\bX$. Then
by Bochner's theorem \cite{rudin2017fourier}, $c$ and $c_1$ are
characteristic functions, namely $c(t) = \int_{\R^d} \exp(\mathrm{i}
\omega^\top t)  F(\dd \omega)= \E\left\{\exp\left(\mathrm{i} V^\top t\right)\right\}$ where $V$ is a random vector with distribution $F$. Moreover, $c_1(r)=c(re_1) = \E\left\{\exp(\mathrm ir V^\top e_1 )\right\}$ where, by isotropy, $e_1$ is any unit $d$-dimensional vector, e.g. $e_1 = (1,0,\dots,0)^\top$. Given the spectral representation~\cite{adler2007random}, we define (when it exists) the $k$th spectral moment by
\begin{equation}
\label{eq:lambda2p}
\lambda_{k} = \int_{\R^d} (\omega^\top e_1)^{k} \; F(\dd \omega).
\end{equation}
For any random field, $\lambda_k = 0$ whenever $k$ is odd. When $k=2p$ is even and $\lambda_{k}<\infty$, we can reexpress the spectral moment in terms of the covariance function \cite{azais2022mean},
$$\lambda_{2p} = \Var \left\{ \frac{\partial^p X(0)}{\partial t_{1}^{p}} \right\} = (-1)^p\frac{(2p)!}{p!}c_{2}^{(p)}(0),
$$
provided $\bX$ and $c_2$ are $p$ times differentiable at $0$, where
$f^{(p)}$ denotes the $p$ times derivative for a $p$ times differentiable function $f$.

Continuity and differentiability for random fields is generally stated in the mean square  or the almost sure sense and has been studied extensively, especially for Gaussian random fields. The existence of derivatives $\bX^\prime$, $\bX^{\prime\prime}$ or more general regularity properties of $\bX$ are intrinsically related to the regularity of the covariance function and the existence of spectral moments, see \cite{adler2007random,azais2009level,riedi2015strong,ladgham2023local,da2023sample}. The diagram in Appendix~\ref{app:regularity} illustrates the links between these concepts.

For $E\subseteq \R^d$, $E^\prime\subseteq \R^{d^\prime}$ (with $d,d^\prime\ge 1$), and $n >0$, we denote by $C^{n}(E,E^\prime)$ the set of functions $f:E\to E^\prime$ which are component-wise $n$ times continuously differentiable on $E$. When $E^\prime=\R$ we write for brevity $C^n(E)$ instead of $C^n(E,\R)$. 
Using the standard multi-index notation, we define for $\alpha\in \mathbb N^d$, the operator $\partial^\alpha= \partial_1^{\alpha_1}\dots\partial_d^{\alpha_d}$ that acts on functions (or random fields) which are at least $|\alpha|=\sum_i \alpha_i$ continuously differentiable on $\R^d$.
For non-integer exponents $\nu>0$, we denote by $C^{\nu}(\R^d)$ the set of functions $f\in C^{\lfloor \nu \rfloor}(\R^d)$ satisfying the local Hölder condition: for all compact $\Delta \subset \R^d$, there exists $\kappa>0$ such that
\begin{equation*}
    | \partial^\alpha f(x) - \partial^\alpha f(y) | \leq \kappa \| x-y \|^{\nu - \lfloor \nu \rfloor}
\end{equation*}
for all $\alpha \in \mathbb{N}^d$ such that $|\alpha| =\lfloor \nu
\rfloor$ and $x,y \in \Delta$. Similarly, we let $C^\nu_{L^2}(\R^d)$
and $C^\nu_{\rm a.s.}(\R^d)$ denote the spaces of $\R^d$-indexed and
$\R$-valued random fields satisfying the local H{\" o}lder condition
with derivatives in the mean square  ($L_2$) or the almost sure (a.s.) sense.

\subsection{Gradient and higher derivatives of random fields}

Considering a random field $\bX$ on $\R^d$,  we denote by $\bX^\prime = \{ X^\prime(t), t\in \R^d\}$ the vector valued gradient random field and by $\bX^{\prime\prime}=\{{X^{\prime\prime}(t), t\in\R^d}\}$ the $d\times d$ matrix valued Hessian random field (assuming existence of gradients and Hessians in an appropriate sense). We abuse notation and also denote when it is not ambiguous by  $X^{\prime\prime}(t)$ the $d(d+1)/2$ half-vectorized version of the Hessian matrix.  
Furthermore, we say that $\bX$ has a critical point at $t\in \R^d$ if
$X^\prime(t)=0$. This critical point is said to be with index $\ell$
for $\ell=0,\dots,d$, if $\iota\{X^{\prime\prime}(t)\}=\ell$ where
$\iota(M)$ denotes the number of negative eigenvalues of any squared
matrix $M$. Thus, a critical point corresponds to a local minimum
(resp.\ maximum) if $\ell=0$ (resp.\ $d$). For any $\L \subseteq \{0,1,\dots,d\}$, we let $\iota_\L\{X^{\prime\prime}(t)\}= \mathbf 1[\iota\{X^{\prime\prime}(t)\} \in \L]$. 

\subsection{Main examples}

\subsubsection{Gaussian Mat{\'e}rn random field}

As a first example we consider a random field $\bX$ with the Matérn correlation function \cite[e.g.][]{stein2012interpolation,minasny2005matern} parameterized by a scale parameter $\varphi>0$, a regularity parameter $\nu>0$, and given for any $r>0$ by
\begin{equation}
    \label{eq:matern}
    c_1(r) = c_2(r^2) = \frac{2^{1-\nu}}{\Gamma(\nu)} \, \left( \frac{r\sqrt{2\nu}}{ \varphi}\right)^\nu \, K_\nu\left( \frac{r\sqrt{2\nu}}{\varphi} \right)
\end{equation}
where $K_\nu$ is the modified Bessel function of the second kind. The
parameter $\nu$ is indeed related to the regularity of $\bX$: it can
be shown that $c_2 \in C^{\nu-\varepsilon}(\R^+)$ \cite[proof of
Proposition 10]{da2023sample} for any $\varepsilon>0$ whereby we
deduce using Appendix~\ref{app:regularity} that $\bX$ is at least
$\lceil\nu\rceil-1$ times continuously differentiable in the mean
square and almost sure sense. For instance,  the condition $\nu>2$ ensures that the Matérn random field is $C^2(\R^d)$
(or $C^{2+\varepsilon}(\R^d)$ for some $\varepsilon>0$) almost surely. In addition, for any integer
$p<\nu$, using recurrence properties of the Bessel function (see~\eqref{eq:derivatives:behaviour:matern} for details), the spectral moment $\lambda_{2p}$  is finite and given by
\begin{equation}
   \label{eq:l2pMatern} \lambda_{2p}= \frac{(2p)!}{2^p\, p!} \; \frac1{\varphi^{2p}} \; \prod_{q=1}^p\frac{\nu}{\nu-q}.
\end{equation}
When $\nu=\infty$, the Matérn correlation function formally
corresponds to the Gaussian correlation function, which leads to a
Gaussian random field also known as the Bargmann-Fock random field
\cite[e.g.\ ][]{beffara2017percolation}. In this case, $\bX$ is obviously infinitely continuously differentiable in the mean square and almost sure sense with spectral moments  given by 
\[
   \lambda_{2p} = \frac{(2p)!}{2^p\, p!} \, \frac1{\varphi^{2p}}.
\]

\subsubsection{Gaussian Random Wave Model}

Our second example is the {Gaussian Random Wave Model} (RWM),
see~\cite{berry2002statistics} in dimension $d=2$ or~e.g.\
\cite{canzani2019topology} and \cite{rivera:hal-03320870} in higher
dimensions. This is the unique stationary Gaussian field (up to a
multiplicative constant) satisfying the Helmotz equation. The unit
variance RWM with scale parameter $\varphi/\sqrt{d}$ has spectral
density function defined as the uniform distribution on the centered
$(d-1)$-dimensional sphere with radius $d\varphi^{-1}$ (the sphere is the discrete set $\{-\varphi^{-1},\varphi^{-1}\}$ when $d=1$).  For $d \ge 1$, its covariance/correlation function is 
\begin{equation}
\label{eq:RWM}
c_1(r) = c_2(r^2) = \Gamma\left(\frac d2\right)  \left( \frac{r\sqrt{d}}{2\varphi}\right)^{-(d/2-1)}  J_{d/2-1}\left( \frac{r \sqrt{d}}\varphi\right)
\end{equation}
where $J_\nu$ is the Bessel function of order $\nu$. When $d=1$, the process is also referred to as the sine-cosine process \cite{azais2009level} with covariance function $c_1(r)=\cos(r/\varphi)$.
The RWM is  almost surely infinitely continuously differentiable on
$\R^d$ since $c_2^{(p)}(r)$ is continuous at 0 for any $p\ge 0$
(see~\eqref{eq:derivatives:behaviour:RWM} for details). Its spectral
moments are given for any $p\ge 1$ by
\begin{equation}\label{eq:l2pRWM}
\lambda_{2p} = \frac{(2p)!}{2^p p!}  \; \frac{1}{\varphi^{2p}} \; \prod_{q=0}^{p-1} \frac d{d+2q}
\end{equation}
which is in particular equal to $1/\varphi^{2p}$ when $d=1$. The spectral moments of the RWM are very close to the ones of a Bargmann-Fock random field when $d$ is large. This justifies our parameterization of the RWM with a scale parameter $\varphi/\sqrt{d}$ instead of~$\varphi$.

\subsection{Point processes} \label{sec:pp}

We consider spatial point processes in $\R^d$ and view a point process as a random locally finite subset $\bY$ of a Borel set $S \subseteq\R^d$, $d\geq 1$. The reader interested in measure theoretical details is referred to e.g.\ 
\cite{moller2003statistical} or \cite{daley2007introductionII}. This setting implies that the point process is simple, i.e. two points cannot occur at the same location. 

For $B\subseteq \R^d$, $\bY\cap B$ is the restriction of $\bY$ to $B$ and we let $|B|$ denote the volume of any bounded $B\subset \R^d$. Local finiteness of $\bY$ means that the number of points $N(B):=\#(\bY \cap B)$ of $\bY \cap B$ is finite a.s., whenever $B$ is bounded. A point process is said to be stationary (resp.\ isotropic) if its distribution is invariant under translations (resp.\ rotations). 
The distribution of $\bY$ can be characterized by the
finite-dimensional distributions of counting variables, by void
probabilities, or by the moment generating functional
\cite{moller2003statistical}. However, in applications of point
processes it is common to only analyze moment properties expressed in
terms of first and higher order intensity functions. We follow
\cite{moller2003statistical} and \cite{daley2007introductionII} and assume that the $k$th ($k\ge 1$) factorial moment measure is absolutely continuous with respect to the Lebesgue measure on $(\mathbb{R}^d)^k$. Its Radon-Nikodym derivative is the $k$th order intensity function $\rho^{(k)}$ characterized by the Campbell theorem: for any nonnegative function $\phi :(\R^d)^k \to \R^+$,
\begin{equation}
\label{eq:campbell} 
\E \sum_{t_1,\dots,t_k  \in \bY}^{\neq} \phi(t_1,\dots,t_k) = \int_{\R^d} \dots \int_{\R^d} \phi(t_1,\dots,t_k)\rho^{(k)}(t_1,\dots,t_k)\dd t_1 \dots \dd t_k
\end{equation}
where the sign $\neq$ means that the sum is defined for any pairwise distinct points $t_1,\dots,t_k$. Roughly speaking, $\rho^{(1)}(t)\dd t$ is the mean number of points in the vicinity of $t$ and more generally, $\rho^{(k)}(t_1,\dots,t_k)\dd t_1\dots \dd t_k$ can be interpreted as the probability of observing simultaneously a point in each of $k$ infinitesimal neighbourhoods around $t_1,\dots, t_k$. When a point process is stationary (resp.\ stationary and isotropic),  $\rho^{(1)}(s)=\rho$ is constant and $\rho^{(2)}(s,t)$ depends only on $t-s$ (resp.\ on $\|t-s\|$). The pair correlation function is defined by $g(s,t)= \rho^{(2)}(s,t)/\{\rho^{(1)}(s)\rho^{(1)}(t)\}$ and in the stationary and isotropic case we abuse notation and write
\[
    g(\|s-t\|) =  \frac{\rho^{(2)}(\|s-t\|)}{\rho^2}.
\]
The pair correlation function measures deviations from the Poisson point
process which is the reference model without any interactions between
points. In particular, the $k$th order intensity of a Poisson point
process exists and equals $\prod_{i=1}^k \rho(t_i)$ which leads to $g
\equiv 1$. Hence, when $g(r)<1$ (resp.\ $g(r) >1$) two points at
distance $r$ are less likely (resp.\ more likely) to appear jointly
than under a Poisson point process. A point process is said to be
repulsive or clustered (resp.\ purely repulsive or purely clustered)
if $g(0)<1$ or $g(0)>1$ (resp.\ $g(r)<1$ or $g(r)>1$ for any
$r$). Finally, we introduce a popular function often used in point
pattern analysis:  Ripley's $K$-function. This is defined in the
stationary case as $1/\rho$ times the expected number of extra events
in $B(0,r)$ (the Euclidean ball centered at 0 with radius $r$) given
that 0 is a point of $\bY$. Ripley's $K$-function is a cumulative version of the pair correlation function since \cite{moller2003statistical}
\begin{equation*}
K(r)  = \frac1\rho \E \left[ N\{B(0,r)\} \mid 0 \in \bY\right] = \int_{B(0,r)}g(\|t\|)\dd t. 
\end{equation*}
In particular $K(r)=v_d r^d$ for Poisson point processes where
$v_d=|B(0,1)|$. Thus, $K(r)<v_d r^d$ for instance typically means that
the point process exhibits repulsion at distances smaller than
$r$. Other intensity functions such as Palm or
Papangelou conditional intensities as well as more complex summary
statistics based on distances are discussed in \cite{moller2003statistical,illian2008statistical,coeurjolly2019understanding}. 

\subsection{Critical point process}

The critical point processes studied in this paper are defined as follows.

\begin{definition}[Critical point process]\label{def:critical}
Let $\bX$ be a centered, stationary and isotropic Gaussian random field on $\R^d$ ($d\ge 1$), with unit variance and at least twice continuously differentiable. For any $\ell=0,\dots,d$, or set of indices $\L \subseteq \{0,\dots,d\}$, we define
\begin{equation}
\label{eq:YL}
\bY_\ell = \left\{t \in \R^d: X^\prime(t)=0 , \iota\{X^{\prime\prime}(t)\}=\ell \right\} 
\quad \text{ and } \quad
    \bY_\L = \cup_{\ell \in \L} \bY_\ell
\end{equation}
to be the point processes consisting of critical points of $\bX$ with index $\ell$ or critical points of $\bX$ with index in $\L$.
By definition, the point processes $\bY_0=\bY_{\{0\}}$ and $\bY_d=\bY_{\{d\}}$ correspond to the random sets of local minima and local maxima of $\bX$ while the set of all critical points $\bY_{\{0,\ldots,d\}}$ will be denoted $\bY_{0:d}$.
\end{definition}

The stationarity and isotropy of $\bX$ implies the stationarity and
isotropy of the point process $\bY_\L$. The assumption that $\bX$ is centered with unit variance is not restrictive since the set of
critical points $\bY_\L$ is identical for $\bX$ and for $\mu +
\sigma \bX$ for any $\mu \in \R, \sigma>0$. Moreover, for a strictly
monotone function $f$, $\bX$ and $f(\bX)=\{f \{ X(t)\}, t\in \R^d\}$
have the same critical points of any index. Hence, this paper is not entirely restricted to Gaussian processes.

For a random field as specified in Definition~\ref{def:critical} we
denote by $N_\L(B)$ the number of critical points with index in $\L$
that fall within  a bounded domain $B \subset \R^d$. When they exist
(see the next section), we denote by $\rho_\L$  the intensity
parameter of $\bY_\L$ and by $\rho^{(k)}_\L$ the $k$th order intensity
function of $\bY_\L$. Moreover, we let $g_\L$ and $K_\L$ denote
respectively the pair correlation function and Ripley's $K$-function for $\bY_\L$. Figure~\ref{fig:example} shows simulations of critical point processes for $d=2$ and $\ell \in \L=\{0,2\}$ (Section~\ref{sec:simulation} discusses numerical and theoretical aspects related to the simulation of critical point processes).

For $\L, \L^\prime \subseteq \R^d$ we define cross second-order intensities by
\begin{equation}
\label{eq:crosscampbell} 
\E \sum_{s \in \bY_{\L},t \in \bY_{\L'}}^{\neq} \phi(s,t) = \int_{\R^d}  \int_{\R^d} \phi(s,t)\rho^{(2)}_{\L,\L'}(\|s-t\|)\dd s \dd t,
\end{equation}
for nonnegative $\phi:(\R^d)^2 \rightarrow \R^+$, and the cross pair correlation function by
\begin{equation}
    \label{eq:def:gLLprime}
    g_{\L,\L^\prime}(\|t-s\|) =   \frac{\rho^{(2)}_{\L,\L^\prime}(\|t-s\|)}{\rho_\L \rho_{\L^\prime}}.
\end{equation}
These definitions extend the previous ones. For instance $g_{\L,\L^\prime}=g_\L$ when $\L=\L^\prime$.

\begin{figure}[htbp]
\begin{center}
\includegraphics[width=.45\textwidth]{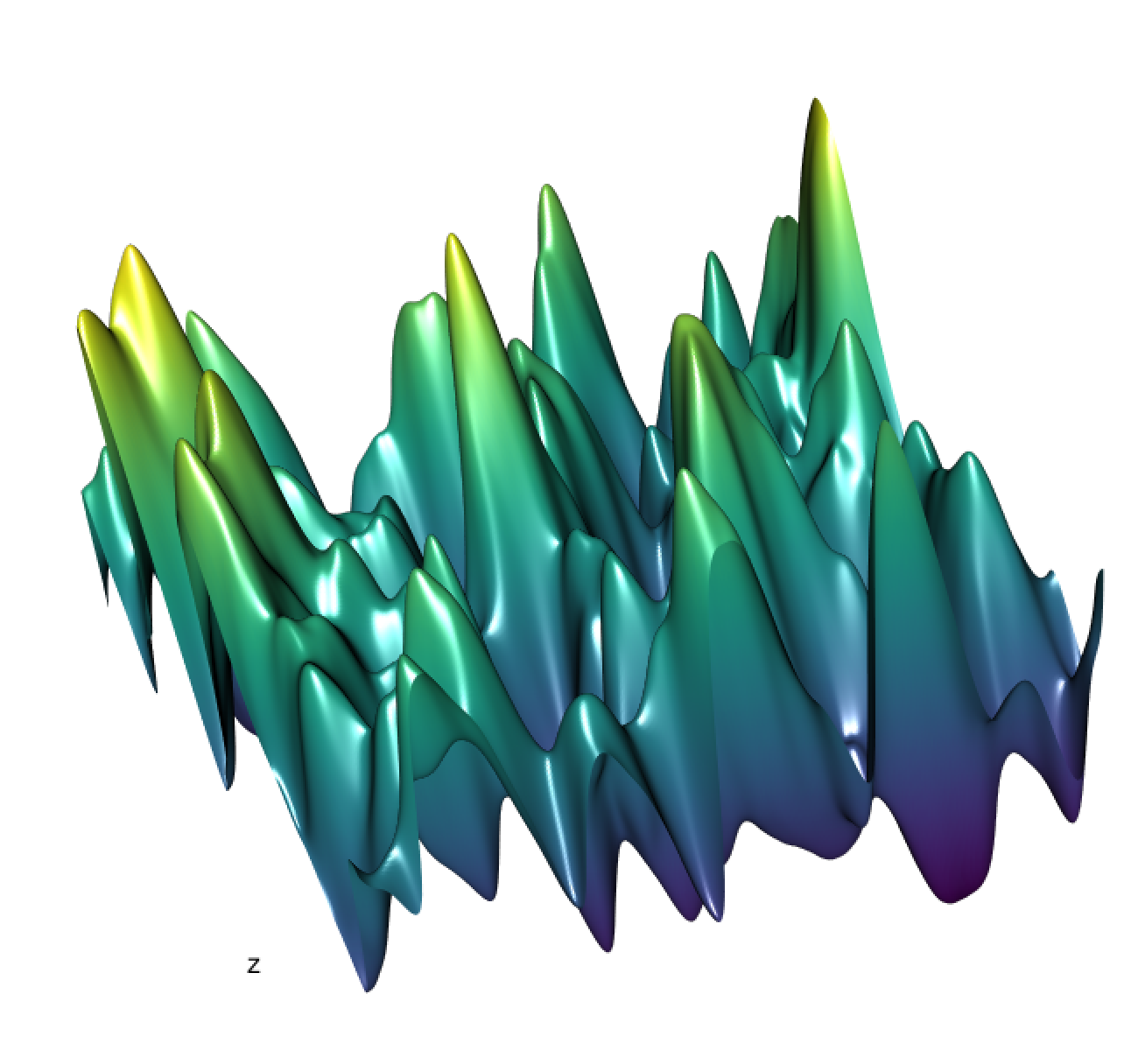} \includegraphics[width=.5\textwidth]{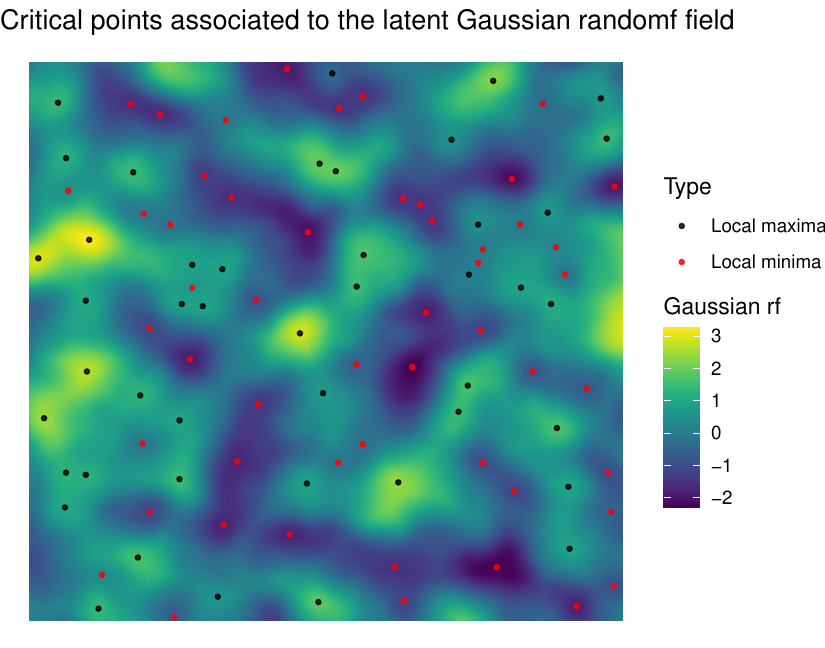}
\caption{(Left) Simulation of a Gaussian random field $\bX$ in $[0,1]^d$ with $d=2$ and with Gaussian correlation function (which formally corresponds to a Matérn random field with $\nu=\infty$); (Right) The set of black (resp.\ red points) is a realization of $\bY_\L$ with $\L=\{d\}$ (resp.\ with $\L=\{0\}$), that is the set of local maxima (resp.\ local minima) of the Gaussian random field. Points are superimposed on the image of the latent random field simulation corresponding to the one on the left panel.}
\label{fig:example}
\end{center}     
\end{figure}


\section{Intensity and pair correlation function for critical point processes} \label{sec:intensity}

\subsection{Assumptions and discussion}

As detailed in the introduction, the problem of evaluating the mean number of critical points in a domain and the dependence between two critical points has been extensively studied \cite{azais2022mean,ladgham2023local,beliaev2020no}. The technique is mainly related to applications of the Kac-Rice formula \cite{azais2009level,berzin2022kac}.
Before going further, we present a set of general conditions required in this section.
\begin{enumerate}[($\mathcal C$.1)]
\item\!\!\!$[\nu]$ $\bX$ is a centered, stationary, and isotropic Gaussian random field with unit variance and $\bX\in  C_{\rm a.s.}^{\nu}(\R^d)$. \label{C:general}
\item\!\!\!$[p]$ The Gaussian random vector $\{\partial^\alpha X(t)
  \}_{1\le |\alpha|\le p}$ is non-degenerate (has full rank covariance
  matrix) for any $t\in \R^d$. 
\label{C:nondegeneracy}
\item \!\!\!$[r]$ The Gaussian random vector $V(r)= \{X^\prime(0)^\top,X^\prime(re_1)^\top\}^\top$ is non-degenerate, where $e_1$ is any unit $d$-dimensional vector. \label{C:nondegeneracy:gradient}
\end{enumerate}

According to Appendix~\ref{app:regularity}, the assumption
that $c \in C^{2(p+\varepsilon^+)}(\R^d)$ or $c_1 \in
C^{2(p+\varepsilon^+)}(\R^+)$ or $c_2 \in  C^{p+\varepsilon^+}(\R^+)$
for some $\varepsilon^+ > \varepsilon>0$ together with the first part
of \ref{C:general}$[p+\varepsilon]$ implies the last part of
\ref{C:general}$[p+\varepsilon]$. Moreover, \ref{C:general}$[p]$
implies  $\lambda_{2q}<\infty$, $q=1,\dots,p$.

Obviously, condition \ref{C:nondegeneracy}$[p]$ only makes sense when
at least the last part of \ref{C:general}$[p]$ holds. When
$p=1$, this is, see\ Appendix~\ref{app:regularity}, equivalent to require that $\lambda_2>0$, which in turn implies that $\lambda_{2q} > 0$ for any $q$ when it exists. Condition~\ref{C:nondegeneracy}$[2]$ implies that the
probability to get $X^\prime(t)=0$ and $\det\{X^{\prime\prime}(t)\}=0$ simultaneously
is null and in turn that the sample paths of $\bX$ are almost surely
Morse, i.e.\ do not have degenerate critical points (see
\cite[Proposition 18]{azais2022mean} for instance). Condition \ref{C:nondegeneracy}$[p]$ follows from \ref{C:general}$[p]$ as soon as the spectral measure admits a density (which is the case for  Matérn random fields for instance). More precisely, one can get the following generalization of \cite[Exercise 3.5]{azais2009level}: let $p>0$, assume condition \ref{C:general}$[p]$ and that the support of the spectral measure has non empty interior, then $\bX$ satisfies \ref{C:nondegeneracy}[$p$]. This necessary condition is not fulfilled by RWM for which the spectral measure is defined as the uniform law on a $(d-1)$-dimensional sphere and hence has support with empty interior. Actually, by definition of the RWM which satisfies a differential equation involving $X(t), X^\prime(t)$ and $X^{\prime\prime}(t)$, \ref{C:nondegeneracy}$[p]$ does not hold for the RWM for any $p> 2$. 
Finally, condition~\ref{C:nondegeneracy:gradient}$[r]$ is used in Theorem~\ref{thm:pcf} to ensure that $g_{\L,\L^\prime}(r)$ is well-defined.

\subsection{Intensity parameter}

The following result appears in \cite{cheng2018expected}. We recall that a $q \times q$ matrix $M$ is said to have the GOE (Gaussian Orthogonal Ensemble) distribution if it is symmetric and all its elements are independent centered Gaussian variables satisfying $E(M_{ii}^2) = 1$ and $E(M_{ij}^2) = 1/2$. for any $i,j=1,\dots,q$, $i\neq j$.

\begin{theorem}{\cite[Proposition 3.9]{cheng2018expected}} \label{thm:intensity}
Consider conditions~\ref{C:general}$[2]$ and \ref{C:nondegeneracy}$[1]$. Then for any $\ell= 0,\dots,d$,
\begin{align}
    \rho_\ell &= \frac1{(2\pi \lambda_2)^{d/2}} \; \E 
    \left[|\det \; X^{\prime\prime}(0)| \; \iota_\L\{X^{\prime\prime}(0)\} \right]
    \nonumber \\
  &= \frac{1}{(d+1)\pi^{(d+1)/2}} \frac{\kappa_d}{\kappa_{d+1}} 
    \left( \frac{\lambda_4}{3\lambda_2}\right)^{d/2} 
   \E \left\{    \exp\left(-\frac{ \mu_{\ell+1}^2}{2} \right)\right\} \label{eq:rhoL}
\end{align}
where $\kappa_m=(2\pi)^{-m/2}\Gamma(3/2)^m \left\{\prod_{q=1}^m \Gamma(1+q/2)\right\}^{-1}$ and where $0<\mu_1<\dots<\mu_{d+1}$ are the ordered random eigenvalues of a $(d+1)\times(d+1)$ GOE matrix. Obviously, for any $\L\subseteq \{0,1,\dots,d\}$, we have $\rho_\L=\sum_{\ell \in \L} \rho_\ell$.
\end{theorem}

\cite{azais2022mean} go further and obtain the probability density
function of $\mu_k$. This fine result allows them to obtain an
explicit expression of~\eqref{eq:rhoL} in terms of the ratio
$\lambda_4/3\lambda_2$ for $d=1,\dots,4$. These expressions, very
useful to tune a simulation study, are summarized in
Table~\ref{tab:rhoL}. It is worth pointing out, for example, the
proportion of local maxima among all critical points: 50\% when $d=1$,
then $25\%, 12.3\%, 6\%$ when $d=2,3,4$.

As a consequence of~\eqref{eq:rhoL} and Table~\ref{tab:rhoL}, the intensity parameter $\rho_\L$ is parameterized through $\lambda_4/3\lambda_2$. For example, for the Matérn correlation model and RWM,  we have using~\eqref{eq:l2pMatern} and~\eqref{eq:l2pRWM},
\[
    \frac{\lambda_4}{3\lambda_2} = \frac{1}{\varphi^2} \frac{\nu}{\nu-2} \quad \text{(Matérn)}
    \qquad  \text{ and }\qquad 
    \frac{\lambda_4}{3\lambda_2} = \frac{1}{\varphi^2} \; \frac d{d+2} \quad \text{(RWM)}.
\]
This ratio is a decreasing function of the scale parameter $\varphi$ and of the regularity parameter $\nu$ for Matérn random fields. It is clear that: (i) the larger $\varphi$, the more two close observations of the random field are positively correlated, (ii) the larger $\nu$, the more regular the latent Gaussian random field. Both comments lead to the conclusion that the latent Gaussian random field  exhibits less critical points. A similar comment regarding $\varphi$ holds for the RWM.

\begin{table}[h]
\begin{center}
\begin{tabular}{p{2cm}p{11cm}}
\hline
Dimension & Intensity parameter $\rho_\L$ \\ &\\
\hline \hline
$d=1$ & $\rho_0=\rho_1=\frac12\rho_{0:1}$ and  $\rho_{0:1}= \frac{\sqrt{3}}\pi \left( \frac{\lambda_4}{3\lambda_2}\right)^{1/2}$   \\ &\\
$d=2$ & $\rho_0=\rho_2=\frac{1}4 \rho_{0:2}$, $\rho_1=\frac12 \rho_{0:2}$ and $\rho_{0:2}=\frac{2}{\pi\sqrt{3}} \left(\frac{\lambda_4}{3\lambda_2}\right) $ \\ &\\
$d=3$ & $\rho_0=\rho_3= \frac{29-6\sqrt 6}{116} \rho_{0:3} \approx 0.123 \rho_{0:3}$, \newline$\rho_1=\rho_2= \frac{29+6\sqrt 6}{116} \rho_{0:3} \approx 0.377 \rho_{0:3}$ and $\rho_{0:3} = \frac{29}{6 \pi^2 \sqrt{3}} \left( \frac{\lambda_4}{3\lambda_2}\right)^{3/2}$\\ &\\
$d=4$ & $\rho_0=\rho_4 \approx 0.06 \rho_{0:4}$, $\rho_1=\rho_3 
= \frac14 \rho_{0:4}$, 
$\rho_2
\approx 0.38 \rho_{0:4}$ and \newline 
$\rho_{0:4}= \frac{25}{6\pi^2\sqrt{3}} \left( \frac{\lambda_4}{3\lambda_2}\right)^{2}$\\ 
\hline
\end{tabular}
\caption{Intensity parameters $\rho_\L$ for $\L=\{\ell\}$ ($\ell=0,\dots,d$) and $\L=\{0,\dots,d\}$ for $d=1,\dots,4$.} 
\label{tab:rhoL} 
\end{center}     
\end{table}

\subsection{Pair correlation function}

The next result provides second-order properties for critical point processes.

\begin{theorem} \label{thm:pcf} Consider
  conditions~\ref{C:general}$[2]$ and \ref{C:nondegeneracy}$[1]$. Then, we have the following results.
\begin{enumerate}[(i)]
    \item For any $r>0$, condition~~\ref{C:nondegeneracy:gradient}$[r]$ is satisfied if and only if
    \begin{equation}\label{eq:conditionVr}
        c_2^\prime(0)^2- c_2^\prime(r^2)^2 > 0
        \quad \text{ and } \quad 
        c_1^{\prime\prime}(0)^2 - c_1^{\prime\prime}(r)^2 > 0.
    \end{equation}
In that case, the  density of $V(r)= \{X^\prime(0)^\top,X^\prime(re_1)^\top\}^\top$ at  $0\in \mathbb{R}^{2d}$ is given by
    \begin{equation}\label{eq:density:Vr}
        f_{V(r)}(0,0)=(2\pi)^{-d/2} 2^{1-d}\left\{c_2^\prime(0)^2 -c_2^\prime(r^2)^2\right\}^{\frac{1-d}{2}} \;\left\{c_1^{\prime\prime}(0)^2 - c_1^{\prime\prime}(r)^2\right\}^{-\frac{1}{2}}.
    \end{equation}
    
    \item For any $r>0$ such that~\ref{C:nondegeneracy:gradient}$[r]$ is satisfied and for all $\L,\L^\prime \subseteq \{0,1,\dots,d\}$,
    \begin{align}
    g_{\L,\L^\prime} (r) =& \frac{1}{\rho_\L \, \rho_{\L^\prime}}\; f_{V(r)}(0,0) \E 
    \Big(
    |\det\{X^{\prime\prime}(0)\}| \times |\det\{X^{\prime\prime}(re_1)\}| \; \times \nonumber  \\
    & \qquad 
    \iota_\L\{X^{\prime\prime}(0)\} \;  \iota_{\L^\prime}\{X^{\prime\prime}(r e_1)\}  
    \, \mid \,  X^\prime(0)=X^\prime(re_1)=0 \Big). \label{eq:gLLprime}        
    \end{align}    
    Furthermore, $g_{\L,\L^\prime}$ is continuous at $r$.
\end{enumerate}

\end{theorem}

The proof of Theorem \ref{thm:pcf} is given in Appendix
\ref{sec:proofPcf}. The conditional expectation in~\eqref{eq:gLLprime}
seems hard to deal with. However, \cite[Lemma 8]{azais2022mean} shows
that the conditional distribution of $\{X^{\prime\prime}(0)^\top
,X^{\prime\prime}(re_1)^\top \}^\top $ given $X^\prime(0)=X^\prime(re_1)=0$ is zero-mean Gaussian with
$d(d+1)\times d(d+1)$  covariance matrix whose elements only depend on
$c_2^\prime(u),c_2^{\prime\prime}(u)$ for $u=0,r^2$. Therefore, an
approximation of~\eqref{eq:gLLprime} can be obtained using Monte Carlo
simulations of this conditional distribution. The simulations can be
generated efficiently so that high Monte Carlo accuracy can be achieved
at a moderate computational cost.

Figure~\ref{fig:plotc1c2} in Appendix~\ref{app:localpcf} illustrates
that~\eqref{eq:conditionVr} is satisfied for Matérn random fields and
RWM when $d>1$. For the sine-cosine process (i.e. for the RWM with
$d=1$), it is easily seen that~\eqref{eq:conditionVr}, or equivalently
\ref{C:nondegeneracy:gradient}$[r]$, is not satisfied for $r/\varphi
\in \pi \mathbb Z$. One might consider other
values of $r$, but since the sine-cosine process can be written
\cite[see Exercise 4.1]{azais2009level} as
$X(t)=Z\cos(t/\varphi + \theta)$  with $Z^2\sim \chi^2_2$ and
$\theta\sim \mathcal U(0,2\pi)$, the critical points are located in
the randomly translated grid $\{t: t/\varphi + \theta \in \pi \mathbb
Z\}$ whereby the second order factorial
  moment measure is not absolute continuous. Hence, the second
  order intensity and the pair correlation function do not exist which leads us to disregard the RWM when $d=1$ in the rest of this paper.

When $r$ is close to zero (and only in this setting)
Theorem~\ref{thm:pcf} essentially corresponds to \cite[Proposition
1.5]{azais2022mean} (see also \cite{ladgham2023local}) when $\L=\{\ell\}$ and $\L=\{0,\dots,d\}$. 
Our contribution is to extend the result for any subsets
$\L,\L^\prime$  and for any $r>0$. The latter extension is of great importance in order first to depict the shapes of different pair correlation functions and Ripley's $K$-functions and second to estimate the parameters of a parametric critical point process using minimum contrast estimation \cite[e.g.][]{waagepetersen2009two}. 

The papers \cite{azais2022mean,ladgham2023local}, and \cite{beliaev2020no} focus
on the behaviour of the pair correlation function and cross pair
correlation functions as $r\to 0$, essentially to characterize a
clustering or repulsion effect at small scales. These three papers are
summarized in Table~\ref{tab:expansion_pcf}. The table shows that the
pair correlation function of all critical points tends to 0 when
$d<1$, to a constant when $d=2$ and to $\infty$ when $d>2$ (which
reveals an extreme clustering effect at small scales). The pair
correlation function of maxima or minima further tends to 0 when $d\le
4$, and at small scales  (i.e.\ when $r\to0$), minima and maxima tend
to repel each other more and more when the dimension is increased. The explicit
expression of $g_\L$ obtained in Theorem~\ref{thm:pcf}(ii) allows us,
using Monte Carlo simulations, to visualize  some of the results in Table~\ref{tab:expansion_pcf} (see Appendix~\ref{app:localpcf}).

\begin{table}
\begin{center}
\begin{tabular}{p{.25\textwidth}p{.69\textwidth}}
\hline
&\\
Critical point indices & Asymptotic result \\
&\\\hline\hline
&\\
$\L=\L^\prime=\{0,\dots,d\}$  &$\bullet$ $g_{0:d}(r) \sim c r^{2-d}$ with $c=\frac{\gamma_{d-1}}{2^3 3^{(d-1)/2}\pi^{d}} (\frac{\lambda_4}{\lambda_2})^{d/2} \frac{\lambda_2\lambda_6-\lambda_4^{2}}{\lambda_2\lambda_4} \frac{1}{\rho_{0:d}^2}$\\
&\quad AD[Theorem 13, under~{\ref{C:general}[3]}]\\
&$\bullet$ $g_{0:d} \to c$ for $d=2$\\
&\quad BCW[Theorem 1.2, under~{\ref{C:general}$[3+\varepsilon]$}]\\ 
&$\bullet$ $g_{0:d} \to c$ for $d=2$ with $c\ge \frac1{8\sqrt 3}\frac1{\rho_{0:2}^2}$ reached for the RWM\\
&\quad LLR[Theorem 4.1, under~{\ref{C:general}$[4+\varepsilon]$}] \\ 
&\\
\hline
&\\
$\L=\L^\prime=\{0,d\}$
& $\bullet$ $g_{\{0,d\}}(r) = O\{r^{(2d-1)\wedge (5-d)-\varepsilon} \}$\\
&\quad AD[Theorems 16-17, under~{\ref{C:general}[4]}]$^\ddag$\\
&$\bullet$ $g_{ \{0,2\} }(r) = O\{ r^3\log(1/r) \}$ ($d=2$)\\
&\quad BCW[Theorem 1.5, under~{\ref{C:general}$[3+\varepsilon]$}]\\ 
& $\bullet$ $g_{ \{0,2\} }(r) = O( r^3 )$ ($d=2$)\\
&\quad LLR[Theorem 4.1, under~{\ref{C:general}$[4+\varepsilon]$}] \\
&\\
\hline
&\\
$\L=\L^\prime=\{\ell\}$  &$\bullet$ $g_{\ell}(r) = O( r^{5-d-\varepsilon})$, for $\ell=0,d$ \\
&\quad AD[Theorem 16, under~{\ref{C:general}[4]}]\\ 
&$\bullet$ $g_{\ell}(r) = O\{ r^3 \log(1/r) \}$ for $\ell=0,\dots,d$ and $d=2$\\
&\quad BCW[Theorem 1.2, under~{\ref{C:general}$[3+\varepsilon]$}] \\ 
&{\quad LLR[Theorem 4.1, under~{\ref{C:general}$[4+\varepsilon]$}]}$^\ddag$ \\
&\\
\hline
&\\
$\L=\{\ell\},\L^\prime=\{\ell+1\}$, & $\bullet$ $g_{\ell,\ell+1}(r) \sim c r^{2-d}$ with $c=\frac{\gamma_{d-1}^\ell}{2^4 3^{(d-1)/2}\pi^{d}} (\frac{\lambda_4}{\lambda_2})^{d/2} \frac{\lambda_2\lambda_6-\lambda_4^{2}}{\lambda_2\lambda_4} \frac1{\rho_{\ell}\rho_{\ell+1}}$\\ 
$\ell=0,\dots,d-1$
&\quad AD[Theorem 14, under~{\ref{C:general}[3]}]\\
&\quad BCW[Theorem 1.2, $d=2$ and under~{\ref{C:general}$[3+\varepsilon]$}] \\ 
&\\
\hline
&\\
$\L=\{0\}, \L^\prime=\{d\}$& $\bullet$ $g_{0,d}(r) = O( r^{2d-1-\varepsilon})$ \\
&\quad AD[Theorem 17, under~{\ref{C:general}[3]}]\\ 
&$\bullet$ $g_{0, 2}(r) = O( r^3 )$ ($d=2$),  $=O( r^7 )$ for the RWM\\
&\quad BCW[Theorem 1.5, $d=2$ and under~{\ref{C:general}$[3+\varepsilon]$}] \\
&$\bullet$ $g_{0, 2}(r) = O( r^3 )$ ($d=2$)\\
&{\quad LLR[Theorem 4.1, $d=2$ and under~{\ref{C:general}$[4+\varepsilon]$}]}$^\ddag$\\ 
&\\\hline
\end{tabular}
\caption{Asymptotic results for (cross) pair correlation functions as
  $r\to 0$ obtained by~AD=\cite{azais2022mean},
  BCW=\cite{beliaev2020no} and LLR=\cite{ladgham2023local}. In
  particular $\gamma_{d-1}= \E\{ \det(G_{d-1}-A I_{d-1})^2\}$ with
  $G_{d-1}$ a $(d-1)\times(d-1)$-GOE matrix and $A\sim \mathcal{N}(0, 1/3)$;
  $\gamma_{d-1}^\ell= \E[ \det(G_{d-1}-A I_{d-1})^2
  \mathbf{1}\{\iota(G_{d-1}-A I_{d-1})=\ell\} ] $. The O() behaviors
  are valid for any $\varepsilon>0$ and as pointed out by
  \cite[p. 434]{azais2022mean}, $\varepsilon=0$ when $d=1$. The
  notation ${}^\ddag$ for a given reference means that the result is
  actually deduced from other results of the reference. }
\label{tab:expansion_pcf}
\end{center}
\end{table}

\subsection{Attraction and repulsion for a critical point process}

There is a vast literature on the notion of regularity, rigidity,
repulsion or hyperuniformity
\cite[e.g.][]{coste2021order,hawat2023estimating,moller2003statistical}. Based
on this literature, we believe that exploring the dependence only at
small scales is not enough. For example, \cite{moller2003statistical}
qualify a point process as purely repulsive if the pair correlation
function is smaller than 1 at all scales, a property which is for
instance achieved for determinantal point processes
\cite{lavancier2015determinantal}. Our Theorem~\ref{thm:pcf}(ii) is
therefore helpful to get an overall picture. This is exploited in
Figure~\ref{fig:p.pcf} which shows Monte Carlo estimates of pair
correlation functions $g_\L$ (using $10^7$ replications). We consider
$d=1,2,3$, Matérn and RWM random fields with $\L=\{0,\dots,d\}$ (all
critical points) and $\L=\{d\}$ (local maxima). For Matérn random
fields, we examine $\nu=2.5, 3.5, 4.5, \infty$ which implies that
$\bX$ fulfills (at least) \ref{C:general}[2], \ref{C:general}[3],
\ref{C:general}[4], and \ref{C:general}$[\infty]$, respectively. Using
Table~\ref{tab:rhoL}, the scale parameter $\varphi$ is, for each
situation, set so that $\rho_\L=100$. The following general comments
also apply for other values of $\rho_\L$. 

First, we relate the behavior of $g_\L(r)$ when $r\to
0$ observed in Figure~\ref{fig:p.pcf} to Table~\ref{tab:expansion_pcf}. 
For example, when $d=1$, $\L=\{0,1\},\{1\}$ and for the
Mat{\'e}rn random field, $g_\L$ tends to 0 at 0 when
\ref{C:general}[3] is true, that is when 
$\nu>3$. The same comment applies for $d=2$ and $\L=\{2\}$. For $d=3$ and $\L=\{3\}$,
the convergence to zero requires $\nu>4$. For the RWM model, the
convergence to zero happens except for $\L=\{0,\ldots,3\}$.
For Matérn random
fields, increasing smoothness (increasing $\nu$) implies increasing
repulsion and overall, the RWM seems to be the most repulsive at
small scales (as proved by~\cite{ladgham2023local}). Conversely,
Mat{\'e}rn models may produce clustering when $d=2$ or $d=3$ and 
with stronger clustering for smaller smoothness.
Figure~\ref{fig:p.pcf} also shows that the pcf
of RWM oscillates a lot and tends to 1 quite slowly for increasing $r$. This can be
indicative of a near lattice behaviour of the RWM critical point
process. In contrast, for any $d,\L$ and $\nu$, $g_\L$ tends quickly to 1 as $r\to\infty$
for Matérn random fields (see condition~\ref{C:integrability} and its
discussion in Section~\ref{sec:statistics}). In
Appendix~\ref{app:localpcf}, we continue this discussion and compare
pcfs using also a global measure of repulsiveness.

\begin{figure}[htbp]
\centering
\includegraphics[width=\textwidth]{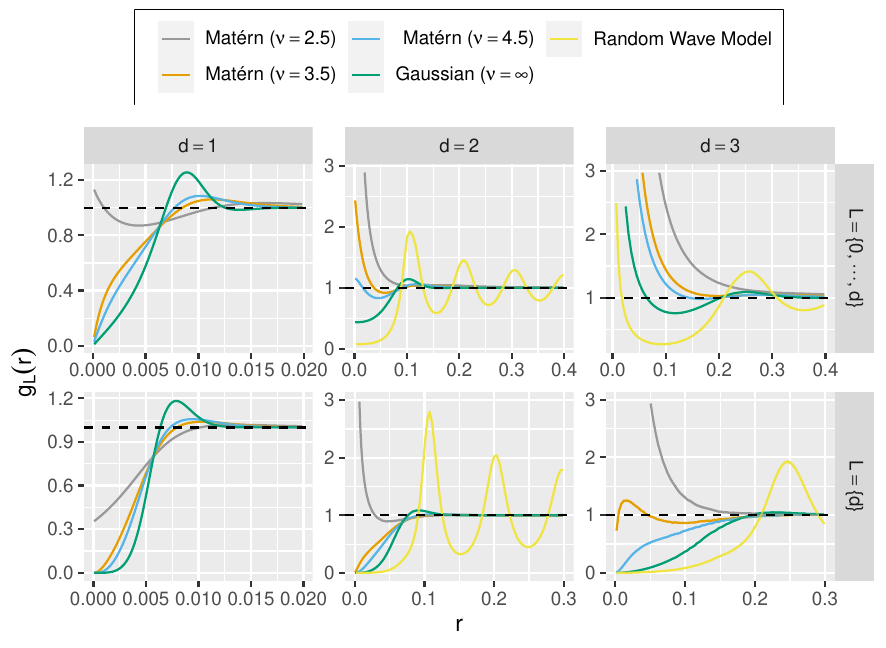}
\caption{For $d=1,2,3$ Monte Carlo estimation of pair correlation
  function for critical points with indices $\L=\{0,\dots,d\}$ and
  $\L=\{d\}$, $d=1,2,3$, based on Matérn random fields with
  regularity parameter $\nu=2.5,3.5,4.5,\infty$ and on the Random Wave
  Model (for $d>1$). The scale parameter $\varphi$ is set such that
  $\rho_\L=100$. Monte Carlo estimates, evaluated pointwise on a
  regular grid of 200 values, are obtained with $10^7$ simulations
  of the distribution of $\{X^{\prime\prime}(0)^\top
  ,X^{\prime\prime}(r)^\top \}^\top $ given
  $X^\prime(0)=X^\prime(re_1)=0$. Envelopes corresponding to 95\%
  pointwise confidence intervals for the Monte Carlo estimates are represented but invisible.}
  \label{fig:p.pcf}
\end{figure}

We end this paragraph with an important consequence of Table~\ref{tab:expansion_pcf}. For any $\L,\L^\prime \subseteq \{0,1,\dots,d\}$ considered in Table~\ref{tab:expansion_pcf}, $r^{d-1} g_{\L,\L^\prime}(r)$ is integrable at 0. In other words, the (cross-)Ripley's $K$-function, $K_{\L,\L^\prime}$ (see~\eqref{eq:derivatives:behaviour:matern} for details) is well-defined for any $r>0$.

\subsection{Higher-order intensity functions}

In this section, we consider an extension of Theorem~\ref{thm:pcf} which provides an explicit expression for the $k$th order intensity of~$\bY_\L$.
\begin{theorem} \label{thm:rhok}
Consider
conditions~\ref{C:general}$[2]$ and \ref{C:nondegeneracy}$[1]$. Let $k\ge
1$ and $t_1,\dots,t_k$ be pairwise distinct points in an open domain $U$ of $\R^d$. Assume that the distribution of \linebreak$V=\{X^\prime(t_1)^\top,\dots,X^\prime(t_k)^\top\}^\top$ is non-degenerate. Then for any $\L\subseteq\{0,\dots,d\}$, the $k$th order intensity function at $(t_1,\dots,t_k)$ of $\bY_\L$ is given by
\begin{align}
\rho^{(k)}_\L(t_1,\dots,t_k) =& f_V(0) \; 
\E 
\Big(
|\det \{X^{\prime\prime}(t_1)\} | \times\dots\times |\det \{X^{\prime\prime}(t_k)\} | \; \times \nonumber \\
& \;
\iota_\L\{X^{\prime\prime}(t_1)\} \times \dots \times \iota_\L\{X^{\prime\prime}(t_k)\}
\mid X^\prime(t_1)=\dots=X^\prime(t_k)=0 
\Big). \label{eq:res_gL_higherdim}        
\end{align}    
\end{theorem}

The proof of Theorem~\ref{thm:rhok} is given in Appendix~\ref{app:proof:rhok}. Comparing with Theorem~\ref{thm:pcf}, it is much harder to give explicit conditions ensuring the non-degeneracy of the vector $V$ when $k \ge 3$.
The integrability of $\rho_\L^{(k)}$ over $W^k$ for
  $k=1,\dots,p$ and hence the existence of the $p$th moment of
  $N_{\L}(W)$ is difficult to deduce from Theorem~\ref{thm:rhok}. 
When $p=2$ this question has a long
history and is related to Geman's condition, see \cite{estrade2016central} and Section~\ref{sec:statistics} for details. \cite{gass2024number} establish the following result for any $p$ under a stronger assumption.

\begin{theorem}[{\cite[Theorem~1.2]{gass2024number}}] \label{thm:gass}
  Assume conditions~\ref{C:general}$[p+1]$ and
  \ref{C:nondegeneracy}$[p]$ for some $p\ge 1$. Then $\E \left\{ N_\L(\Delta)^p \right\}<\infty$ for any compact $\Delta\subseteq W$.
\end{theorem}
\cite{gass2024number} actually use a different version
of~\ref{C:nondegeneracy}$[p]$, since they include $X(t)$ in the
vector. However, an inspection of their proof shows that the current
condition \ref{C:nondegeneracy}$[p]$ is sufficient. 
To illustrate Theorem~\ref{thm:gass} for a Matérn random field, if one assumes that the regularity parameter is such that $\nu>5$, \ref{C:general}$[5]$ and\ref{C:nondegeneracy}$[4]$ holds which leads to $\E\{N_\L(W)^4\}<\infty$.


\section{Simulation of a critical point process} \label{sec:simulation}

Let $\bX$ be a Gaussian random field defined on $\R^d$ with covariance function $c$. In this section, we address the problem of simulating its corresponding critical point process $\bY_\L$ in a bounded domain, say $W=[0,1]^d$. 
We are unable to derive the full distribution of $\bY_\L$ and we
believe that there is no other way of simulating $\bY_\L$ on $W$ than
using the following two steps: (i) simulate a continuous approximate  realization of the Gaussian random field in $W$; (ii) find the critical points of this realization. Step~(ii) is essentially an optimization step. Several approaches could be considered, e.g. methods based on arithmetic intervals \cite{hansen2003global,juliaintervals}. We do not discuss this in the present manuscript and focus on step~(i). 

\subsection{Approximate simulation of smooth Gaussian random fields}

Many methods are available to simulate a stationary and isotropic random field in $W=[0,1]^d$ \cite{lantuejoul2013geostatistical} and the \texttt{R} package \texttt{RandomFields} \cite{schlather2015analysis}. Given the fact that we need a smooth realization to find the critical points in step (ii), we may classify all simulation methods into two categories: (a) Simulate $\bX$ on a lattice and smoothly interpolate it on the continuous space; (b) Approximate a continuous realization of $\bX$  as the ``average'' of continuous realizations. 

Methods of type (a) consist in generating exactly a Gaussian random vector, e.g.\ the Cholesky method or the circulant-matrix based method \cite{wood1994simulation}. To formalize this approach, let $k:\mathbb{R}^d \to \mathbb{R}$ be a smooth bounded density function, compactly supported in $[-1,1]^d$ (an example of it being the so called bump function) and, for all $\xi >0$, let us denote $k_\xi(x) = \xi^{-d} k(\xi^{-1}x)$ its scaled version. 
Also, let $(\xi_n)_n$ be a sequence of real numbers and $(L_n)_n$ be a sequence of lattices on $\mathbb{R}^d$. In particular, we consider the sequence $L_n=\{i/n, i\in (\mathbb Z + 1/2)^d\}$ so that  $W=\cup_{x \in L_n\cap W} \,\mathcal C_{n,x}$ where $\mathcal C_{n,x}$ are disjoint hypercubes centered at $x \in L_n\cap W$ with side-length $1/n$. Note that $L_n\cap W$ densifies in $W$ as $n\to \infty$, that $\#(L_n\cap W)=n^d$ and $|\mathcal C_{n,x}|=n^{-d}$. Thus, we generate a Gaussian random vector $\{X(x),x\in L_n\cap W\}$ and define a smoothed version $\bX_n$ by
\begin{equation} \label{eq:Xa}
  X_n (t) = \sum_{x\in L_n \cap W} n^{-d} k_{\xi_n}(t-x) X(x)
\end{equation}
for all $t \in W$. Theorem~\ref{thm:Xab}(i) details conditions on $\xi_n$ required to prove the convergence of $\bX_n$ to $\bX$ on~$W$.

To illustrate methods of type~(b), we present the spectral approach \cite{lantuejoul2013geostatistical}: let $U_i$, $V_i$ and $W_i$ for $i=1,\dots,n$ ($n\ge 1$) be realizations of independent random variables with $U_i\sim \mathcal{U}([0,2\pi])$, $V_i \sim F$ (where $F$ is the spectral measure of $\bX$)  and $W_i\sim \mathcal U([0,1])$. Then we define for any $t\in [0,1]^d$,
\begin{equation} \label{eq:Zib}
    Z_i(t) = \sqrt{-2\log(W_i)} \cos\left( U_i + t^\top V_i\right).
  \end{equation}
  For isotropic correlation functions, the simulation of the $d$-dimensional random variables $V_i$ can be reduced to generate one-dimensional random variables using the inverse transform method.
The random fields $\bZ_i=\{Z_i(t), t\in W\}$ are obviously infinitely
continuously differentiable on $W$ and can be shown to be centered with correlation function $c$. The factor $\sqrt{-2\log(W_i)}$ makes the marginal distribution of $Z_i(t)$ Gaussian for any $t\in W$ but $\bZ_i$ is not a Gaussian random field. By appeal to the central limit theorem, the fields $\bZ_i$ are therefore averaged in the final step of the spectral method. In other words, a continuous approximate realization $\bX_n$ of $\bX$ is defined for any $t\in W$ by
\begin{equation} \label{eq:Xb}
   X_n(t)= \frac{1}{\sqrt{n}} \sum_{i=1}^{n} Z_i(t).
\end{equation}

The advantage of type (a) methods is that $\bX_n$ is exact on the
finite set $L_n\cap W$ and constitutes a Gaussian random field on
$W$. However, the covariance function of $\bX_n$ differs from $c$. The
advantage of the spectral method is that we directly obtain a
continuous random field with the same correlation function as
$\bX$. But in this case $\bX_n$ only converges to a Gaussian random
field and is not a Gaussian random field itself. 

\subsection{Convergence of approximate simulations}

Following the previous section, an important feature is to understand how critical points of $\bX_n$ converge to the ones of $\bX$. We first consider this problem and then apply general conditions on random fields  $\bX_n$ given by~\eqref{eq:Xa} or~\eqref{eq:Xb}.

Let $S=\{x_1,\dots,x_p\}$ and, for any $n \ge 1$, $S^n=\{x^n_1,\dots,x^n_{p_n}\}$ be finite subsets of
$\mathbb{R}^d$. We say that $S^n$ converges to $S$, $S^n\to S$, if and
only if $\lim_{n \rightarrow \infty} p_n=p$ and 
\begin{equation} \label{eq:convcrit}
   \exists\, \phi_n:\{1,\dots,p\}\to
    \{1,\dots,p_n\}, \, n \ge 1,  \text{ such that } \forall k=1,\dots,p, \, x^n_{\phi_n(k)} \to x_k.
\end{equation}
For any set $E \subset \R^d$ and any $\ell\ge 1$, we consider the space $C^\ell(E)$ equipped with the norm 
\begin{equation}\label{eq:extendedsupnorm}
    \|f\|_{C^\ell(E)}  = \sum_{|\alpha|\leq \ell} \sup_{x\in E} |\partial^{\alpha} f(x)|
\end{equation}
for any $f\in C^\ell(E)$. In the following, the convergence  $f_n \to
f$ in $C^\ell(E)$ as $n\to \infty$ means that
$\|f_n-f\|_{C^\ell(E)}\to0$, and the almost sure convergence of random
fields in $C^\ell(E)$ follows accordingly. Moreover, a sequence of
random fields $(\bX_n)_n$ is said to converge in distribution to
$\bX$ in $C^\ell(E)$, and we write $\bX_n \xrightarrow{\mathcal{D}} \bX$ in
$C^\ell(E)$, if $\E\{ f(\bX_n) \} \to \E\{ f(\bX) \}$ for any  continuous function from ${C^\ell(E)}$ to $\R$.
Recall that $f\in C^2(W)$ is a \emph{Morse function} if it has no
degenerate critical points, i.e.\ $f^\prime(t) = 0$ implies
$\det\{f^{\prime\prime}(t)\} \neq 0$ for all $t \in W$. We further
call $f$ an \emph{MB function} if it is a Morse function and has no
critical point on the boundary of $W$. The next result provides a sufficient condition for the convergence of the set of critical points of any specified index.

\begin{theorem} \label{thm:convcrit}
Let $\bX$ (resp.\ $(\bX_n)_n$) be a random field (resp.\ sequence of
random fields) on $W$ with sample paths in $C^2(W)$. Assume that $\bX$ is a.s.\ an MB function, and that
$\bX_n \xrightarrow{\bullet} \bX$ in $C^2(W)$ as $n\to \infty$. Then, for all $\ell = 0,\dots,d$, 
    $\bY_{n,\ell} 
    \xrightarrow{\bullet}
    \bY_\ell$
where $\bY_\ell$ and $\bY_{n,\ell}$ are defined according to Definition~\ref{def:critical} but ``$t\in \mathbb{R}^d$'' is replaced by ``$t\in W$'' (and $\bX$ is replaced by $\bX_n$ for the definition of $\bY_{n,\ell}$) and where the convergence mode $\bullet$ is $\mathcal D$ or a.s.
\end{theorem}
A proof is given in Appendix~\ref{app:convergence:critical:points}. We now apply this result.

\begin{theorem} \label{thm:Xab} 
The following results hold. \\
\noindent(i) Assume \ref{C:general}$[2+\varepsilon]$ and let
$(\bX_n)_n$ be a sequence of random fields given by~\eqref{eq:Xa}. Assume that the kernel density $k \in C^3(\R^d)$ and that as $n \to \infty$, $\xi_n\to 0$ and $n^{-1} \xi_n^{-d-3} \to 0$. Then, as $n\to \infty$, $\bX_n \xrightarrow{a.s.} \mathbf{X}$
in $C^2(W)$ and thus,  $\bY_{n,\ell} \xrightarrow{a.s.} \bY_\ell$.
\\ 
\noindent(ii) Assume \ref{C:general}$[3+\lfloor d/2 \rfloor]$  and let $(\bX_n)_n$ be a sequence of random fields given by~\eqref{eq:Xb}. Then, as $n\to \infty$, $\bX_n \xrightarrow{\mathcal{D}} \mathbf{X}$
in $C^2(W)$  and thus,  $\bY_{n,\ell} \xrightarrow{\mathcal D} \bY_\ell$. 
\end{theorem}

\noindent A proof is given in Appendix~\ref{app:convergence:simulation}.  To rephrase this result, if $(\bX_n)_n$ is given by~\eqref{eq:Xa}, we have a slightly stronger condition on $\bX$, namely \ref{C:general}$[2+\varepsilon]$, than the one used in Theorem~\ref{thm:pcf} (for which $\varepsilon=0$). Also, we need the bandwidth to satisfy some specific convergence rates. Starting with a random field $(\bX_n)_n$  given by~\eqref{eq:Xb} seems easier (no bandwidth tuning is required). However, we have a stronger condition on the regularity of $\bX$ (at least $C^{5}_{\rm a.s.}(W)$ in dimension 2 for instance).


\section{Asymptotics for functionals of critical point processes} \label{sec:statistics}

Considering $\bY_\L$ as a statistical model, it is relevant to investigate asymptotic
results for functionals of $\bY_\L$. We consider increasing domain
asymptotics and assume, to ease the presentation, that the critical
point process is observed in the sequence of bounded domains
$W_n=[-n/2,n/2]^d$ (with volume $n^d$), so that $W_n\to \R^d$ as $n\to
\infty$. \cite{estrade2016central,nicolaescu2017clt}, and
\cite{azais2024multivariate} considered consistency and
asymptotic normality for normalized versions of $N_\L(W_n)$, the number of critical
points above a level $u$, and linear combinations of these numbers for different sets $\L$. In Sections~\ref{sec:integralandhermite} and \ref{sec:asymptoticresults} we significantly extend the results in the previous papers by considering general linear and bilinear statistics described in Section~\ref{sec:linearstatistics}. Our results require a few additional assumptions described and discussed in Section~\ref{sec:assumptions}. 

\subsection{Linear and bilinear statistics}\label{sec:linearstatistics}

Let $\L \subseteq\{0,\dots,d\}$ and let $\{\phi_{1,n}\}_{n
  \ge 1}$ and $\{\phi_{2,n}\}_{n \ge 1}$  be two sequences of 
measurable (test) functions with $\phi_{1,n}:\R^d \to \R$ and
$\phi_{2,n}:\R^d \times \R^d \to \R$. We  denote
generic points in $\R^d \times \R^d$ as $\mathbf{t} = (t^{(1)},
t^{(2)})$ or $\mathbf{s} = (s^{(1)}, s^{(2)})$, and consider the functionals $\Phi_{1,n}$ and $\Phi_{2,n}$ defined by
\begin{align}
\Phi_{1,n} &= \sum_{t \in \bY_\L \cap W_n} \phi_{1,n}(t) \label{eq:Phi1} \\
\Phi_{2,n} &= \sum_{\mathbf t\in (\bY_\L \cap W_n)^2}^{\neq} \phi_{2,n}(\mathbf{t}). \label{eq:Phi2}
\end{align} 
We restrict attention to functions $\phi_{2,n}$ with support included in the diagonal cylinder
\begin{equation}\label{eq:D:eta:R}
    D_{\eta, R} := \left\{ \mathbf{t} \in (\mathbb{R}^d)^2 : \eta\leq \|t^{(1)} - t^{(2)} \| \leq R
    \right\},
\end{equation}
for some $0 < \eta < R$ fixed in the rest of the paper. The use of
$\eta>0$ renders  the sign $\neq$ in \eqref{eq:Phi2} superfluous. The parameter $\eta$
will be motivated below.

Statistics of the form~\eqref{eq:Phi1} (resp.\ \eqref{eq:Phi2}) are called linear statistics (resp.\ bilinear statistics). We list a few examples. If $\phi_{1,n}(t)=\mathbf 1(t\in \Delta_n)$ for any bounded domain $\Delta_n \subseteq W_n$, $\Phi_{1,n}=N_\L(\Delta_n)$ is the number of critical points with indices $\L$ in $\Delta_n$. 
Thus, $\phi_{1,n}(t)=\mathbf 1(t\in W_n)/|W_n|$, yields
$\Phi_{1,n}=N_\L(W_n)/|W_n|= \hat \rho_\L$, i.e.\ the natural unbiased
estimator of the intensity parameter $\rho_\L$
\cite[e.g.][]{moller2003statistical}. Another example is
$\phi_{1,n}(t) = \rho_\L^{-1} n^{-d} \phi_1(t/n)$ where $\phi_1$ is an
integrable function on $W_1=[-1/2,1/2]^d$. Then, $\Phi_{1,n}$ is an
unbiased estimate of $I=\int_{[-1/2,1/2]^d} \phi_1(t)\dd t$ since
$\E(\Phi_{1,n}) = I$ by the Campbell theorem. Using point processes as
quadrature points has recently been considered
by~\cite{bardenet2020monte} and \cite{coeurjolly2021monte}. 

Regarding $\Phi_{2,n}$ consider for some $r \in [\eta, R]$, the two-dimensional functional 
\begin{equation}\label{eq:exPhi2}
    \phi_{2,n}(\mathbf{t}) = \frac{\mathbf{1}\{\mathbf{t} \in D_{\eta, r}\}}{\rho_\L^2 |W_n \cap (W_n)_{t^{(1)} - t^{(2)}}|},
\end{equation}
where $\Delta_{t} = \{s + t, s\in \Delta\}$ for any domain $\Delta$
and $t \in \mathbb{R}^d$. By the second-order Campbell Theorem,
$\E(\Phi_{2,n}) = K_{\eta,\L}(r)$ with
\begin{equation}
    K_{\eta,\L }(r) = \int_{B(0,\eta,r)} g_\L(\|t\|) \dd t
\end{equation}
where $B(x,\eta,r)=B(x,r)\setminus B(x,\eta)$ is the shell between the hyperspheres of radius $\eta$ and $r$.
Obviously, $K_{0,\L} = K_{\L}$ is the standard Ripley's $K$-function. However, the modification with $\eta>0$ is  considered
mainly for two reasons: (a) to upper bound the Hermite coefficient
$d_{\mathbf{a}}(r)$, defined below in Proposition~\ref{prop:chaos},
uniformly in $r\in[\eta,R]$; (b) to control the quantity
$\gamma_{\mathbf{a},\mathbf{b}}(\mathbf{t},\mathbf{s})$, defined below
in Lemma~\ref{lem:gamma}, when $\mathbf{t}$ or $\mathbf{s}\in
\mathbb{R}^d\times \mathbb{R}^d$ away from the diagonal. To extend
the results of the present paper to the standard Ripley's $K$-function ($\eta=0$)
requires further challenging derivations that we leave for future
work.

To study the asymptotic behavior of the translation edge corrected estimator of
the modifed Ripley's $K$-function estimator, given by 
\begin{equation}\label{eq:def:K:hat}
\hat K_{\eta, \L}(r) = \Phi_{2,n}\; \; \frac {\rho_\L^2}{\hat \rho_\L^2},
\end{equation}
the first step is clearly to derive results for $\hat \rho_\L=\Phi_{1,n}$ (for $\phi_{1,n}=\mathbf{1}( \cdot \in W_n)/|W_n|$) and 
for $\Phi_{2,n}$ given by~\eqref{eq:exPhi2}. 
The second step is to investigate the joint distribution of $(\Phi_{1,n},\Phi_{2,n})$. Both steps are tackled in Theorem~\ref{thm:clt}.

Other edge corrected estimates using e.g.\ the border or isotropic
edge correction \cite[see eg][]{moller2003statistical} can also be
rewritten as two-dimensional functionals similar
to~\eqref{eq:exPhi2}. Furthermore, \eqref{eq:exPhi2} could be extended
to bilinear functionals where $t^{(1)} \in \bY_\L$ and $t^{(2)} \in
\bY_{\L^\prime}$ for different sets $\L,\L^\prime\subseteq
\{0,1,\dots,d\}$. Such functionals would lead to estimates of modified
cross-Ripley's $K$-functions between $\bY_\L$ and
$\bY_{\L^\prime}$. We do not detail the asymptotics in this case but
outline the required changes in 
Theorems~\ref{thm:variance}-\ref{thm:clt} in a remark after the
presentation of these results.

\subsection{Additional assumptions}\label{sec:assumptions}

Our main results require additional assumptions that we now present. 
\begin{enumerate}[($\mathcal C$.1)]
\setcounter{enumi}{3}
\item There exists a decreasing function $\Xi$ on $\R^+$, such that $
  \max_{p=1,\ldots,4}\{ r^{p} |c_2^{(p)}(r^2)| \} \le \Xi(r)$ and such
  that $\int_0^\infty r^{d-1} \Xi(r) \dd r<\infty$.\label{C:integrability} 
\item The Gaussian random vector  $\{X^\prime(0)^\top,X^{\prime\prime}(0)^\top,X^\prime(re_1)^\top,X^{\prime\prime}(re_1)^\top\}^\top$ is \linebreak non-degenerate for any $r \in [\eta,R]$. \label{C:nondegeneracy:hessian}
\end{enumerate}

Assume conditions~\ref{C:general}$[p+1]$ and
\ref{C:nondegeneracy}$[p]$. By Theorem~\ref{thm:gass}, $\E \{
N_{0:d}(\Delta)^p\}<\infty$ for any compact subset $\Delta$ of
$\R^d$. Following \cite[Remark 1]{gass2024number}, the two conditions
also ensure that the function $v \mapsto \E \{ N_{0:d}(\Delta,v)^p\}$
is continuous, where $N_{0:d}(\Delta,v)=\#\{X^\prime(t)=v, t\in
\Delta\}$ (note $N_{0:d}(\Delta)=N_{0:d}(\Delta,0)$). Therefore, we can apply the same argument as in
\cite[Proof of Proposition 1.1, p. 3872]{estrade2016central} to prove
that as $\varepsilon\to 0$, 
\begin{equation} \label{eq:NXprimeLp}
    N_{0:d,\varepsilon}(\Delta) := \int_\Delta |\det \{X^{\prime\prime}(t)\}| \delta_\varepsilon \{X^\prime(t)\}\dd t \to N_{0:d}(\Delta) \text{ a.s. and in } L^p
\end{equation} 
where $\delta_\varepsilon(\cdot) = \mathbf 1(\| \cdot \|\le
\varepsilon)/|B(0,\varepsilon)|$ is the uniform density on $B(0,\varepsilon)$. The $L^p$ convergence is used in Proposition~\ref{prop:representation}(ii). 

When $p=2$, the continuity of $ v \mapsto\E \{ N_{0:d}(\Delta,v)^p\}$
and~\eqref{eq:NXprimeLp} is also known as Geman's condition. Conditions~\ref{C:general}$[3]$ and \ref{C:nondegeneracy}$[2]$ can be weakened to obtain this result. Indeed,~\cite[Proposition 2.4.]{azais2024multivariate} require only \ref{C:general}[2] and the following assumption:
\begin{equation}\label{eq:c4at0}
\exists \delta>0 \; \text{ such that }
\int_{\|t\|\le \delta} \frac{\|c^{(4)}(t)-c^{(4)}(0)\|}{\|t\|^d} \dd t<\infty.
\end{equation}
In Lemma~\ref{lem:isotropic}, we prove that~\eqref{eq:c4at0} is fulfilled if there exists $\delta>0$ such that 
\begin{equation} \label{eq:Xi1}
\int_0^\delta r^{d-1}\widetilde{\Xi}(r)\dd r<\infty \text{ where }
\widetilde{\Xi}(r) = \max \left\{ 
\frac{|c_2^{\prime\prime}(r^2)-c_2^{\prime\prime}(0)|}{r^d}  , 
\frac{|c_2^{\prime\prime\prime}(r^2)|}{r^d-2} , 
\frac{| c_2^{(4)}(r^2)|}{r^{d-4}}
\right\}.    
\end{equation}
Lemma~\ref{lem:gemanL2} investigates~\eqref{eq:Xi1} for the Matérn and RWM. 
We can easily show that~\ref{C:general}[2] and \eqref{eq:Xi1} (hence Geman's condition) are always valid for RWM and valid for $\nu\ge 5/2$ for Matérn random fields.
Our main results on bilinear functionals (integral representation,
chaos expansions and asymptotic results) require~\eqref{eq:NXprimeLp}
for $p=4$, which is why we do not insist too much
on~\eqref{eq:Xi1}. It is an open question to obtain for $p>2$ an
assumption similar to~\eqref{eq:Xi1} that would be a less strict
alternative to condition \ref{C:general}$[p+1]$.

Condition~\ref{C:integrability} is essential for controlling
covariances involving Hermite polynomials and to ensure finiteness of limiting variances of $\Phi_{1,n}$ and $\Phi_{2,n}$. Lemma~\ref{lem:isotropic} in Appendix~\ref{sec:link:to:anisotropy} proves that condition~\ref{C:integrability} is the isotropic version of Assumption~(A3) of~\cite{azais2024multivariate} known as Arcones' condition,
\begin{equation}\label{eq:arcones}
    \max\{ |\partial^\alpha c(t)|, \alpha \in \mathbb N^d, 0\le |\alpha|\le 4\} \in L^1(\R^d).
\end{equation}
Condition~\ref{C:integrability} for $|\alpha|=0$ just means that $c
\in L^1(\R^d)$ and hence by Bochner's theorem ensures that the
Gaussian random field $\bX$ admits a spectral density function,
denoted by $f$ in the following. Existence of $f$ is the fundamental
ingredient of the proof of the central limit theorem, Theorem~\ref{thm:clt},  based on rewriting  functionals $\Phi_{1,n}$ and $\Phi_{2,n}$ as multiple
Wiener-Itô integrals. 

We inspect condition~\ref{C:integrability} for the Matérn correlation function and the RWM. For the Matérn case, using recurrence and asymptotic properties for the modified Bessel function \cite{abramowitz1968handbook} (see Appendix \ref{sec:bessel:derivatives} for details), we have for any integer $p\ge 1$ and any $r>0$ and $z=r\sqrt{2\nu}/\varphi$,
\begin{align}
c_2^{(p)}(r^2 ) &= \frac{(-1)^p}{2^p \varphi^{2p}} \, \frac{2 \nu^p}{\Gamma(\nu)} \; \left(\frac z2\right)^{\nu-p} K_{\nu-p}(z) \sim_{z\to \infty} \kappa \, z^{\nu-p-1/2} \exp(-z)  \label{eq:derivatives:c2:matern} 
\end{align}
for some constant $\kappa$. The exponential decay at infinity ensures condition~\ref{C:integrability} for any $\nu>0$. 
For the RWM, using again Appendix~\ref{sec:bessel:derivatives}, we have for any $p\ge 1$, $r>0$ and $z=r\sqrt{d}/\varphi$,
\begin{align}
    c_2^{(p)}(r^2 ) &=  \frac{(-1)^p}{2^p \varphi^{2p}} \,
    \left(\frac d2\right)^p \Gamma(d/2) \;
    \left(\frac{z}2\right)^{-(d/2-1+p)} J_{d/2-1+p}(z)\label{eq:derivatives:c2:RWM}\\
    &=    \kappa \frac{1}{r^{d/2-1+p+1/2}} \, 
        \left[\cos
        \left\{ \frac{r\sqrt{d}  }{\varphi}  - \left( \frac{d}2-1+p \right) \frac\pi 2 -\frac \pi4
        \right\} + O(r^{-1})   \right] \nonumber
\end{align}
as $r\to \infty$. The decay at infinity is definitely too slow and it
is easily seen that \ref{C:integrability} does not hold for any $d\ge 1$.  

Finally, condition~\ref{C:nondegeneracy:hessian} is essential for
Hermite expansions and asymptotic results for bilinear statistics. Of
course, \ref{C:nondegeneracy:hessian} extends
\ref{C:nondegeneracy:gradient}$[r]$ for any $r \in [\eta,R]$. A
summary of the validity of the conditions required in this paper
applied to the Matérn model and the RWM is provided in
Table~\ref{tab:maternAndRWM}. Our asymptotic results
Theorems~\ref{thm:variance}-\ref{thm:clt} do not cover the RWM for which
\ref{C:integrability} is not valid.

\begin{table}[htbp]
\begin{center}
\begin{tabular}{r|c|c}
& Matérn model & RWM \\
\hline
&&\\
\ref{C:general}$[p+\varepsilon]$& \cmark \; for any $\nu>p$ and $\varepsilon<\lceil \nu \rceil-1-p$ & \cmark \;  for any $p,\varepsilon\ge 0$ \\
&&\\
\ref{C:nondegeneracy}$[p]$ & \cmark \;  if \ref{C:general}$[p]$ is satisfied & \cmark \;  for $p=1$, \xmark \;  for any $p\ge 2$ \\ 
&&\\
\ref{C:nondegeneracy:gradient}$[r]$ & Numerically checked & Numerically checked \\
&&\\
\ref{C:integrability} & \cmark \; for any $\nu>0$ & \xmark\\
&&\\
\ref{C:nondegeneracy:hessian} & Numerically checked & Numerically checked \\
\hline\end{tabular}
\caption{A summary of validity of assumptions considered in this paper for the Matérn random field and random wave models.}\label{tab:maternAndRWM}
\end{center}
\end{table}

\subsection{Integral representations and Hermite expansions}\label{sec:integralandhermite}

The following result extends \cite[Proposition 1.1]{estrade2016central} (see also \cite[Lemma~3.1]{azais2024multivariate}) and provides an integral representation for functionals of the form~\eqref{eq:Phi1}-\eqref{eq:Phi2}.

\begin{proposition} \label{prop:representation}
We have the following limits in the a.s.\ and  $L^2$-sense.\\
(i) Assume \ref{C:general}$[3]$, \ref{C:nondegeneracy}[2],
and let $\phi_{1,n}:W_n\to \R$ be a bounded and piecewise continuous
function. Then the functional $\Phi_{1,n}$ given by~\eqref{eq:Phi1} has the following representation for any $n\ge 1$,
\begin{equation}\label{eq:representationH1}
\Phi_{1,n} =   \lim_{\varepsilon \to 0} \Phi_{1,n,\varepsilon} 
\quad \text{ with  }\quad
\Phi_{1,n,\varepsilon}  = \int_{W_n} \phi_{1,n}(t) |\det \{X^{\prime\prime}(t)\}|
\, \iota_\L \{X^{\prime\prime}(t)\} \delta_\varepsilon\{X^\prime(t)\} \, \dd t.
\end{equation}
\noindent(ii) Assume \ref{C:general}[5], \ref{C:nondegeneracy}[4],
and let $\phi_{2,n}:W_n\times W_n\to \R$ be a bounded and piecewise
continuous function. Then the functional $\Phi_{2,n}$ given by~\eqref{eq:Phi2} has the following representation for any $n\ge 1$,
\begin{align}
\Phi_{2,n} =&\lim_{\varepsilon\to 0} \Phi_{2,n,\varepsilon}
\quad \text{ with } \\
\Phi_{2,n,\varepsilon}&=
\int_{(W_n)^2} \phi_{2,n}(\mathbf{t}) 
|\det \{X^{\prime\prime}(t^{(1)})\}| \times |\det \{X^{\prime\prime}(t^{(2)})\}| \times\nonumber\\
& \hspace*{2cm}  \iota_\L\{X^{\prime\prime}(t^{(1)})\}  \iota_{\L} \{X^{\prime\prime}(t^{(2)})\}
\delta_{\varepsilon}\{X^\prime(t^{(1)}) \}  \delta_{\varepsilon}\{X^\prime(t^{(2)}) \}
 \dd \mathbf t. 
\label{eq:representationPhi2}
\end{align}
\end{proposition}
The proof of Proposition~\ref{prop:representation} is given in Appendix~\ref{sec:representation}.
The consequence of Proposition~\ref{prop:representation} is that both
$\Phi_{1,n}$ and $\Phi_{2,n}$ can be seen as functionals of
$X^\prime(t)$ and $X^{\prime\prime}(t)$ for $t\in W_n$. Thereby we can
obtain (Proposition~\ref{prop:chaos} below) a chaos expansion using
Hermite expansion, following ideas from~\cite{kratz2006second} and \cite{estrade2016central}. 

Before turning to the chaos expansion some more notation is
needed. Let $D = d+d(d+1)/2$. For all $n\in \mathbb{N}$, denote $H_n$ the $n$th Hermite polynomial. For any $p\ge 1$ and
multi-index $a\in \mathbb N^p$, we remind that $|a| = \sum_{i=1}^{p} a_i$ and let $a! = \prod_{i=1}^p a_i!$. Denote $H_{\otimes a}$ the multivariate polynomial defined by $H_{\otimes a}(y) = \prod_{i=1}^p H_{a_i}(y_i)$ for all $y\in \R^p$. These polynomials form an orthonormal basis of squared integrable functions in $\R^p$ with respect to the $p$-dimensional standard Gaussian density.
For any $t\in \R^d$ let $\check X(t)=\{X^\prime(t)^\top,X^{\prime\prime}(t)^\top\}^\top\in \R^D$ and, for any $\mathbf{t}\in D_{\eta,R}$, $\check X(\mathbf{t})=[\check X\{t^{(1)}\}^\top,\check X\{t^{(2)}\}^\top]^\top \in \R^{2D}$ (here and in the following we abuse notation by distinguishing functions by the notation for their arguments). Let $\Sigma(t) = \Var\{\check X(t)\}$ and $\Sigma(\mathbf t) = \Var\{\check X(\mathbf{t})\}$. By assumption \ref{C:nondegeneracy}[2] and \ref{C:nondegeneracy:hessian} respectively, $\Sigma(t)$ and $\Sigma(\bf t)$ are non-degenerate.
Thus, $\check Y(t):=\Sigma(t)^{-1/2} \check X(t)$ and $\check Y(\mathbf{t}):=\Sigma(\mathbf{t})^{-1/2} \check X(\mathbf{t})$ are respectively $D$- and $2D$-dimensional standard normal Gaussian random vectors, where $M^{-1/2}$ denotes any inverse square root matrix of a definite positive matrix $M$ such that $M^{-1/2}M^{-1/2}=M^{-1}$. By stationarity, isotropy and independence of $X^\prime(t)$ and $X^{\prime\prime}(t)$ for any $t$ \cite[see e.g.][]{adler2007random}, 
\begin{equation*}
\Sigma(t)=\Sigma(0)=\mathrm{diag}[ \Var\{ X^\prime(0)\},\Var\{ X^{\prime\prime}(0)\}] =: \mathrm{diag} \{\lambda_2 I_d , \Sigma_2(0)\}    
\end{equation*}
and $\Sigma(\mathbf{t})=\overline{\Sigma}(\|t^{(1)} - t^{(2)}\|) $ is the block matrix depending on $\|t^{(1)}-t^{(2)}\|$ given for any $r>0$ by
\begin{equation*}
\bar \Sigma(r) =
\begin{pmatrix}
 \Sigma(0) & \tilde \Sigma(r) \\ \tilde \Sigma(r) & \Sigma(0)  
\end{pmatrix} \quad \text{ with } \tilde\Sigma(r)= \Cov \left\{ \check X(0), \check X(re_1)\right\}.
\end{equation*}
We define $\check X_0(t)= \{0^\top,X^{\prime\prime}(t)^\top\}^\top$ and $\check X_0(\mathbf{t})= [0^\top,X^{\prime\prime}\{t^{(1)}\}^\top,0^\top,X^{\prime\prime}\{t^{(2)}\}^\top]^\top$. To distinguish Hermite expansions of linear and bilinear statistics, we use the notation $a\in \mathbb N^D$ and $\mathbf a \in \mathbb N^{2D}$. 
Finally, we define a set $\L\subseteq \{0,\dots,d\}$ as symmetric if $i\in \L \Rightarrow d-i \in \L$ for $i=0,\dots,d$. An example of such set is the set of all critical points $\L=\{0,\dots,d\}$ or the set of local extrema $\L=\{0,d\}$. For $\L\subseteq \{0,\dots,d\}$, we let $\mathbb N_{\L}^*$ denote the set $2 \mathbb N \setminus \{0\}$ when $\L$ is symmetric and $\mathbb N\setminus \{0\}$ otherwise.

\begin{proposition} \label{prop:chaos}
The two following statements hold.\\
(i) Assume \ref{C:general}$[3]$ and \ref{C:nondegeneracy}[2],
and let $\phi_{1,n}:W_n\to \R$ be a bounded and piecewise continuous function for any $n\ge 1$. Then the centered functional $\cent{\Phi}_{1,n}=\Phi_{1,n}-\E(\Phi_{1,n})$ can be represented as
\begin{align}
\cent{\Phi}_{1,n} =  \sum_{q\in \mathbb{N}_\L^*} 
 \sum_{\substack{a\in \mathbb N^D, \\ |a|=q }} 
 d_a \int_{W_n} \phi_{1,n}(t) H_{\otimes a}\{\check Y(t)\} \dd t\label{eq:chaosh1}
\end{align}
where for any $a=(\underbar{a},\bar{a})\in \mathbb N^D$ with $\underbar a\in \mathbb N^d$ and $\bar a\in \mathbb N^{D-d}$, the Hermite coefficient $d_a$, which depends on  $\L$, is given by
\begin{align}
d_a &=  \frac{1}{a!}
f_{X^\prime(0)}(0)  \times  
\E 
\left[
H_{\otimes a} \{\Sigma(0)^{-1/2} \check X_0(0) \} 
\,|\det\{X^{\prime\prime}(0)\}| \,
\iota_\L \{X^{\prime\prime}(0)\}  \; \mid \; X^\prime(0)=0 
\right] \nonumber \\
&= \frac{1}{a!}
\frac{H_{\otimes \underbar a}(0)}{(2\pi \lambda_2)^{d/2}}  \times 
\E
\left[
H_{\otimes \bar a} \{\Sigma_2(0)^{-1/2} X^{\prime\prime}(0) \} 
\,|\det\{X^{\prime\prime}(0)\}| \,
\iota_\L \{X^{\prime\prime}(0)\}
\right] \label{eq:da}
\end{align} 
where $a!=(\underbar a!)(\bar a!) = \prod_{i=1}^D a_i!$.

\noindent(ii) Assume \ref{C:general}[5], \ref{C:nondegeneracy}[4],
\ref{C:nondegeneracy:hessian} and let  $\phi_{2,n}:W_n\times W_n\to \R$ be a bounded and piecewise continuous function with support in $D_{\eta,R}$ for any $n\ge 1$. Then the centered functional $\cent{\Phi}_{2,n}=\Phi_{2,n} - \E(\Phi_{2,n})$, where $\Phi_{2,n}$ is given by~\eqref{eq:Phi2}, can be represented as
\begin{align}
\cent{\Phi}_{2,n}
&=   \sum_{q \in \mathbb N_{\L}^*} \sum_{\substack{ \mathbf a\in \mathbb N^{2D}, \\ |\mathbf a|=q }} 
\int_{(W_n)^2} 
\phi_{2,n}(\mathbf{t}) d_{\mathbf a}(\mathbf{t}) H_{\otimes \mathbf a} \left\{ \check Y(\mathbf{t})\right\}
\dd \mathbf{t}. \label{eq:chaosh2}
\end{align}
The Hermite coefficient $d_{\mathbf a}(\mathbf{t})=d_{\mathbf a}\{\|t^{(1)} - t^{(2)}\|\}$, which depends on  $\L$ and $\|t^{(1)} - t^{(2)}\|$, is given for 
$\mathbf a \in \mathbb N^{2D}$ and $r= \|t^{(1)} - t^{(2)}\| \in [\eta,R]$ by
\begin{align}
 d_{\mathbf a}  (r) = &\frac{1}{\mathbf a!} \; f_{V(r)}(0,0) \; \times\;
 \E \Bigg(
H_{\otimes \mathbf a} \left\{ \bar \Sigma(r)^{-1/2} \check X_0(0,re_1)\right\} \times  
\,|\det\{X^{\prime\prime}(0)\}| \,|\det\{X^{\prime\prime}(re_1)\}| \,\nonumber\\
& \hspace*{3cm}  \iota_\L \{X^{\prime\prime}(0)\} \iota_{\L} \{X^{\prime\prime}(re_1)\} 
\mid \; X^{\prime}(0)=X^{\prime}(re_1)=0
\Bigg). \label{eq:dast}
\end{align}
Furthermore, $d_{\mathbf a}(\cdot)$ is continuous on $[\eta,R]$ and there exists $\kappa$ independent of $\mathbf a$ such that 
\begin{equation} \label{ineq:da}
\sup_{\eta \leq r \leq R} \, d_{\mathbf a}(r) \le \kappa\frac{ (4D)^{2|\mathbf a |}}{\sqrt{\mathbf a!}}.
\end{equation}
\end{proposition}

The proof of Proposition~\ref{prop:chaos} is given in Appendix~\ref{app:proof:prop:chaos}. The Hermite expansions~\eqref{eq:chaosh1} and~\eqref{eq:chaosh2} are very useful as they separate almost all contributions (test functions $\phi_{1,n}$ or $\phi_{2,n}$, $\L$ and $\bX$) to the functionals $\cent{\Phi}_{1,n}$ and $\cent{\Phi}_{2,n}$ respectively:  the Hermite polynomial factors $H_{\otimes a}\{\check Y(t)\}$ (resp.\ $H_{\otimes \mathbf{a}}\{\check Y(\mathbf{t})\}$) do not depend on the test functions and the subset $\L$ while the Hermite coefficients $d_a$ (resp.\ $d_{\mathbf a}  (\mathbf{t})$) do not depend on the test function.

\subsection{Limiting variance and asymptotic normality}\label{sec:asymptoticresults}

Using Proposition~\ref{prop:chaos}, we follow the approach proposed by~\cite{estrade2016central} (see also \cite{azais2024multivariate}) and derive in Theorem~\ref{thm:variance} below the asymptotic variance of centered functionals  $\cent{\Phi}_{1,n}$ and $\cent{\Phi}_{2,n}$. Before turning to Theorem~\ref{thm:variance} we need a preliminary lemma and some notation for this.

For any $t\in \mathbb{R}^d$ and $\mathbf{t}, \mathbf{s} \in \mathbb{R}^{2d}$, we denote by
$t \ominus \mathbf{s} = \{t - s^{(1)}, t - s^{(2)}, s^{(1)} - s^{(2)}\}$ and $\mathbf{t} \ominus \mathbf{s} = \{t^{(1)} - t^{(2)}, t^{(1)} - s^{(1)}, t^{(1)} - s^{(2)}, t^{(2)} - s^{(1)}, t^{(2)} - s^{(2)}, s^{(1)} - s^{(2)}\}$, the triplet and 6-tuple of all the differences between two elements of  $\{t,s^{(1)}, s^{(2)}\}$ and 
$\{t^{(1)}, t^{(2)}, s^{(1)}, s^{(2)}\}$. We abuse notation and denote
by $\|t\ominus \bf s\|$ and $\|\bf t\ominus \bf s\|$ the vectors of
norms of the differences.  Furthermore, let $\|M\|_{\max}:= \ell
\times\max_{i,j} |M_{ij}|$ denote the max-norm of an $\ell \times \ell$
square matrix $M$ modified to become sub-multiplicative (i.e.\
$\|AB\|_{\max}\le\|A\|_{\max}\|B\|_{\max}$ for $\ell \times \ell$
matrices $A$ and $B$). The following useful lemma is proved in Appendix~\ref{app:var}.

\begin{lemma}\label{lem:gamma}
  Assume \ref{C:general}$[2]$ and \ref{C:nondegeneracy:hessian}. Let $a,b\in \mathbb N^D$, $\mathbf a,\mathbf b\in \mathbb N^{2D}$, $t,s \in \R^d$, $\mathbf t, \mathbf s \in \R^{2d}$. We define three covariance functions by
\begin{align*}
\gamma_{a,b} (t,s )  = \gamma_{a,b} (t-s)&:= \E  
\left[ 
H_{\otimes a} \{\check Y(t)\}H_{\otimes b} \{\check Y(s)\} 
\right] \\
\gamma_{\mathbf a,\mathbf b} (\mathbf t,\mathbf s ) = \gamma_{\mathbf a,\mathbf b} (\mathbf t\ominus \mathbf s) &:= \E  
\left[ 
H_{\otimes \mathbf a} \{\check Y(\mathbf t)\}H_{\otimes \mathbf b} \{\check Y(\mathbf s)\} 
\right]  \\
\gamma_{a,\mathbf b} (t,\mathbf s ) = \gamma_{a,\mathbf b} (t\ominus \mathbf s) &:= \E  
\left[ 
H_{\otimes a} \{\check Y(t)\}H_{\otimes \mathbf b} \{\check Y(\mathbf s)\} 
\right] . 
\end{align*}  
Let $q\ge 1$ and $a,b,\mathbf a ,\mathbf b$ be such that
$|a|=|b|=|\mathbf a|=|\mathbf b|=q$. Under \ref{C:integrability}, we
have the following upper bounds for the covariance functions for some $\kappa>0$.
\begin{align}
|\gamma_{a,b} (t,s )| &\le \kappa \, \|\Sigma(0)^{-1/2}\|_{\max}^{2q}  \; \Xi(\|t-s\|)^q\label{eq:gammaab1} \\
|\gamma_{\mathbf a,\mathbf b} (\mathbf t,\mathbf s )| &\leq \kappa \,
\|\bar\Sigma(\|t^{(1)}-t^{(2)}\|)^{-1/2}\|_{\max}^{q} 
\|\bar\Sigma(\|s^{(1)}-s^{(2)}\|)^{-1/2}\|_{\max}^{q} \;
\Xi\{\delta(\mathbf{t}, \mathbf{s})\}^q \nonumber\\
\text{ with } & \delta(\mathbf{t}, \mathbf{s})=\min\left\{
\|t^{(1)}-s^{(1)}\|, \|t^{(1)} - s^{(2)}\|, \|t^{(2)} - s^{(1)}\|,\| t^{(2)} - s^{(2)}\|
\right\} \label{eq:gammaab2}\\
|\gamma_{a,\mathbf b} (t,\mathbf s )| 
&\le \kappa \,
\|\Sigma(0)^{-1/2}\|_{\max}^{q} 
\|\bar\Sigma(\|s^{(1)}-s^{(2)}\|)^{-1/2}\|_{\max}^{q} \;
\Xi\{\delta(t, \mathbf{s})\}^q \nonumber\\
\text{ with } & \delta(t, \mathbf{s})=\min\left\{
\|t-s^{(1)}\|, \|t - s^{(2)}\|
 \right\}. \label{eq:gammaab3}
\end{align} 
\end{lemma}

Lemma~\ref{lem:gamma} is fundamental for the proofs of Theorems~\ref{thm:variance}-\ref{thm:clt}. For $t,s \in \mathbb R^d$ or $\mathbf t,\mathbf s \in \mathbb R^{2d}$, we observe that the upper bounds are composed of a variance factor  ($\Sigma(0)$ and $\bar\Sigma(\|u^{(1)}-u^{(2)}\|))$ (for $u=t,s$) and a factor involving covariances at $t$ and $s$ (resp.\ at $\mathbf t$ and $\mathbf t$, at $t$ and $\mathbf s$).

We denote for any measurable function $\phi\in L^1(\R^d)$ by $\hat \phi$ the Fourier transform of $\phi$ defined by $\widehat \phi(\omega)=\int_{\R^d} \phi(t) \exp(-\mathrm i t^\top \omega)\dd \omega$. Asymptotic variances are considered in the following theorem.
\begin{theorem} \label{thm:variance}
The two following statements hold.\\
(i) Assume \ref{C:general}$[3]$, \ref{C:nondegeneracy}[2],
\ref{C:integrability}, and 
let $\phi_{1,n}$ be a sequence of bounded piecewise continuous functions. Assume there exists a bounded measurable function $\phi_1\in L^1(\R^d)$ such that as $n\to \infty$
\begin{equation}\label{eq:assphi1hat}
\;\widehat{\phi}_{1,n,W_n} \left( \frac{\omega}n\right) \to \hat \phi_1(\omega), \ \text{ or equivalently, }\ 
n^d \phi_{1,n,W_n}(nt) \to \phi_1(t)
\end{equation}
for any $\omega,t\in\R^d$ where $\phi_{1,n,W_n}(\cdot) = \phi_{1,n}(\cdot)\, \mathbf{1}(\cdot \in W_n)$.
Then, $\phi_1$ has compact support in $W_1$ and as $n\to \infty$,
\begin{equation} \label{eq:varPhi1}
n^{d} \; \Var\left( \cent{\Phi}_{1,n} \right) \rightarrow \mathcal V(\phi_1)= \|\phi_1\|^2_{L^2(\R^d)}  \times   \sum_{q \in  \mathbb N_\L^*} \sum_{\substack{ a,b\in \mathbb N^{D}, \\ |a|=q,|b|=q }} d_{a} d_{b} \widehat{\gamma}_{a,b}(0).
\end{equation}

\noindent(ii) Assume \ref{C:general}[5], \ref{C:nondegeneracy}[4], 
\ref{C:integrability}, \ref{C:nondegeneracy:hessian}, and let $\phi_{2,n}$ be the sequence of functions given by
\begin{equation} \label{eq:assphi2n}
\phi_{2,n}(\mathbf{t})  = \frac{\phi_2\{\|t^{(1)} - t^{(2)}\|\}}{|W_n \cap (W_n)_{t^{(1)} - t^{(2)}}|}
\end{equation}
where $\phi_2$ is a bounded and piecewise continuous function, compactly supported on $[\eta,R]$. First for any $\mathbf a,\mathbf b \in \mathbb{N}^{2D}$ with $|\mathbf a|=|\mathbf b|=q \geq 1$,
\begin{equation}\label{eq:integK2}
w \mapsto \sup_{u,v \in B(0,\eta,R)} 
\gamma_{\mathbf a, \mathbf b}
\{(w+v,w+v-u)\ominus (v,0)\}
\in L^1(\R^d).
\end{equation}
Second,  $d_{\mathbf a}(\cdot)$ is continuous (and so bounded) on $[\eta,R]$ and
\begin{align}
n^{d} \; \Var\left( \cent{\Phi}_{2,n} \right) &\rightarrow  \mathcal V(\phi_2)\quad \text{ with } \nonumber\\
\mathcal V(\phi_2)=&    \sum_{q \in \mathbb N_{\L}^*} \sum_{\substack{ \mathbf a,\mathbf b\in \mathbb N^{2D}, \\ |\mathbf a|=q,|\mathbf b|=q }} 
\int_{\R^d} \int_{B(0,\eta,R)} \int_{B(0,\eta,R)} \!\!\! \phi_2(\|u\|)d_{\mathbf a}(\|u\|) \phi_2(\|v\|) d_{\mathbf b}(\|v\|) \times \nonumber \\
&\hspace*{1cm} 
\gamma_{\mathbf a,\mathbf b}
\{(w+v,w+v-u)\ominus (v,0)\}
\dd u \dd v \dd w.\label{eq:varPhi2}
\end{align}
(iii) Consider the assumptions made in~(i)-(ii). First, for any $a\in \mathbb N^D$ and $\mathbf b \in \mathbb N^{2D}$ with $|a|=|\mathbf b|=q\ge 1$,
\begin{equation} \label{eq:integCov}
u \mapsto \sup_{v \in B(0,\eta,R)} \, \gamma_{a,\mathbf b} 
\left\{
(u+v) \ominus (v,0)
\right\} \in L^1(\R^d).
\end{equation}
Second, as $n \to \infty$
\begin{align}
n^d &\, \Cov(\cent{\Phi}_{1,n}, \cent{\Phi}_{2,n}) \to \mathcal C(\phi_1,\phi_2)  \qquad \text{ with }\nonumber\\
\mathcal C(\phi_1,\phi_2) = \mathcal{C}(\phi_2,\phi_1) &= \sum_{q\ge 1} \sum_{\substack{a\in \mathbb N^D, \mathbf b \in \mathbb N^{2D} \\ |a|=|\mathbf b|=q }} d_a \phi_1(0) \nonumber\\
& \  \times \int_{\R^d} \int_{B(0,\eta,R)}\phi_2(\|v\|) d_{\mathbf b}(\|v\|) \gamma_{a,\mathbf b} 
\left\{ (u+v) \ominus (v,0)\right\} \dd v \dd u. \label{eq:covPhi1Phi2}
\end{align}
(iv) Consider the assumptions made in~(i)-(ii). Let
$\phi_{1,n}^1,\dots,\phi_{1,n}^{m_1}$ (resp.\
$\phi_{2,n}^1,\dots,\phi_{2,n}^{m_2}$) be functions
satisfying~\eqref{eq:assphi1hat} (resp.\
\eqref{eq:assphi2n}). Consider centered linear and bilinear statistics given by these functions, denoted by $\cent{\Phi}_{1,n}^i$ and $\cent{\Phi}_{2,n}^j$ for $i=1,\dots,m_1$ and $j=1,\dots,m_2$. Let $\beta =(\beta_1^\top,\beta_2^\top)^\top \in \R^{m}$ with $m=m_1+m_2$ and 
\begin{align*}
\phi&= (\phi_1^\top, \phi_2^\top)^\top = \left(
\phi_{1}^1, \dots,\phi_{1}^{m_1}, 
\phi_{2}^1, \dots,\phi_{2}^{m_2}
\right)^\top \\
\cent{\Phi}_n&= \left\{
\cent{\Phi}_{1,n}^1, \dots,\cent{\Phi}_{1,n}^{m_1}, 
\cent{\Phi}_{2,n}^1, \dots,\cent{\Phi}_{2,n}^{m_2}
\right\}^\top.
\end{align*}
Then, as $n\to \infty$
\begin{equation} \label{eq:linearcombination}
n^d \Var \left( \beta^\top \cent{\Phi}_n \right) \to 
\mathcal{V}(\beta_1^\top \phi_1) + 
2 \, \mathcal{C}(\beta_1^\top \phi_1, \beta_2^\top \phi_2) +
\mathcal{V}(\beta_2^\top \phi_2).
\end{equation}
\end{theorem}
A proof of Theorem~\ref{thm:variance} is provided in Appendix~\ref{app:var}. Note that the function $\phi_{1,n}(t)=n^{-d}\phi_1(t/n)$ with $\phi_1$ piecewise continuous and compactly supported on $W_1=[-1/2,1/2]^d$ trivially satisfies the assumption~\eqref{eq:assphi1hat}. Note also that~\eqref{eq:varPhi1} means that as $n\to \infty$,
\begin{equation}\label{eq:ratiovar}
    \Var( \cent{\Phi}_{1,n} )/\Var\{N_\L(W_n)\} \to \|\phi_1\|^2_{L^2(\R^d)}.
\end{equation} 
This result is actually true for any stationary spatial point process with pair correlation function $g$ such that $g-1 \in L^1(\R^d)$, see Appendix~\ref{app:var}. 
Regarding (ii), the sequence $\{\phi_{2,n}\}_{n\ge 1}$ corresponding to the modified translation-corrected (normalized) estimator of  Ripley's $K$-function at distance  $r>0$, defined by~\eqref{eq:exPhi2}, satisfies the assumption~\eqref{eq:assphi2n} with $\phi_2(x)=\mathbf 1(\eta \le x \leq r)$.

We now present our last result which in particular provides the asymptotic distribution of $\hat \rho_\L$ and  finite-dimensional distributions of the modified Ripley's $K$-function estimator $\hat K_{\eta,\L}$.

\begin{theorem} \label{thm:clt} 
Assume \ref{C:general}[5], \ref{C:nondegeneracy}[4], 
\ref{C:integrability}, and \ref{C:nondegeneracy:hessian}. The two following statements hold. \\
(i) Let $\phi_{1,n}$ and $\phi_{2,n}$ be two sequences of bounded
piecewise continuous functions  satisfying respectively \eqref{eq:assphi1hat} and \eqref{eq:assphi2n}. Denoting $\cent{\Phi}_n=(\cent{\Phi}_{1,n},\cent{\Phi}_{2,n})^\top$, we have $n^{d/2} \cent{\Phi}_n \stackrel{\mathcal{D}}{\rightarrow} \mathcal{N}\{ 0, \Sigma(\phi_1,\phi_2) \}$ as $n \to \infty$, where
\begin{equation} \label{eq:tclPhi}
    \Sigma(\phi_1,\phi_2) = 
\begin{pmatrix}
\mathcal V(\phi_1) & \mathcal C(\phi_1,\phi_2) \\
\mathcal C(\phi_1,\phi_2) & \mathcal V(\phi_2)
\end{pmatrix}
\end{equation}
with $\mathcal V(\phi_1), \mathcal V(\phi_2), \mathcal
C(\phi_1,\phi_2)$  given respectively by~\eqref{eq:varPhi1},~\eqref{eq:varPhi2}, and~\eqref{eq:covPhi1Phi2}.\\
(ii) Let $m\ge 1$ and $\eta \leq r_1<\dots<r_m\le R$. Define
\begin{align*}
\zeta_\L:=&    \left\{\hat \rho_\L -\rho_\L , 
\hat K_{\eta,\L}(r_1) - K_{\eta,\L}(r_1),\dots,
\hat K_{\eta,\L}(r_m) - K_{\eta,\L}(r_m)
\right\}^\top  
\end{align*}
where the $\hat K_{\eta,\L}(r_i)$'s are defined according to \eqref{eq:def:K:hat}.
Then, there exist a symmetric nonnegative definite matrix $\Sigma(\rho_\L,K_{\eta,\L})$ (see Appendix~\ref{app:CLTii} for details) such that as $n \to \infty$,
\begin{equation} \label{eq:TCLrhoK}
n^{d/2} \zeta_\L \stackrel{\mathcal{D}}{\longrightarrow} \mathcal N \left\{0, \Sigma(\rho_\L,K_{\eta,\L})  \right\}.
\end{equation}
\end{theorem}
A proof of this result is provided in Appendix~\ref{app:CLT}. We
emphasize that in (i), if we are only interested in the asymptotic
normality of $n^{d/2}\cent{\Phi}_{1,n}$ then it suffices to assume~\ref{C:general}$[3]$, \ref{C:nondegeneracy}[2], 
and \ref{C:integrability}.

For ease of exposition, we have considered bilinear functionals of the
form~\eqref{eq:Phi2} where $t^{(1)}$ and $t^{(2)}$ are elements of
$\bY_\L$. However, results similar to 
Theorems~\ref{thm:variance}-\ref{thm:clt} could be proved for bilinear
statistics depending on different subsets of indices, i.e. when
$t^{(1)}\in \L$ and $t^{(2)}\in \L^\prime$ for two different subsets
of indices of $\{0,1,\dots,d\}$. This would require much additional
notation. In particular, in Theorem~\ref{thm:variance}(iv),
$\cent{\Phi}_{1,n}^{j}$ (resp.\ $\cent{\Phi}_{2,n}^{j}$) should also
depend on a set of indices $\L_1^{j}$ (resp.\ on two sets of indices
$\L_2^{j,1}$, $\L_2^{j,2}$). The main result, however,  remains true, and in particular for any $\L,\L^\prime\subseteq \{0,1,\dots,d\}$, the random vector defined by
$$\zeta_{\L,\L^\prime}:=   \left\{
\hat \rho_\L -\rho_\L , 
\hat \rho_{\L^\prime} -\rho_{\L^\prime} , 
\hat K_{\eta,\L,\L^\prime}(r_1) - K_{\eta,\L,\L^\prime}(r_1),\dots,
\hat K_{\eta,\L,\L^\prime}(r_m) - K_{\eta,\L,\L^\prime}(r_m)
\right\}^\top$$
satisfies as $n\to \infty$, $n^{d/2} \zeta_{\L,\L^\prime} \stackrel{\mathcal{D}}{\longrightarrow} \mathcal N \left\{0, \Sigma(\rho_\L,\rho_\L^\prime,K_{\eta,\L,\L^\prime}) \right\}$ for some nonnegative definite matrix $\Sigma(\rho_\L,\rho_\L^\prime,K_{\eta,\L,\L^\prime})$.

\section{Conclusion and perspectives} 
\label{sec:conclusion}

This paper lays out a foundation for the use of critical point
processes in statistical applications. However, some interesting
problems remain for further research.  First, the simulation problem
for a critical point process is not completely solved since we need to
understand better how to identify the critical points of a simulated
random field and to assess how far an approximate simulation differs
in distribution from the target critical point process. Second, the
parameter $\eta>0$ introduced to study bilinear functionals
$\Phi_{2,n}$ is a technical and difficult nuisance.

Third, considering Theorems~\ref{thm:variance}-\ref{thm:clt}, it
remains to identify conditions ensuring that the limiting variances are
positive or positive definite. According to~\eqref{eq:ratiovar},
positivity of the limiting variance of a linear statistic is
equivalent to positivity of the asymptotic variance of $n^{-d}
\Var\{N_\L(W_n)\}$. To the best of our knowledge, the work by
\cite{nicolaescu2017clt} is the only one showing that this variance
is positive when $\L=\{0,1,\dots,d\}$ and for any dimension $d\ge
1$. Positivity of the asymptotic variance for bilinear statistics is a
completely open question.

Fourth,  parameter estimation is a generic statistical problem. In
other words, can we estimate for example the parameter vector
$\theta=(\nu,\phi)$ for a Matérn random field, based only on an
observation of $\bY_\L$ within a bounded window. For given $\theta$, from
Section~\ref{sec:intensity}, we have expressions of the intensity parameter $\rho_\L(\theta)$
and, up to a Monte-Carlo approximation, of the modified $K$-function $K_{\eta,\L}(\,\cdot\, ;
\theta)$. Therefore, minimum contrast estimation methods
\cite[see e.g.][]{waagepetersen2009two} could be applied for
parameter estimation, e.g.\ by minimizing
    \[
\Big\{ \hat \rho_\L - \rho_\L(\theta)\Big\}^2  +
\int_{r_1}^{r_2} \Big\{ \hat K_{\eta,\L}(r) - K_{\eta,\L}(r;\theta)\Big\}^2 \dd r
\]
where $r_1,r \ge \eta$ are hyperparameters and  $\hat \rho_\L$ and $\hat{K}_{\eta,\L}$ are non-parametric estimates of $\rho_\L$ and $K_{\eta,\L}$. Asymptotic
results for the minimum contrast estimates could be obtained using
Theorems~\ref{thm:clt} in combination with a standard central limit
theorem for the Monte Carlo approximation of $K_{\eta,\L}(\cdot;\theta)$.

Fifth, this paper opens the door to investigate  several other
probabilistic properties. For example, deriving an expression of the
void probabilities or the moment generating functional and characterizing  dependence in terms of mixing properties are open (and probably difficult) questions.

Sixth and final, this paper concerns critical point processes obtained
from stationary and isotropic Gaussian
random fields. It would be of great practical interest to introduce
inhomogeneity, for example using spatial covariates such as precipitation or
wind speed if the aim is to model lightning strike impacts
as a critical point process. It is not obvious how to include such
covariates in the model. For instance, including the covariates in a
mean function for the Gaussian random field may not suffice.
  We leave this important question for future research.

\section*{Acknowledgements} 

The research of J.\ Chevallier and J.-F.\ Coeurjolly is supported by
Labex PERSYVAL-lab ANR-11-LABX-0025. The research of J.C. is supported
by ANR-19-CE40-0024 (CHAllenges in MAthematical NEuroscience). R.\
Waagepetersen was supported by grants VIL57389, Villum Fonden, and
NNF23OC0084252, Novo Nordisk Foundation. This research has been conducted while J.\ Chevallier was in Statify team at Centre Inria de l'Université Grenoble Alpes. The authors would like to thank Anne-Lise Porté who initiated this work.

\appendix


\section{Regularity of Gaussian random fields} \label{app:regularity}

We have the following diagram, where $n$ is any natural integer and $\varepsilon \ge 0$s.
$$
\begin{tikzcd}
    &
    c \in C^{2n+2\varepsilon}(\R^d) \arrow[Leftrightarrow]{d} & 
    \\
    \lambda_{2n} < \infty \arrow[Leftarrow]{r}\arrow[Rightarrow, bend right=10, "\varepsilon=0" below]{r} &
    c_{1} =c_2(\cdot^2)\in C^{2n+2\varepsilon}(\R^+) \arrow[Leftrightarrow]{d} &
  \,   \bX \in C_{L^{2}}^{n+\varepsilon}(\R^d) 
    \arrow[Leftrightarrow]{l} \arrow[Rightarrow, "\varepsilon > \varepsilon^- >0"]{d}\\
    &
    c_{2} \in C^{n+\varepsilon}(\R^+) &
\bX \in C_{\rm a.s.}^{n+\varepsilon^-}(\R^d)
\end{tikzcd}
$$
The regularity in terms of covariance function $c$ (respectively $c_1$ and $c_2$) can be restricted to local regularity around $0\in \mathbb{R}^d$ (respectively $0\in \mathbb{R}$) since this local regularity automatically extends to bounded domains~\cite[Remark 8]{da2023sample}.

The equivalence of the statements regarding $c$ and $c_1$ is a direct consequence of isotropy. The equivalence regarding $c_1$ and $c_2$ is related to Tauberian theory and the paper by~\cite{riedi2015strong}. The equivalence regarding $c_1$ and the regularity of $\bX$ in quadratic mean follows from the definition of this regularity. The implication towards almost sure regularity corresponds to~\cite{da2023sample}[Theorem~7].


\section{Additional figure, behavior of the pair correlation function at 0 and measure of repulsion} \label{app:localpcf}

\begin{figure}[htbp]
\centering
\includegraphics[width=.98\textwidth]{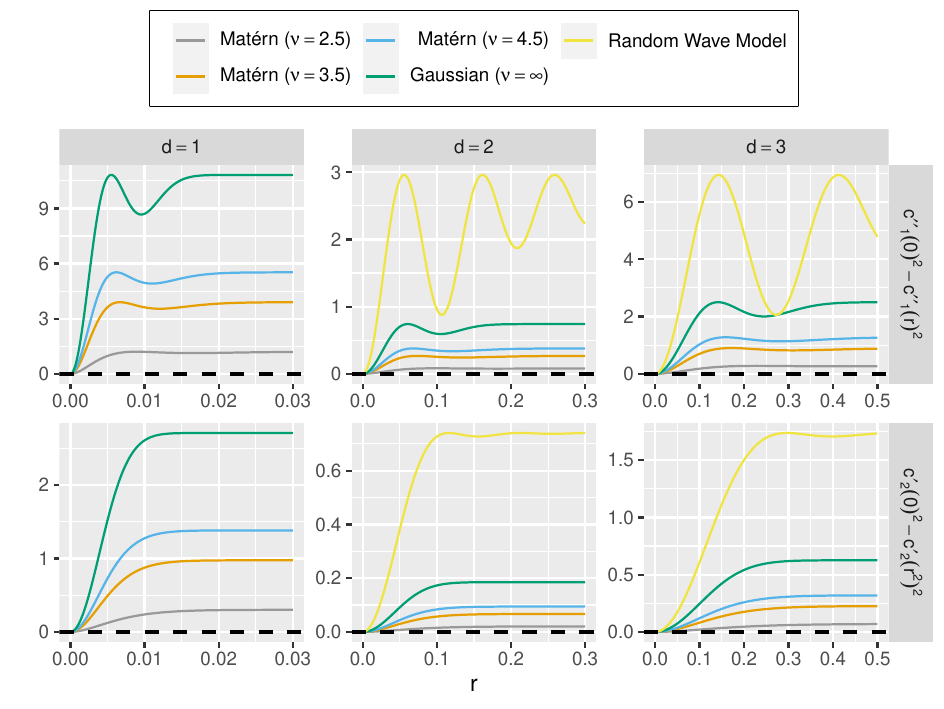}
\caption{For $d=1,2,3$, plots of $c_1^{\prime\prime}(0)^2 - c_1^{\prime\prime}(r)^2$ and $c_2^\prime(0)^2- c_2^\prime(r^2)^2$ in terms of $r$ for Matérn correlation functions with $\nu=2.5,3.5,4.5,\infty$ and for the Random Wave Model (for $d>1$). The scale parameter $\varphi$ is set such that $\rho_\L=100$ (for $\L=\{0,\dots,d\}$). Values have been multiplied by $0.01,10,1000$ for $d=1,2,3$ for clarity of the $y$-axis.
The curves are above $y=0$ except at $r=0$ which shows that $V(r)= \{X^\prime(0)^\top,X^\prime(re_1)^\top\}^\top$ is non-degenerate for any $r>0$. Similar behaviors are observed for different values of $\rho_\L$.} \label{fig:plotc1c2}   
\end{figure}

In this section we investigate the behavior of $g_\L(r)$ as $r\to 0$ given in Table~\ref{tab:expansion_pcf} in the particular cases $\L=\{0,\dots,d\}$ (all critical points) and $\L=\{d\}$ (local maxima). Recall that for $r \rightarrow 0$, $g_\L(r)\sim c r^{2-d}$ when $\L=\{0,1,\dots,d\}$ under \ref{C:general}[3]
and $g_\L(r)=O(r^{5-d-\varepsilon})$ when $\L=\{d\}$ under \ref{C:general}[4].
We depict $\log g_\L(r)$ in terms of $\log(r)$ for small values of $r$ and estimate the slope via a simple linear regression. Figure~\ref{fig:logAll} illustrates this analysis for Matérn random fields for  $\rho_\L=100$ and for $\nu=2.5,3.5,4.5,\infty$ (the latter corresponding to the Gaussian correlation). Other values of $\rho_\L$ yield similar results. According to Table~\ref{tab:expansion_pcf}, we choose $\nu=2.5$ (resp $\nu=3.5$) to break assumption \ref{C:general}[3] (resp.\ \ref{C:general}[4]).

Considering the Mat{\'e}rn model and both cases of $\L$, the main message from Figure 4 is that the corresponding critical point process becomes more repulsive for increasing $\nu$ and decreasing $d$. It can, however, be clustered if $\nu$ is small enough or $d$ is large enough. The fitted slopes agree well with the asymptotic slopes $2-d$ or $5-d$ for large $\nu$. For RWM, smaller $d$ also gives more repulsion and the fitted slopes fit quite well with $2-d$ or $5-d$.

Another measure of repulsion, also partly proposed by~\cite{biscio2016quantifying}, compares the pair correlation function with the nominal value 1. We suggest to define the repulsion function index $I_\L(r)$ for $r>0$ by
\begin{equation}\label{eq:IL}
I_\L(r) = 1+ \rho_\L^{-1}\int_{B(0,r)} \{g_\L(\|t\|)-1\}\dd t = 1+ \frac{2\pi^{d/2}}{\rho_\L\Gamma(d/2)}\int_0^r z^{d-1}\{ g_\L(z)-1\}\dd z.
\end{equation} 
To rephrase this functional index, a value of $I_\L(r)$ smaller than 1, means that $g_\L$ is less than 1 for a wide range of distances $0 \le z \le r$, so the related point process is quite repulsive at distances between 0 and $r$. The functional index $I_\L(r)$ is also related to the asymptotic variance of the number of critical points. Indeed,  it can be shown (see Section~\ref{sec:statistics}) that as $n\to \infty$,
$n^d \Var\{N_\L(W_n)\} \to \mathcal V_\L(\phi_1)$ where 
\[
    \mathcal V_\L(\phi_1) = \rho_\L - \rho_\L^2 \int_{\R^d} \{g_\L(\|t\|)-1\}\dd t = \rho_\L \times  \lim_{r\to \infty} I_\L(r).
\]
Finally, $I_\L(r)$ is also closely related to Ripley's $K$-function for $\bY_\L$, since 
\[
  I_\L(r)= 1+ \rho_\L^{-1} \left\{ |B(0,r)| - \int_{B(0,r)} g_\L(\|t\|) \dd t \right\} =   
  1+ \rho_\L^{-1} \left\{ \frac{\pi^{d/2} r^d}{\Gamma(d/2+1)} - K_\L(r) \right\}.
\]
Empirical results are depicted in Figure~\ref{fig:plotI} for $\L=\{0,1,\dots,d\}$ and $\L=\{d\}$  obtained from Matérn correlations with $\nu=2.5,3.5,4.5,\infty$ and from RWM, $\rho_\L=100$ and $d=1,2,3$. The conclusion is quite similar to the one obtained from Figure~\ref{fig:logAll}. In case of the Mat{\'e}rn model, $I_\L(r)$ is smaller (more repulsion) for larger values of $\nu$ and smaller values of $d$. For RWM, $|g_{\L}-1|$ is not integrable over $\R^d$ (this is related to the violation of condition~\ref{C:integrability}). This is reflected by the oscillatory behaviour of $I_\L(r)$ as a function of $r$.

\begin{figure}[!htbp]
\centering
\begin{tabular}{c}
\includegraphics[width=.9\textwidth]{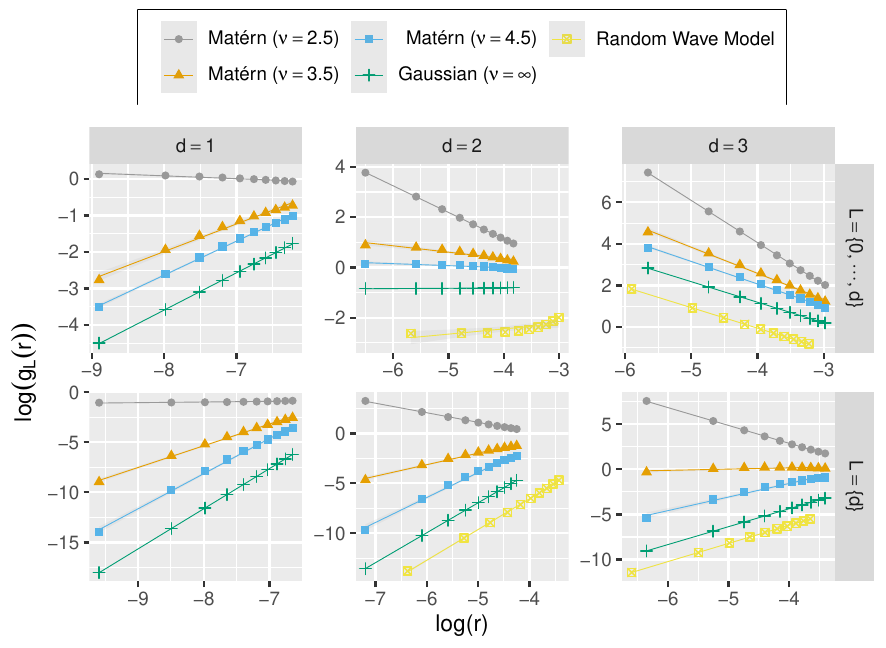}\\[-20pt]
\includegraphics[width=.7\textwidth]{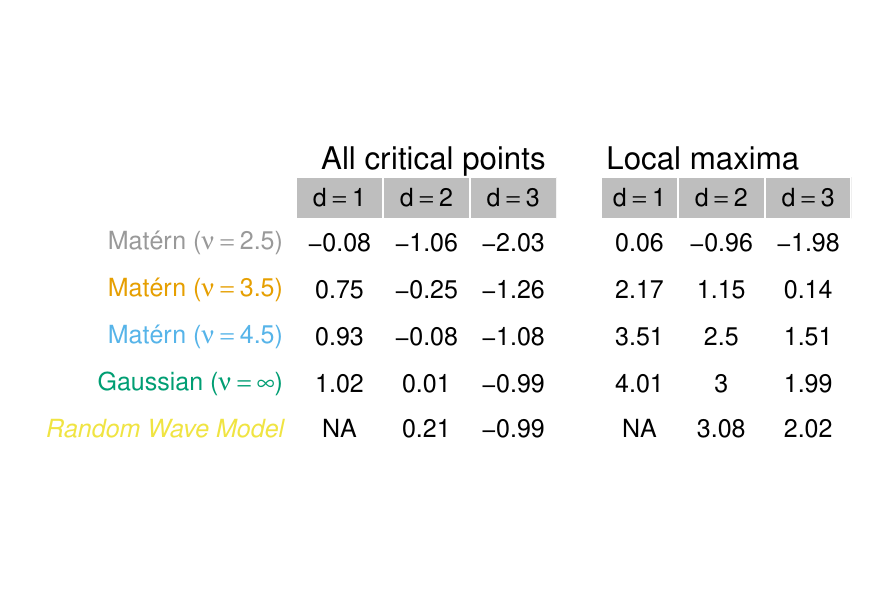} 
\end{tabular}
\caption{For $d=1,2,3$, plots of $\log g_\L(r)$ in terms of $\log(r)$ for small values of $r$ for Matérn correlation functions with regularity parameter $\nu=2.5,3.5,4.5,\infty$ and Random Wave Model (for $d>1$). The scale parameter $\varphi$ is set such that $\rho_\L=100$. The upper plot (resp.\ lower plot) focuses on all critical points (resp.\ local maxima). In addition, we present the fitted linear regressions. The slopes values are summarized in the bottom table. The pair correlation function $g_\L(r)$ is obtained using Monte Carlo estimation (with $B=10^7$ replications, see Section~\ref{sec:intensity} for more details).} \label{fig:logAll}
\end{figure}

\begin{figure}[!htbp]
\centering
\includegraphics[width=\textwidth]{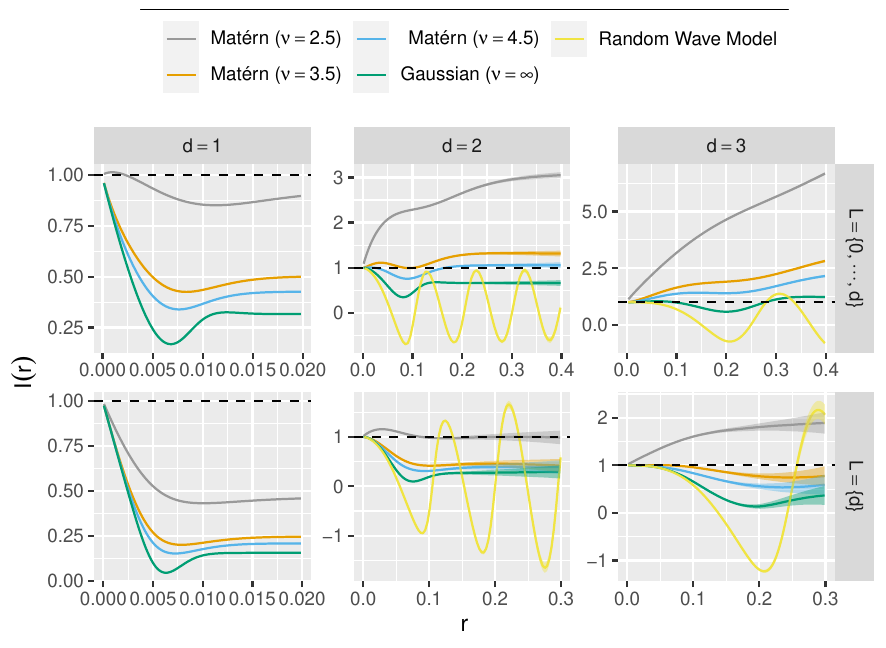}
\caption{For $d=1,2,3$, plots of Monte Carlo estimates of the function $I_\L(r)=1+\rho_\L^{-1}\int_{B(0,r)} \{g_\L(\|t\|)-1\}\dd t$ for $\L=\{0,1,\dots,d\}$ and $\L=\{d\}$, obtained from Matérn correlation functions with regularity parameter $\nu=2.5,3.5,4.5,\infty$ and Random Wave Model (for $d>1$). The scale parameter $\varphi$ is set such that $\rho_\L=100$.} \label{fig:plotI}
\end{figure}


\section{Kac-Rice formulae and Proofs of Theorems~\ref{thm:pcf}-\ref{thm:rhok}}

For sake of completeness we present Kac-Rice formulae, well-established in the literature ({\cite[Theorem~6.2]{azais2009level},\cite[Theorem~6.4]{azais2009level}), and apply them to obtain  explicit expressions of pair correlation functions $g_\L$, cross-pair correlation functions $g_{\L,\L^\prime}$ and $k$th order indensity functions of $\bY_\L$. 

\subsection{Kac-Rice formulae}

The two following results are classical Kac-Rice formulae \cite[see e.g.\ ][]{azais2009level} stated here for the sake of self-containedness of the present manuscript. 

\begin{theorem}[{\cite[Theorem~6.2]{azais2009level}}]\label{thm:rice1}
Let $Z : U \rightarrow \mathbb{R}^{d}$ be a random field defined on $U$ an open set of $\R^d$ and let $u\in \R^d$. Assume  (i) $Z$ is Gaussian; (ii) the map $t \mapsto Z(t)$ is almost surely $C^{1}(U,\R^d)$; (iii) for any $t\in U$, the distribution of $Z(t)$ is non-degenerate; (iv) $\mathrm P\left[  \exists t\in U: Z(t)=0, \det\{Z^\prime(t)\}=0 \right]=0$. Then, for any Borel $B \subseteq U$
\begin{equation}
\E\{N_{u}(Z,B)\} = \int_{B} \E\left[ \vert \det \{Z'(t)\} \, \vert \big\vert Z(t) = u \right] f_{Z(t)}(u) \dd t
\end{equation}
where $N_u(Z,B)= \#\{t\in B: Z(t)=u\}$ and where $f_{Z(t)}(u)$ is the density of $Z(t)$ at $u$.
\end{theorem}

\begin{theorem}[{\cite[Theorem~6.4]{azais2009level}}]\label{thm:rice2}
Let $Z$ be a random field satisfying the assumptions of Theorem~\ref{thm:rice1}. Assume for any $t\in U$, there exists another random field $Y^{t}:D \rightarrow \mathbb{R}^{d_0}$ with $D$ any topological space such that: (a) $(\omega,t,w) \rightarrow \{Y^{t}(\omega)\}(w)$ is  measurable and almost surely $(t,w) \mapsto Y^{t}(w)$ is continuous; (b) for any $t \in U$, the random process $(s,w) \mapsto \{Z(s), Y^{t}(w)\}$ defined on $U \times D$ is Gaussian. If in addition $J : U \times C^0(D,\mathbb{R}^{d_0}) \rightarrow \mathbb{R}$ is a bounded and continuous function  for the topology of the uniform convergence on every compact on $C^0(D,\mathbb{R}^{d_0})$, then for every compact subset $I\subset U$ and $u \in \R^d$ we have
\begin{equation}
\E \left\{ \sum_{t \in I,\\ Z(t)=u} J(t, Y^{t}) \right\} = 
\int_{I} \E\left[
\vert \det\{Z'(t)\} \vert \, J(t, Y^{t})  \; \big\vert Z(t)=u
\right]
f_{Z(t)}(u) \dd t.
\end{equation}
\end{theorem}


\subsection{Proof of Theorem~\ref{thm:pcf}}
\label{sec:proofPcf}

\begin{proof}
The proof follows closely the proof of \cite[Proposition 19]{azais2022mean}.

\noindent(i) From \cite[Appendix B.2.1]{azais2022mean} we have
\begin{align*}
\Var\{ X^\prime(0)\} &= \Var\{X^\prime (r e_{1})\} = \lambda_{2} I_{d} = -2c_2^\prime(0) I_d,\\
\E \{ X^\prime(0) X^\prime(r e_{1})^\top\} &= \mathrm{diag}\left\{ -2c_2^\prime(r^2) - 4r^{2} c_{2}^{\prime\prime}(r^{2}), -2c_2^\prime(r^2),\dots, -2c_2^\prime(r^2) \right\}.
\end{align*}
By rearranging the terms of $V(r)$, one obtains that $\Var\{V(r)\}$ is a block diagonal matrix given by $\mathrm{diag}(B_{1}, B_{2}, B_{2},\dots,B_{2})$ where
\begin{equation}\label{eq:B1B2}
B_{1} = \begin{pmatrix} -2c_2^\prime(0) & -2c_2^\prime(r^2) - 4r^2 c_2^{\prime\prime}(r^2) \\ -2c_2^\prime(r^2) - 4r^{2} c_{2}^{\prime\prime}(r^2) & -2c_2^\prime(0) \end{pmatrix},\; 
B_{2} = \begin{pmatrix} -2c_2^\prime(0) & -2c_2^\prime(r^2) \\ -2c_2^\prime(r^2) & -2c_2^\prime(0) \end{pmatrix}.    
\end{equation}
In particular,~\ref{C:nondegeneracy:gradient}$[r]$ is satisfied if and only if both matrices have positive determinant.
Using the fact that $c_1(r) = c_2(r^2)$, one gets $c_1^{\prime\prime}(r) = 4r^2c_2^{\prime\prime}(r^2) + 2c_2^{\prime}(r^2)$ so that~\eqref{eq:conditionVr}-\eqref{eq:density:Vr} follow easily.

\noindent (ii) Let $r>0$ be such that condition~\ref{C:nondegeneracy:gradient}$[r]$ is satisfied. In particular, \eqref{eq:conditionVr} is satisfied. By continuity of $c_{2}^{\prime}$ and $c_{1}^{\prime\prime}$, there exists $\varepsilon >0$ such that for all ${t}=(t_1,t_2)\in B(0,\varepsilon)\times B(r e_{1}, \varepsilon)$, the random vector $Z({t}) = \left\{ X^\prime(t_1)^\top,X^\prime(t_2)^\top\right\}^\top$ is non-degenerate.  Let $S_1 = B(0,\varepsilon)$ and $S_2 = B(re_{1},\varepsilon)$. In particular, $S_1 \cap S_2 = \emptyset$ {for $\varepsilon$ small enough}. 

Let $\mathcal{L},\mathcal{L}^\prime \subseteq \{0,1,\dots,d\}$. We apply Theorem~\ref{thm:rice2} to  $Z(t)$ with $U = S_1\times S_2$, $u=(0,0)$ and the trivial situation $D=\{0\}\subset \R^{d(d+1)}$. In the following, we identify $\mathbb{R}^{d(d+1)}$ as the product of the vectorized spaces $\mathbb{R}^{d(d+1)/2}\times \mathbb{R}^{d(d+1)/2}$ and, for any $f\in C^0(D, \mathbb{R}^{d(d+1)})$ we write $f=(f_1,f_2)$ for $f_1, f_2 \in C^0(D, \mathbb{R}^{d(d+1)/2})$. Further, for all ${t} \in U$, we define $Y^{{t}}(0)$ as the Gaussian vector $\{X^{\prime\prime}(t_1)^\top, X^{\prime\prime}(t_2)^\top\}^\top$ and we define $J: U \times C^0 (D, \mathbb{R}^{d(d+1)})  \rightarrow  \mathbb{R}$  by $({t},f)  \mapsto  
\iota_\L \{f_1(0)\} \times \iota_{\L^\prime} \{f_2(0)\}$.

We now check the assumptions of Theorem~\ref{thm:rice2} (which in particular include the assumptions of Theorem~\ref{thm:rice1}). First, conditions~(i)-(iii) of Theorem~\ref{thm:rice1} are satisfied by the assumptions of Theorem~\ref{thm:pcf}. Second, since $Z^\prime(t) = \mathrm{diag}\{X^{\prime\prime}(t_1), X^{\prime\prime}(t_2)\}$ condition~(iv) becomes
\begin{equation*}
    \mathbb{P}\left[ \exists {t} \in S_1\times S_2 : X^\prime(t_1) = X^\prime(t_2) = 0, \det\{X^{\prime\prime}(t_1)\} \times \det\{X^{\prime\prime}(t_2)\} = 0 \right] = 0.
\end{equation*}
Obviously, that probability is less than $\mathbb{P}\left[ \exists t_1 \in S_1 : X^\prime(t_1) = 0, \det\{X^{\prime\prime}(t_1)\}  = 0 \right]$ which is null due to the fact that the sample paths are a.s. Morse.
Third, assumptions a) and b) of Theorem~\ref{thm:rice2} are
satisfied. Since $J\in \{0,1\}$, the assumption that it is a continuous function is not
satisfied. To resolve this, we consider a sequence of increasing, bounded and continuous functions $(J_m)_{m\ge 1}$  which converges pointwise to $J$. We can, for example, take $J_m(x)=\min \{ 1,m\times d(x,F)\}$ with $F=J^{-1}(\{0\})$. Applying Theorem~\ref{thm:rice2} to $J_m$ yields that for every compact $I$
\begin{equation*}
    \E\left\{ \sum_{{t}\in I, Z({t})=0} J_m({t},Y^{{t}}) \right\} = 
    \int_{I} \E \left[ \vert \det \{Z^\prime(t)\} \vert J_m({t},Y^{{t}}) \big\vert Z({t})= (0,0) \right] f_{Z({t})}(0,0) \dd t
\end{equation*}
whereby we deduce by monotone convergence that
\begin{equation}\label{eq:F}
    \E \left\{ \sum_{{t}\in I, Z({t})=0} J({t},Y^{{t}}) \right\} = \int_{I} \E\left[ \vert \det \{Z^\prime({t})\} \vert J({t},Y^{{t}}) \big\vert Z({t})= (0,0) \right] f_{Z({t})}(0,0) \dd t.
\end{equation}
Taking $I=S_1^\prime\times S_2^\prime :=\overline{B(0,\varepsilon^\prime)}\times \overline{B(re_1,\varepsilon^\prime)}$ for some $0<\varepsilon^\prime<\varepsilon$, we have
\[
    \sum_{t\in I, Z({t})=0} J(t,Y^t) = \sum_{t_1 \in Y_\L,t_2\in Y_{\L^\prime}}^{\neq} \mathbf 1(t_1\in S_1^\prime)\,\mathbf 1(t_2\in S_2^\prime).
\]
By the Campbell theorem,
\begin{align}\label{camp}
\E \left\{ \sum_{t_1 \in Y_\L,t_2\in Y_{\L^\prime}}^{\neq} \mathbf 1(t_1\in S_1^\prime)\, \mathbf 1(t_2\in S_2^\prime) \right\} = \int_{S_1^\prime\times S_2^\prime}\rho_{\mathcal{L},\mathcal{L}^\prime}^{(2)}(t_1,t_2)   \dd t. 
\end{align}
Identifying \eqref{eq:F} and \eqref{camp} yields that
\begin{align*}
    \rho_{\mathcal{L},\mathcal{L}^\prime}^{(2)}(0,re_1) = & f_{V(r)}(0,0) \times \E 
    \Big[
    |\det\{X^{\prime\prime}(0)\}| \times |\det\{X^{\prime\prime}(re_1)\}| \; \times \nonumber \\
    & \hspace*{2cm} \iota_\L\{X^{\prime\prime}(0)\} \times \iota_{\L^\prime}\{X^{\prime\prime}(r e_1)\} 
    \mid X^\prime(0)=X^\prime(re_1)=0
    \Big]
\end{align*}
which leads to~\eqref{eq:gLLprime}. It only remains to prove the continuity of $g_{\L,\L^\prime}$. First, $u\mapsto f_{V(u)}(0,0)$ is clearly continuous at $r$ since $V(r)$ is non degenerate. Hence, it suffices to check the continuity of $u \mapsto \mathbb{E}\{ \varphi(\tilde{X}_{0}, \tilde{X}_u) \}$, where $(\tilde{X}_{0}, \tilde{X}_{u})$ is distributed according to the conditional distribution of $\{X^{\prime\prime}(0)^\top ,X^{\prime\prime}(ue_1)^\top \}^\top $ given $X^\prime(0)=X^\prime(ue_1)=0$, and
\begin{equation*}
    \varphi(x_0, x_u) = |\det(x_0)| \times |\det(x_u)| 
    \iota_\L(x_0)  \iota_{\L^\prime}(x_u).
\end{equation*}
By \cite[Lemma 8]{azais2022mean}, this conditional distribution corresponds to the one of a centered Gaussian random vector with $d(d+1) \times d(d+1)$ covariance matrix say $\Sigma(u)$  with elements depending only on $c_2^\prime(v),c_2^{\prime\prime}(v)$ for $v=0,u^2$. From \cite[Lemma 8]{azais2022mean}, one can check that $u\mapsto \Sigma(u)$ is continuous at $r$ which in turn implies the distribution of $\tilde{X}_{u}$ is continuous at $r$ with respect to the weak convergence. Finally, by continuity of $\varphi$, the continuous mapping theorem implies the continuity of $u \mapsto \mathbb{E}\{ \varphi(\tilde{X}_{0}, \tilde{X}_{u}) \}$ at $r$.
\end{proof}


\subsection{Proof of Theorem~\ref{thm:rhok}}
\label{app:proof:rhok}

\begin{proof} The proof extends the one of Theorem~\ref{thm:pcf}. Since $t_1,\dots,t_k$ are pairwise distinct, there exists $\varepsilon>0$ such that $S_i=B(t_i,\varepsilon)$ for $i=1,\dots,k$ are pairwise disjoint. Let $s\in S= \times_{i=1}^k S_i$ and $V(s)=\{X^\prime(s_1)^\top,\dots,X^\prime(s_k)^\top\}^\top$. Its covariance matrix involves terms of the form $c_1^{\prime\prime}(\|s_i-s_j\|)$. By continuity of $c_1^{\prime\prime}$, the distribution of $V(s)$ is non-degenerate for all $s\in S$.
We apply Theorem~\ref{thm:rice2} to $Z(s)=V(s)$, $D=\{0\}\subset \R^{kd(d+1)/2}$ and the functional $J$ given by
$$  
\begin{matrix} J: &
    S \times C^0 (D, \mathbb{R}^{kd(d+1)/2}) & \rightarrow & \mathbb{R}\\ &
    ({s},f) & \mapsto & \iota_\L \{f_1(0)\} \times \dots \times \iota_\L \{f_k(0)\}.
\end{matrix}
$$
Again, all assumptions of Theorem~\ref{thm:rice2} are satisfied except the continuity of $J$.
Then, we proceed in a similar way as for proof of Theorem~\ref{thm:pcf} to get~\eqref{eq:F}. Taking, $I=\times_{i=1}^k S_i^\prime$ with $S_i^\prime=\overline{B(s_i,\varepsilon^\prime)}$ with $0<\varepsilon^\prime<\varepsilon$, the conclusion follows using the $k$th order Campbell Theorem after noticing that 
\[
    \sum_{s\in I, Z({s})=0} J(s,Y^s) = \sum_{s_1,\dots,s_k \in Y_\L}^{\neq} \mathbf 1(s_1\in S_1^\prime) \times \dots \times \mathbf 1(s_k\in S_k^\prime).
\]
\end{proof}


\section{Proof of Theorem~\ref{thm:convcrit}}
\label{app:convergence:critical:points}

The proof relies on two technical results dealing with the behavior of ``generic'' roots of a $C^1$ function when one performs small $C^1$ modifications of that function. First, the case of a single root is treated in Lemma~\ref{lem:convergence:one:root} which in turn is extended to the set of roots in Proposition~\ref{prop:convergence:all:roots}. Finally, Theorem~\ref{thm:convcrit} follows from the continuous mapping theorem.

\begin{lemma}\label{lem:convergence:one:root}
Let $D\subset \mathbb{R}^d$ be a compact set and $f:D \to
\mathbb{R}^d$ be a $C^1(D,\R^d)$ function such that $f(x_0)=0$
  for a unique $x_0 \in D$. Furthermore, assume that $x_0 \in
\interior D$ and $f^\prime(x_0)$ is invertible. Let $(f_n)_n$ be a
sequence of functions such that $f_n \to f$ in $C^1(D,\R^d)$.
Then
there exists $n_0$ such that for any $n\geq n_0$, there exists a
unique $x_n \in D$ such that $f_n(x_n)=0$ and $\iota \{f_n^\prime({x_n})\} = \iota \{f^\prime({x_n})\}$. Furthermore, $x_n \to x_0$.
\end{lemma}

\begin{proof}
The proof is inspired by the proof of the inverse function Theorem by \cite{tao2023analysis}. 
For all $x\in \interior D$, let $\rho(x) = \min\{|\lambda|, \lambda
\text{ eigenvalue of } f^\prime(x)\}$. Note that $\rho(x_0) >0$ by
assumption. Let $g:D\to \mathbb{R}^d$ be defined by $g(x) = f(x) -
f^\prime(x_0)(x-x_0)$. By definition, $g(x_0)=0$ and
$g^\prime(x_0)=0$. Since $f\in C^1(D,\R^d)$, $D$ is compact and
$x_0\in \interior D$, there exist $r, \eta>0$ such that  for all $x\in B(x_0,r)$: 
(i) $\rho(x) > 0$; (ii) $\|f(x)\|>\eta$; (iii) the derivative  of $g$ in the direction $v$, $\partial_v g(x) = \lim_{t\to 0} \{g(x+tv)-g(x)\}/t$, satisfy
\begin{equation}\label{eq:bound:derivative:g}
        \|\partial_v\, g(x)\| < \frac{\rho(x_0)}{2d}\|v\| \quad \text{ for all } v\in \mathbb{R}^d.
\end{equation}
In particular, (i) implies that $\iota \{f^\prime(x)\}$ is constant for
$x\in B(x_0,r)$. Let us also define $\tilde{f},
\tilde{g}:B({0},r) \to \R^d$ by
$\tilde{f}(y)=f^\prime(x_0)^{-1}f(y + x_0)$ and $\tilde{g}(y) =
\tilde{f}(y) - y = f^\prime(x_0)^{-1}g(y + x_0)$. Remark that the directional derivatives of $\tilde{g}$ satisfy \eqref{eq:bound:derivative:g} for all $y\in B({0},r)$ where $\rho(x_0)$ is replaced by $1$. The fundamental theorem of calculus implies that for all $y_1,y_2\in B({0},r)$,
\begin{equation*}
\tilde{g}(y_1)-\tilde{g}(y_2) = \int_{0}^{1}  \partial_{y_1-y_2}\tilde{g}\{y_2 + t(y_1-y_2)\} \dd t.
\end{equation*}
By~\eqref{eq:bound:derivative:g}, each component of the integrand is bounded by $\|y_1-y_2\|/2d$,
which in turn implies that $\| \tilde{g}(y_1)-\tilde{g}(y_2) \| \leq
\|y_1-y_2\|/2$. By \cite[Lemma 6.6.6]{tao2023analysis}, the map $\tilde{f} = \tilde{g} + I$ is injective and $\tilde{f}\{B({x_0},r)\} \supset B({x_0},r/2)$. In turn, since $f^\prime(x_0)$ is injective and $f(x) = f^\prime(x_0)\tilde{f}(x-x_0)$, we know that $f$ is injective on $B(x_0,r)$ and $f\{B(x_0,r)\} = f^\prime(x_0)[\tilde{f}\{B(0,r)\}] \supset B\{0,\rho(x_0)r/2\}$.
    
We now define for any $n\ge 1$, $\rho_n$ and $g_n$ analogously,
replacing $f$ by $f_n$. Let $0<\varepsilon <r$. It only remains to
prove that there exists $n_0$ such that for any $n\geq n_0$, there
exists a unique $x_n \in B(x_0, \varepsilon)$ such that $f_n(x_n)=0$
and $f_n(x) \neq 0$ for all $x\notin B(x_0,\varepsilon)$ . 
Since $f_n \to f$ in $C^1(D,\R^d)$, there exists $n_0$ such that, for
all $n\geq n_0$: (1) $\iota \{f_n^\prime(x_0)\} = \iota
\{f^\prime(x_0)\}$; (2) for all $x\in B(x_0,r)$, $\rho_n(x) >
\rho(x)/2$; (3) for all $x\in B(x_0,r)$, the directional derivatives
of $g_n$ satisfy \eqref{eq:bound:derivative:g}; (4) $\|f_n -
f\|_{\infty} < \min\{\rho(x_0) \varepsilon / 4, \eta\}$. 
From now on, assume that $n\geq n_0$. In particular, $\iota \{f_n^\prime(x)\}$ is constant equal to $\iota \{f^\prime(x_0)\}$ for $x\in B(x_0,r)$.    
Using property (3) above, the same arguments as the ones we applied to
$f$ imply that $f_n$ is injective on $B(x_0,r)$ and : (a)
$f_n\{B(x_0,r)\} \supset B\{f_n(x_0),\rho_n(x_0)r/2\}$; (b) $f_n\{B(x_0,\varepsilon)\} \supset
B\{f_n(x_0),\rho_n(x_0)\varepsilon/2\}$. By property (4) above, we
know that $\|f_n(x_0)\| < \rho(x_0) \varepsilon /4 $
which implies, thanks to property (2), that $0 \in
B\{f_n(x_0),\rho_n(x_0) \varepsilon/2\} \subset
B\{f_n(x_0),\rho_n(x_0) r/2\}$. Hence, property (a) implies
that there exists a unique $x_n \in B(x_0,r)$ such that $f_n(x_n)=0$
and property (b) implies furthermore that $x_n\in B(x_0,\varepsilon)$. Finally, property
(4) implies that for all $x\notin B(x_0,r)$,
$f_n(x)\neq  0$ which ends the proof.
\end{proof}

The result below states the continuity of the non-degenerate
roots for a function $f$ having no roots 
boundary of its  domain.

\begin{proposition}\label{prop:convergence:all:roots}
Let $D\subset \mathbb{R}^d$ be a compact set and $f:D \to
\mathbb{R}^d$ be a $C^1(D,\R^d)$ function such that: for all $x\in
\partial D$, $f(x) \neq 0$ and for all $x\in D$ such that $f(x)=0$,
$f'(x)$ is invertible. Let $(f_n)_n$ be a sequence of functions such
that $f_n \to f$ in $C^1(D,\R^d)$. Then, for all $\ell=0,\dots,d$, the following convergence holds as $n\to \infty$
\begin{equation*}
\{x\in D, f_n(x)=0, \iota \{f_n^\prime(x)\}=\ell\} \to \{x\in D, f(x)=0, \iota \{f^\prime(x)\}=\ell\}.
\end{equation*}
\end{proposition}
\begin{proof}
Since the roots of $f$ are non-degenerate, the points of $S:= \{x\in D, f(x)=0,
\iota \{f^\prime(x)\}=\ell\}$ are isolated which in turn implies that
it is a finite set since $D$ is compact. Let $S =
\{x_1,\dots,x_p\}$. There exist $r,\eta>0$ such that for all $k\neq
\ell$, $\overline{B(x_k,r)} \cap \overline{B(x_\ell,r)} = \emptyset$
and, since $D\setminus \{\cup_{k=1}^p B(x_k,r)\}$ is compact, for all
$x\notin \cup_{k=1}^p B(x_k,r)$, $\|f(x)\| > \eta$. Let $n_0(k)$ be
the integer appearing in Lemma~\ref{lem:convergence:one:root} when
applied to the compact set $\overline{B(x_k,r)}$. Choose $n_0'$ be
such that for all $n\geq n_0'$, $\|f_n - f\|_{\infty} < \eta$. It is
then clear that for all $n\geq \max\{n_0(1),\dots,n_0(p), n_0'\}$, the cardinality of $S_n := \{x\in D, f_n(x)=0, \iota\{f_n'(x)\}=\ell\}$ is equal to $p$ and finally that $S_n \to S$ as $n\to \infty$.
\end{proof}

\begin{proof}[Proof of Theorem~\ref{thm:convcrit}]
    By Proposition \ref{prop:convergence:all:roots}, the mapping $f\in C^1([0,1]^d) \mapsto \{x\in D, f(x)=0, \iota \{f^\prime(x)\}=\ell\}$ is continuous at every $f\in M$, where
    \begin{align*}
    M = & \Bigg\{f\in C^1([0,1]^d): \forall x\in \partial [0,1]^d, f(x) \neq 0 \text{ and } \forall x\in [0,1]^d, \\
    & \hspace*{5cm} f(x)=0 \Rightarrow \det\{f^\prime(x)\}\neq 0 \Bigg\}.
    \end{align*}
    By assumption, $\bX^\prime \in M$ almost surely and $\bX^\prime_n
    \xrightarrow{\mathcal{D}} \bX^\prime$ in
    $C^1([0,1]^d,\R^d)$.
\end{proof}


\section{Proof of Theorem~\ref{thm:Xab}}
\label{app:convergence:simulation}

The proofs of the two items are given in two separate sections.

\subsection{Proof of Theorem~\ref{thm:Xab}-(i)}

The proof relies on a decomposition of the difference $\bX_n -
  \bX$ and its derivatives. Using the regularity of the field $\bX$
and the kernel $k$, one can control the convergence rate of each term
in the decomposition.

\begin{proof}
By assumption, $\bX \in C_{a.s.}^{2+\varepsilon}(W)$ and by construction $\bX_n \in C_{a.s.}^{3}(W)$.
We intend to show that almost surely
$\|\bX_n - \bX \|_{C^2(W)} \to 0$. We let
$K = [-1,1]^d$ denote the compact support of the kernel $k$ and
consider $n$ be large enough so that $\xi_n \mathrm{diam}(K)<1$. We
further define $L_n^- = L_n \cap (W)_{\ominus \xi_n K}$,
$L_n^\prime=L_n\cap \{W\setminus (W)_{\ominus \xi_n K}\}$ and
$L_n^{\prime\prime} = L_n\cap \{(W)_{\oplus \xi_n K}\setminus
(W)_{\ominus \xi_n K} \}$. 

By definition of $k$, $L_n$ and $W$, we have for any $x\in L_n\cap W$ and as $n\to \infty$
\begin{align}
|\mathcal C_{n,x}| &= n^{-d}, \quad \operatorname{diam}(\mathcal C_{n,x})= n^{-1},\quad |L_n \cap W|=n^d, \quad  |L_n^-| = O(n^d), \label{eq:latt1}\\
|L_n^\prime|&= O\{( n \xi_n)^{d-1}\} \quad \text{ and } \quad |L_n^{\prime\prime}|= O\{( n \xi_n)^{d-1}\} . \label{eq:latt2}
\end{align}
Moreover, by the assumptions on the kernel function $k$, it holds for
any $\ell \in \{0,1,2,3\}$ that
\begin{align}
\|k_{\xi_n}\|_{C^\ell(\R^d)} &= \sum_{\alpha \in \mathbb Z^d, |\alpha| \le \ell } \sup_{y \in \R^d} |\partial^\alpha k_{\xi_n}(y)| \nonumber\\
&\le     \sum_{\alpha \in \mathbb Z^d, |\alpha| \le \ell } \sup_{y \in \R^d} \xi_n^{-d-|\alpha|} |\partial^\alpha k(y/\xi_n)| \nonumber\\
& \le \xi_n^{-d-\ell} \|k\|_{C^\ell(\R^d)} = O(\xi_n^{-d-\ell}) \label{eq:k}.
\end{align}
Let $t\in W$ and $\alpha\in \mathbb{N}^d$ such that  $|\alpha| \leq
2$. We decompose
$\partial^\alpha (X_n - X) (t) = \sum_{i=1}^5 \Delta_{\alpha,i}(t)$,
where
\begin{align}
\Delta_{\alpha,1}(t) &:= n^{-d} \sum_{x \in L_n^\prime} \partial^\alpha k_{\xi_n}(t-x) X(x)  \\
\Delta_{\alpha,2}(t) &:= \sum_{x \in L_n^-} X(x) \left\{ |\mathcal C_{n,x}| \partial^\alpha k_{\xi_n}(t-x) - \int_{\mathcal{C}_{n,x}} \partial^\alpha k_{\xi_n}(t-y) \dd y \right\} \\
\Delta_{\alpha,3}(t) &:= \sum_{x\in L_n^-} \int_{\mathcal C_{n,x}} \partial^\alpha k_{\xi_n}(t-y) \{X(x) - X(y)\} \dd y\\
-\Delta_{\alpha,4}(t) &:= \sum_{x\in L_n^{\prime\prime}} \int_{\mathcal C_{n,x}} \partial^\alpha k_{\xi_n}(t-y) X(y)\dd y \\
\Delta_{\alpha,5}(t) &:= \int_{(W)_{\oplus \xi_n K}} \partial^\alpha k_{\xi_n}(t-y) X(y) \dd y - \partial^\alpha X(t).
\end{align}
Note that $\bX_n$ is defined on $W$ only whereas $\bX$ is defined on a larger domain (we take $\mathbb{R}^d$ without loss of generality).
In the following we use extensively \eqref{eq:latt1}-\eqref{eq:k} and the fact that $\bX$ satisfies \ref{C:general}$[2+\varepsilon]$. We have
\begin{align*}
|\Delta_{\alpha,1}(t)| &\le  \|\bX\|_{C^0(W)} \, \|k\|_{C^{|\alpha|}(\R^d)} \,  n^{-d} \,\xi_n^{-d-|\alpha|} \,|{L_n^{\prime}}|   \\
&= O\{ n^{-d} \xi_n^{-d-2} (n\xi_n)^{d-1}\} = O(n^{-1} \xi_n^{-3}).    
\end{align*}
For the second term, using the mean-value theorem on $\partial^{\alpha} k_{\xi_n}$, we obtain
\begin{align*}
|\Delta_{\alpha,2}(t)| &\le  \|\bX\|_{C^0(W)} \sum_{x \in L_n^-} \int_{\mathcal{C}_{n,x}} \|k_{\xi_n}\|_{C^{|\alpha|+1}(\R^d)} \, \|y-x\| \dd y \\
&\le \|\bX\|_{C^0(W)} \|k\|_{C^3(\R^d)} \,\xi_n^{-d-3}  \,|L_n^-| \,|\mathcal C_{n,x}| \, \mathrm{diam}(\mathcal C_{n,x})\\
&= O(n^{-1} \xi_n^{-d-3}).
\end{align*}
The third term satisfies
\begin{align*}
|\Delta_{\alpha,3}(t)| &\le \|k\|_{C^{|\alpha|}(\R^d)} \,\xi_n^{-d-|\alpha|}\, \sum_{x\in L_n^-} \int_{\mathcal C_{n,x}} |X(x)-X(y)| \dd y \\
 &\le \|k\|_{C^2(\R^d)} \, \|\bX\|_{C^1(W)} \,\xi_n^{-d-2}\, |L_n^-| \,|\mathcal C_{n,x}| \,\mathrm{diam}(\mathcal C_{n,x} ) \\
&= O(n^{-1} \xi_n^{-d-2}).     
\end{align*}
For the fourth one, we have
\begin{align*}
|\Delta_{\alpha,4}(t)| & \le    \|k\|_{C^{|\alpha|}(\R^d)} \, \|\bX\|_{C^0\{ (W)_{\oplus K}\}}\, \xi_n^{-d-|\alpha|} \,|\mathcal C_{n,x}| \, |L_n^{\prime\prime}| \\
&=O(n^{-1} \xi_n^{-3}).
\end{align*}
Finally, for the fifth term, since $k$ is supported on $K$
\begin{equation*}
    \Delta_{\alpha,5}(t)   =  
    \int_{\R^d} \partial^\alpha k_{\xi_n}(t-y) X(y) \dd y - \partial^\alpha X(t).
\end{equation*}
The integral term is the convolution product $\partial^\alpha k_{\xi_n} * X$. By integration by parts\begin{align*}
|\Delta_{\alpha,5}(t)|
&=\left| \int_{\mathbb{R}^d} k_{\xi_n}(t-y) \left\{ \partial^\alpha X(y) - \partial^\alpha X(t)\right\} \dd y \right|\\
&\leq  \| \bX \|_{C^2\{ (W)_{\oplus K}\}} \int_{\mathbb{R}^d} k_{\xi_n}(t-y) \left\| t-y \right\|^\varepsilon \dd y\\
&\leq  \| \bX \|_{C^2\{ (W)_{\oplus K}\}}\,  \xi_n^\varepsilon \int_{\mathbb{R}^d} k(z) z^\varepsilon \dd z\\
&=  O (\xi_n^\varepsilon).
\end{align*}
All in all, we have
\begin{equation*}
    \|\bX_n - \bX \|_{C^2(W)} = O( n^{-1} \xi_n^{-d-3} + \xi_n^\varepsilon) = o(1)
\end{equation*}
by assumption on $\xi_n$, which ends the proof.
\end{proof}

\subsection{Proof of Theorem~\ref{thm:Xab}-(ii)}

The convergence is obtained by a central limit theorem (CLT). To the best of
our knowledge, $C^2$ is not a convenient Banach space for the
CLT. We therefore use Sobolev embedding and view our random fields
as elements of a Hilbert Sobolev space $\mathcal{W}$ where we are
sure that the CLT holds. Then, the proof follows three steps: 1) find a convenient embedding, 2) control the norms of our random fields as $\mathcal{W}$ valued random variables, 3) check that the Gaussian limit variable, denoted $G$, in $\mathcal{W}$ corresponds to the Gaussian field $\bX$.

\begin{proof}
    Define $m=\lfloor d/2 \rfloor +1$. Let $\mathcal W=H^{2+m}$ be the Hilbert Sobolev space of all real-valued functions on $(0,1)^d$ whose weak derivatives up to order $2+m$ are functions in $L^2([0,1]^d)$. Its associated norm is given for all $h\in \mathcal W$ by
    \begin{equation*}
        \|h\|_{\mathcal W}^2 := \sum_{|\alpha| \leq 2+m} \|\partial^{\alpha}h\|_{L^2}^2.
    \end{equation*}
    Here are the two reasons why $\mathcal W$ is considered here:
    \begin{itemize}
        \item by the Sobolev embedding Theorem \cite[4.12, PART
          II]{adams2003sobolev}, the embedding $\mathcal W \to
          C^2([0,1]^d)$ is continuous since $(0,1)^d$ satisfies the
          strong local Lipschitz condition. In particular, the
          convergence $\mathbf{X}^n \to \mathbf{X}$ in $\mathcal W$
          implies its convergence in $C^2([0,1]^d)$ by the continuous
          mapping theorem. Hence, we now turn to the proof of this
          convergence in $\mathcal W$.
        \item since $\mathcal W$ is a separable Hilbert space, we can
          apply the central limit theorem, Theorem~10.5 in
          \cite{ledoux2013probability}.
    \end{itemize}
Let $\bZ$ be a random field given by~\eqref{eq:Zib} with
$U\sim\mathcal U([0,2\pi])$, $V\sim F$, $W\sim\mathcal U([0,1])$ and
$U,V,W$ independent. We first prove that $\mathbf{Z}$ can be
considered as a random variable in $\mathcal W$ and furthermore prove
that it admits a finite second moment.
    Letting $\alpha$ be a multi-index, we have
    \begin{equation*}
        \partial^{\alpha}Z(s) = \sqrt{-2\log(W)} V^{\alpha} g_{|\alpha|}\left( U + s^\top V \right)
    \end{equation*}
    where $V^{\alpha} = \prod_{l=1}^{d} V_l^{\alpha_l}$, and, for all non-negative
    integer $k$, $g_{4k} = \cos$, $g_{4k+1} = -\sin$, $g_{4k+2} = -\cos$ and $g_{4k+3} = \sin$.
    Then, by boundedness of sinus and cosinus functions, it is clear
    that $\|\mathbf{Z}\|_{\mathcal W}^2 \leq \sqrt{-2\log(W)}
    \sum_{|\alpha|\leq 2+m} (V^{\alpha})^2 \leq  \kappa \sqrt{-2\log(W)} \,  \|V\|^{4+2m}$
    for some  $\kappa>0$. Yet, $4+2m = 6+2\lfloor d/2 \rfloor$ which ensures that $\mathbb{E}\left( \|V\|^{4+2m} \right) < \infty$ by assumption. In turn, the independence of $V$ and $W$ implies that $\mathbf{Z}$ has a finite second moment as a random variable in $\mathcal W$. By definition its expectation, denoted $\mathbb{E}\left( \mathbf{Z} \right)$, is the element of $\mathcal W$ such that, for all $h \in \mathcal W$, $\mathbb{E}\left( \left< h, \mathbf{Z} \right>_{\mathcal W} \right) = \left< h, \mathbb{E}\left( \mathbf{Z} \right) \right>_{\mathcal W}$, where
    \begin{equation*}
        \left< h_1,h_2 \right>_{\mathcal W} = \sum_{|\alpha|\leq 2+m} \int  \partial^{\alpha}h_1(s) \partial^{\alpha}h_2(s)  \dd s.
    \end{equation*}
    For all $|\alpha|\leq 2+m$, 
    \begin{equation*}
        \int  |\partial^{\alpha}h(s) \partial^{\alpha}Z(s)| ds \leq \|h\|_{\mathcal W} \|\mathbf{Z}\|_{\mathcal W}
    \end{equation*}
    and $\mathbb{E}\left( \|\mathbf{Z}\|_{\mathcal W} \right) <\infty$. Hence, by Fubini, $\mathbb{E}\left\{ \int  \partial^{\alpha}h(s) \partial^{\alpha}Z(s) \dd s \right\} = \int  \partial^{\alpha}h(s) \mathbb{E}\left\{ \partial^{\alpha}Z(s) \right\} \dd s$. Yet, by independence between $U,V,W$ and since $U\sim \mathcal{U}([0,2\pi])$, we obtain $\mathbb{E}\left\{ \partial^{\alpha}Z(s) \right\} = 0$ for all $s$. All in all, we get that 
    $\mathbb{E}\left( \left< h, \mathbf{Z} \right>_{\mathcal W}
    \right) = 0$ for all $h\in \mathcal W$ which in turn implies that
    $\mathbb{E}\left( \mathbf{Z} \right)=0$. Hence, $\mathbf{Z}$ is a
    zero mean random variable with finite second moment in $\mathcal
    W$. Theorem~10.5 in \cite{ledoux2013probability} then implies that $\mathbf{X}_n \xrightarrow[]{\mathcal{D}} G$, where $G$ is a centered Gaussian variable in $\mathcal W$ with the same covariance as $\mathbf{Z}$. 
    
    It remains to prove that $G$ has the same distribution as $\mathbf{X}$ to end the proof. By Sobolev embedding, $\mathcal W \subset C^2([0,1]^d)$ so the (evaluation) linear forms $\phi_t : h\in \mathcal W \mapsto h(t) \in \mathbb{R}$ are continuous for all $t\in [0,1]^d$. Hence, for any $n\ge 1$, $\alpha_1,\dots,\alpha_n\in \mathbb{R}$ and $t_1,\dots,t_n\in [0,1]^d$, $G_{\alpha,t} = \sum_{i=1}^{n} \alpha_i G(t_i)$ is a real Gaussian variable. Moreover, since $G$ has the same covariance as $\mathbf{Z}$, the variance of $G_{\alpha,t}$ is equal to the variance of $\sum_{i=1}^{n} \alpha_i Z(t_i)$ which in turn is equal to the variance of $\sum_{i=1}^{n} \alpha_i X(t_i)$ since they share the same covariance function $c$. In turn, the finite distributions of $G$ and $\bX$ are equal which concludes the proof.
\end{proof}


\section{Lemma~\ref{lem:isotropic} and its proof} 
\label{sec:link:to:anisotropy}

The following Lemma investigates conditions~\eqref{eq:c4at0} and~\eqref{eq:arcones} in the isotropic setting.

\begin{lemma} \label{lem:isotropic}
Assume~\ref{C:general}$[2]$. Conditions~\eqref{eq:Xi1} and~\ref{C:integrability} imply conditions~\eqref{eq:c4at0} and~\eqref{eq:arcones}.
\end{lemma}

\begin{proof}
Let $i,j,k,l$ be pairwise distinct indices of $\{1,\dots,d\}$ (when possible). Then, by assumption and since $c(t)=c_2(\|t\|^2)$, we have
\begin{align*}
\frac{\partial c}{\partial t_i}(t) &= 2t_i c_2^\prime(\|t\|^2) \\
\frac{\partial^2 c}{\partial t_i^2} (t)&= 2 c_2^\prime(\|t\|^2) + 4t_i^2 c_2^{\prime\prime}(\|t\|^2), \qquad 
\frac{\partial^2 c}{\partial t_i\partial t_j}(t) =  4t_it_j c_2^{\prime\prime} (\|t\|^2)   \\
\frac{\partial^3 c}{\partial t_i^3}(t) &= 12 t_i c_2^{\prime\prime}(\|t\|^2) + 8t_i^3 c_2^{\prime\prime\prime}(\|t\|^2), \quad \frac{\partial^3 c}{\partial t_i^2\partial t_j} (t)= 4t_j c_2^{\prime\prime}(\|t\|^2)+8t_i^2t_j c_2^{\prime\prime\prime}(\|t\|^2)\\
\frac{\partial^3 c}{\partial t_i\partial t_j \partial t_k} (t)&= 8t_it_jt_k c_2^{\prime\prime\prime}(\|t\|^2)
\end{align*}
and 
\begin{align*}
\frac{\partial^4 c}{\partial t_i^4} (t)&= 12 c_2^{\prime\prime}(\|t\|^2) + 48t_i^2 c_2^{\prime\prime\prime}(\|t\|^2) + 16t_i^4 c_2^{(4)}(\|t\|^2) \\
\frac{\partial^4 c}{\partial t_i^3\partial t_j} (t)&= 24t_it_j c_2^{\prime\prime\prime}(\|t\|^2) + 16t_i^3t_j c_2^{(4)}(\|t\|^2) \\
\frac{\partial^4 c}{\partial t_i^2\partial t_j^2} (t)&= 4c_2^{\prime\prime}(\|t\|^2) + 8t_i^2 c_2^{\prime\prime\prime}(\|t\|^2) + 8t_j^2 c_2^{\prime\prime\prime}(\|t\|^2) + 16t_i^2t_j^2 c_2^{(4)}(\|t\|^2) \\
\frac{\partial^4 c}{\partial t_i^2\partial t_j\partial t_k} (t)&= 8t_jt_k c_2^{\prime\prime\prime}(\|t\|^2) + 16t_i^2t_jt_k c_2^{(4)}(\|t\|^2) \\
\frac{\partial^4 c}{\partial t_i\partial t_j\partial t_k \partial t_l} (t)&= 16t_it_jt_kt_l c_2^{\prime\prime\prime}(\|t\|^2).
\end{align*}   
Let $r=\|t\|$. Since for any $i$, $|t_i|\leq r$, it is clear that for
any $\alpha\in \mathbb N^d$ such that $|\alpha| \in \{1,\dots,4\}$,
there exists a nonnegative real number $\kappa$  such that
$|\partial^\alpha c(t)| \leq \kappa \, \Xi(r)/r^{d-1}$ which
yields~\eqref{eq:arcones}. In the same way let $\alpha\in \mathbb N^d$
such that $|\alpha|=4$. Then there exists $\kappa>0$ such that
\[
    \frac{|\partial^\alpha c(t)-\partial^\alpha c(0)|}{\|t\|^d} 
    \leq \kappa 
    \left\{ 
    \frac{|c_2^{\prime \prime}(r)-c_2^{\prime\prime}(0)|}{r^d} + 
    \frac{|c_2^{\prime\prime\prime}(r^2)|}{r^{d-2}} 
    \frac{|c_2^{(4)}(r^{2})|}{r^{d-4}}
    \right\}
\]
which yields~\eqref{eq:c4at0}.
\end{proof}


\section{On Bessel functions derivatives}
\label{sec:bessel:derivatives}

Recall that $J_\nu$ and $K_\nu$ respectively denote the Bessel function of the first kind and the modified Bessel function of the second kind. This section provides recurrence formulas for the following related functions $f_\nu, g_\nu : \mathbb{R}^+ \to \mathbb{R}$ defined by,
\begin{equation*}
    f_\nu(r) = (\sqrt{r})^{-\nu} J_\nu(\sqrt{r}) 
    \quad \text{and} \quad
    g_\nu(r) = (\sqrt{r})^{\nu} K_\nu(\sqrt{r}).
\end{equation*}

\begin{lemma}
    For all $\nu\in \mathbb{R}$, $p\in \mathbb{N}$ and $r>0$,
    \begin{equation*}
        f_\nu^{(p)}(r) = (-2)^{-p} f_{\nu+p}(r)
        \quad \text{and} \quad
        g_\nu^{(p)}(r) = (-2)^{-p} g_{|\nu-p|}(r).
    \end{equation*}
\end{lemma}
\begin{proof}
By \cite[Section 9.1.27]{abramowitz1968handbook}, we have
$J_{\nu}'(r) = -J_{\nu+1}(r) + \frac{\nu}{r} J_{\nu}(r)$. Hence, it
follows from th echain rule that
\begin{equation*}
    f_\nu'(r) = -\frac{\nu}{2} r^{-\frac{\nu}{2}-1} J_{\nu}(\sqrt{r}) + \frac{1}{2} r^{-\frac{\nu}{2} - \frac{1}{2}}\{ -J_{\nu+1}(\sqrt{r}) + \nu r^{-\frac{1}{2}} J_{\nu}(\sqrt{r})\} = (-2)^{-1} f_{\nu+1}(r)
\end{equation*}
and the desired result follows by induction on $p$. The result regarding $g_{\nu}$ follows similarly using the fact that $K_{\nu}'(r) = -K_{\nu-1}(r) - \frac{\nu}{r} K_{\nu}(r)$ (see \cite[Section 9.6.26]{abramowitz1968handbook}) and $K_{-\nu} = K_{\nu}$.
\end{proof}

Finally, the expressions \eqref{eq:derivatives:c2:matern} and
\eqref{eq:derivatives:c2:RWM} follow from that $c_2(r) = A g_\nu(ar)$ for Matérn and $c_2(r) = B f_{d/2-1}(br)$ for the RWM, with $A = 2^{1-\nu}/\Gamma(\nu)$, $a=2\nu/\varphi^2$, $B=2^{d/2-1}\Gamma(d/2)$ and $b=d/\varphi^2$. We finally remind two useful results to establish the behaviors as $r\to0$ in~\eqref{eq:derivatives:behaviour:matern} and~\eqref{eq:derivatives:behaviour:RWM}.
\begin{lemma}{\cite[
    ]{abramowitz1968handbook}} \label{lem:abramowitz} As $z\to 0$
\begin{align*}
    (z/2)^{|\alpha|}K_{|\alpha|}(z) &\to \frac{\Gamma(|\alpha|)}2, \quad \forall \alpha\neq 0 \\
    K_0(z) & \sim -\log(z/2) - \gamma \\
    (z/2)^{-\alpha}J_{\alpha}(z) &\to \frac1{\Gamma(\alpha+1)}, \quad \forall \alpha \neq -1,-2,\dots
\end{align*}
where $\gamma$ is the Euler–Mascheroni constant.
\end{lemma}

The next result investigates~\eqref{eq:Xi1} for the Matérn and random
wave models.

\begin{lemma} \label{lem:gemanL2} ${ }$\\
(i) For the Matérn model, we have 
\begin{align}
c_2^{(p)}(r^2 ) &\sim \left\{
\begin{array}{ll}
\kappa_1 \, r^{\nu-p-|\nu-p|} & \text{ when } r\to 0 \text{ and } \nu-p\neq 0\\
-\kappa_2  \, \log(r) + \kappa_3& \text{ when } r\to 0 \text{ and } \nu-p= 0\\ 
\end{array} 
\right. \label{eq:derivatives:behaviour:matern}
\end{align}
for some constants $\kappa_1,\kappa_2, \kappa_3$, which implies~\eqref{eq:Xi1} is satisfied for any $\nu\ge 5/2$. \\
(ii) For the RWM, we have for any $p\ge 1$, $r>0$ and $z=r\sqrt{d}/\varphi$,
\begin{align}
c_2^{(p)}(r^2 ) = \frac{(-1)^pp!}{(2p)!}\,\lambda_{2p}\{1+o(1)\} & \text{ as } r\to 0 \label{eq:derivatives:behaviour:RWM}
\end{align}
which implies~\eqref{eq:Xi1}.
\end{lemma}

\begin{proof}
The results follow from Lemma~\ref{lem:abramowitz} and Lemma~\ref{lem:isotropic}
\end{proof}


\section{Proof of Proposition~\ref{prop:representation}} \label{sec:representation}

The idea of the proof is to approximate the test function $\phi_{1,n}$ or $\phi_{2,n}$ by step functions. For (i),  the result then immediately follows from~\cite[Proposition 1.2]{estrade2016central} (see also \cite[Lemma~3.1]{azais2024multivariate}). For (ii), an additional ingredient, namely~\eqref{eq:NXprimeLp}, is required to proceed similarly. We only present the proof of (ii).


\begin{proof}
(ii) The proof is divided in two steps.
\paragraph{Step 1.} Consider a test function of the form $\phi_{2,n}(\mathbf{t}) =
\mathbf{1}( {\mathbf{t} \in {T}})$ for some ${T}
{\subseteq} (W_n)^2$. Then $\Phi_{2,n}$ is just the number of
$(t^{(1)}, t^{(2)}) \in {T}$ such that both $t^{(1)}$ and
$t^{(2)}$ are critical points with indices in $\L$. The almost sure
convergence of $\Phi_{2,n,\varepsilon} \to \Phi_{2,n}$ follows
from~\cite[Theorem 11.2.3.]{adler2007random} applied to
$f(\mathbf{t}) = \{X^\prime(t^{(1)})^\top,
X^\prime(t^{(2)})^\top\}^\top$, $g(\mathbf{t}) =
\{X^{\prime\prime}(t^{(1)})^\top,
X^{\prime\prime}(t^{(2)})^\top\}^\top$ where in particular
$B$ is the inverse image of $\L$ by $\iota$. Let $\Delta$ be a compact subset of $\mathbb{R}^d$ such that ${T} \subset \Delta \times \Delta$ and observe that $\left| \Phi_{2,n,\varepsilon} \right| \leq \{ N_{0:d,\varepsilon}(\Delta) \}^2$. As argumented in the main text, following \cite[Remark 1]{gass2024number}, conditions~\ref{C:general}$[5]$ and \ref{C:nondegeneracy}$[4]$ imply that the function $v \mapsto \E \{ N_{0:d}(\Delta,v)^p\}$ is continuous, where $N_{0:d}(\Delta,v)=\#\{X^\prime(t)=v, t\in \Delta\}$. In turn, following \cite[Proof of Proposition 1.1, p. 3872]{estrade2016central}, one can prove that the convergence~\eqref{eq:NXprimeLp} holds, that is $\{ N_{0:d,\varepsilon}(\Delta) \}^2$ converges as $\varepsilon \to 0$ in $L^2$ and in particular it is eventually dominated in $L^2$.
The dominated convergence theorem finally shows that the  convergence of $\Phi_{2,n,\varepsilon} \to \Phi_{2,n}$ also holds in $L^2$.

\paragraph{Step 2.} Let us now consider any bounded piecewise continuous function $\phi_{2,n}$ and let $(\phi_{2,n}^m)_m$ be a sequence of step functions from $(W_n)^2$ to $\R$ such that $\lim_{m\to \infty} \|\phi_{2,n}^{m}-\phi_{2,n}\|_\infty =0$. For all $m$, let $\Phi_{2,n}^{m}$ and $\Phi_{2,n,\varepsilon}^{m}$ denote the functionals associated with $\phi_{2,n}^{m}$. We use the following decomposition:
\begin{equation*}
    \Phi_{2,n} - \Phi_{2,n,\varepsilon} = \left( \Phi_{2,n} - \Phi_{2,n}^{m} \right) + 
    \left( \Phi_{2,n}^{m} - \Phi_{2,n,\varepsilon}^{m} \right) +     \left(  \Phi_{2,n,\varepsilon}^{m} - \Phi_{2,n,\varepsilon} \right).
\end{equation*}
By boundedness of $\phi_{2,n}$, $| \Phi_{2,n,\varepsilon}^{m} - \Phi_{2,n,\varepsilon} | \leq \|\phi_{2,n}^{m}-\phi_{2,n}\|_\infty \,  N_{0:d,\varepsilon}(W_n)^2$. Therefore by~\eqref{eq:NXprimeLp} and Step 1 respectively, the following statements hold almost surely and in $L^2$ sense: (i)  $\limsup_{\varepsilon\to 0} | \Phi_{2,n,\varepsilon}^{m} - \Phi_{2,n,\varepsilon} | \le  \|\phi_{2,n}^{m}-\phi_{2,n}\|_\infty  N_{0:d}(W_n)^2$; (ii)  $\lim_{\varepsilon\to0}| \Phi_{2,n}^{m} - \Phi_{2,n,\varepsilon}^{m} |= 0$. Since $N_\L(W_n)^2 \leq  N_{0:d}(W_n)^2$, $\left| \Phi_{2,n} - \Phi_{2,n}^{m} \right| \le\|\phi_{2,n}^{m}-\phi_{2,n}\|_\infty \, N_{0:d}(W_n)^2$. By Theorem~\ref{thm:gass}, $N_{0:d}(W_{n})^4$ is almost surely finite with finite expectation. In turn,  $\|\phi_{2,n}^{m}-\phi_{2,n}\|_\infty \, N_{0:d}(W_n)^2$ almost surely and in $L^2$ sense as $m\to \infty$, which concludes the proof.
\end{proof}


\section{Proof of Proposition~\ref{prop:chaos}}
\label{app:proof:prop:chaos}

The main idea of the proof is to get the Hermite expansions of
$\cent{\Phi}_{1,n}$ and $\cent{\Phi}_{2,n}$ as limits of the Hermite
expansions of the integral representations given by
Proposition~\ref{prop:representation}. This is justified by the
convergence of a mollified version of the conditional distribution involved in Kac-Rice formula (Lemma~\ref{lem:coeffhermite}). 

The upper bound~\eqref{ineq:da} for the function $d_{\mathbf{a}}$ relies on a control of $2p$th moments of scaled Hermite polynomials (see Lemma~\ref{lem:hermite}) combined with a
somehow surprising property of Gaussian vectors: rescaling a conditional distribution by a matrix that standardizes the full distribution gives a ``rather standard'' Gaussian vector in the sense that its covariance is a projection matrix (see Lemma~\ref{lem:proj}).

\subsection{Auxiliary results}

\begin{lemma} \label{lem:coeffhermite}
    Let $(X_1,X_2,X_3,X_4) \in \mathbb{R}^{d_1}\times \mathbb{R}^{d_2}
    \times \mathbb{R}^{d_3}\times \mathbb{R}^{d_4}$, $d_i \ge 1$,
    $i=1,\ldots,4$, be a non-degenerate centered Gaussian random
    vector with density $f$. Let $\varphi : \mathbb{R}^d \to
    \mathbb{R}$ with $d=\sum_{i=1}^4 d_i$ be a continuous and bounded test function in a neighborhood of $(0,0,x_3,x_4)$ for any $x_3,x_4$
    and with $\E \left[
      {\varphi(0,0,X_3,X_4)}\{f(0,0,X_3,X_4)\}^{-1/2} \right]< \infty$. Then, as $\varepsilon\to 0$
    \begin{align} 
        \mathbb{E}\left\{ \varphi(X_1,X_2,X_3,X_4) 
        \delta_\varepsilon(\|X_1\|)
        \delta_\varepsilon(\|X_2\|)
        \right\} &\to \nonumber\\
        f_{(X_1,X_2)}(0,0) \;\times \; &\E \left\{
        \varphi(0,0,X_3,X_4) \, \mid \, X_1=0,X_2=0 \right\}. \label{eq:reslem}
    \end{align}
where $f_{(X_1,X_2)}$ denotes the marginal density of $(X_1,X_2)$.
\end{lemma}

Note that the condition $\E \left[
{\varphi(0,0,X_3,X_4)}\{f(0,0,X_3,X_4)\}^{-1/2} \right]< \infty$
in the above lemma is satisfied in particular if $\varphi(0,0,x_3,x_4)$ is a polynomial function of $x_3$ and $x_4$ of any order.

\begin{proof}
Let 
\begin{align*}
\check \varphi_{\varepsilon}(x_3,x_4)&=
\int_{\R^{d_2}}\int_{\R^{d_1}} \varphi(x_1,x_2,x_3,x_4) 
\delta_\varepsilon(x_1)
\delta_\varepsilon(x_2)\times \\
&\qquad\qquad\qquad f_{(X_3,X_4)}^{X_1=x_1,X_2=x_2}(x_3,x_4)f_{(X_1,X_2)}(x_1,x_2) \dd x_1\dd x_2.
\end{align*}
where $f_{(X_3,X_4)}^{X_1=x_1,X_2=x_2}$ denotes the conditional
  density of $(X_3,X_4)$ given $(X_1,X_2)=(x_1,x_2)$.
By continuity of the conditional density, the marginal density of
$(X_1,X_2)$ and the function $\varphi(\cdot,\cdot,x_3,x_4)$ in a
neighborhood of $(0,0)$, we have as  $\varepsilon\to 0$ the following pointwise convergence
\[
    \check\varphi_{\varepsilon}(x_3,x_4) \to \varphi(0,0,x_3,x_4)f_{(X_3,X_4)}^{X_1=0,X_2=0}(x_3,x_4) f_{(X_1,X_2)}(0,0).
\]
So the result ensues from the dominated convergence theorem if one proves that there exists an integrable function, say $\psi(x_3,x_4)$ such that $|\check\varphi_{\varepsilon}(x_3,x_4)|\le \psi(x_3,x_4)$. 
Let $\Sigma$ be the covariance matrix of $(X_1,\dots, X_4)$, and $d_M$ be the Mahalanobis distance defined for all $x,y\in \mathbb{R}^{d}$ by
\begin{equation*}
d_M(x,y)^2 = {(x-y)^\top\Sigma^{-1}(x-y)}.
\end{equation*}
By the triangle inequality and convexity of the square function, we have $d_M(y,0)^2 \leq 2\{d_M(y,x)^2 + d_M(x,0)^2\}$ and so
\begin{equation}\label{eq:inequality:mahalanobis:square}
d_M(x,0)^2 \geq \frac{1}{2} d_M(y,0)^2 - d_M(y,x)^2.
\end{equation}
Let $(x_1,x_2)\in \mathbb{R}^{d_1}\times \mathbb{R}^{d_2}$.
Using~\eqref{eq:inequality:mahalanobis:square} with $x=(x_1^\top,x_2^\top,x_3^\top,x_4^\top)^\top$ and $y=(0^\top,0^\top,x_3^\top,x_4^\top)^\top$, we have
\begin{align*}
f(x_1,&x_2,x_3,x_4) =\det(2\pi \Sigma)^{-1/2}\exp\left[-\frac{1}{2} d_M\left\{(x_1^\top,x_2^\top,x_3^\top,x_4^\top)^\top, 0\right\}^2\right] \\
&\leq \det(2\pi \Sigma)^{-1/4} \exp\left[\frac{1}{2} d_M\left\{(0^\top,0^\top,x_3^\top,x_4^\top)^\top, (x_1^\top,x_2^\top,x_3^\top,x_4^\top)^\top\right\}^2\right]  \sqrt{f(0,0,x_3,x_4)}.
\end{align*}
Hence,  $|\check\varphi_{\varepsilon}(x_3,x_4)|\le \psi(x_3,x_4)$ where
\begin{equation*}
\psi(x_3,x_4) = \det(2\pi \Sigma)^{-1/4} \kappa_1 \; \kappa_2 \; \sqrt{f(0,0,x_3,x_4)}
\end{equation*}
with 
\begin{align*}
\kappa_1&=\sup_{\substack{\|x_1\|<\varepsilon_1\\ \|x_2\|<\varepsilon_2}} \varphi(x_1,x_2,x_3,x_4)\\
\kappa_2&=
\sup_{\substack{\|x_1\|<\varepsilon_1\\ \|x_2\|<\varepsilon_2}} 
\exp\left[\frac{1}{2} d_M\left\{(0^\top,0^\top,x_3^\top,x_4^\top)^\top, (x_1^\top,x_2^\top,x_3^\top,x_4^\top)^\top\right\}^2\right].
\end{align*} 
By assumption, $\psi$ is integrable, which ends the proof.
\end{proof}

\begin{lemma}\label{lem:proj} Let $X=(X_1^\top,X_2^\top)^\top$ be a zero-mean $d=d_1+d_2$-dimensional Gaussian
random vector with positive definite covariance matrix
$\Sigma$. Let $\bar{X}_1$ denote a random vector distributed according to the conditional distribution of $X_1$ given $X_2=0$. Then, $Y_0=\Sigma^{-1/2}(\bar X_1^\top ,0^\top)^\top \sim \mathcal{N}(0,\Gamma)$ where $\Gamma$ is a projection matrix with rank $d_1$.
\end{lemma}
    
\begin{proof}
Let us denote
$$
\Sigma = 
    \begin{pmatrix}
        \Sigma_{11} & \Sigma_{12 } \\ \Sigma_{21} & \Sigma_{22}
    \end{pmatrix}
$$
where $\Sigma_{11}$ (resp.\ $\Sigma_{22}$) is the $d_1\times d_1$
(resp.\ $d_2 \times d_2$) covariance matrix of $X_1$ (resp.\ $X_2$) so
that $d=d_1+d_2$. It is well known that the conditional distribution of $X_1 | X_2=0$ is $\mathcal{N}(0,\bar \Sigma)$ where $\bar \Sigma= \Sigma_{11} - \Sigma_{12} \Sigma_{22}^{-1}\Sigma_{21}$. Hence, the vector $Y_0$ is a Gaussian random vector with covariance matrix 
$$
\Gamma=\Sigma^{-1/2} \begin{pmatrix} \bar \Sigma & 0\\0&0\end{pmatrix} \Sigma^{-1/2}.
$$
Moreover, since the top-left block of $\Sigma^{-1}$ is $\bar \Sigma^{-1}$, one can check that 
$$
\begin{pmatrix} \bar \Sigma & 0\\0&0\end{pmatrix} \Sigma^{-1} 
\begin{pmatrix} \bar \Sigma & 0\\0&0\end{pmatrix}  = \begin{pmatrix} \bar \Sigma & 0\\0&0\end{pmatrix} 
$$
whereby we deduce that $\Gamma$ is idempotent in addition to being symmetric. It is therefore a  projection matrix with rank $d_1$.

\end{proof}

\begin{lemma}\label{lem:hermite}
The two following statements hold.\\
(i) Let $p\geq 1$ and $Y\sim N(0,1)$. There exists $\kappa>0$ such that for all $\sigma>0$ and $n\ge 1$ 
\begin{equation}
\label{eq:EHn2p}
\frac{\E \left\{ H_n(\sigma Y)^{2p} \right\}^{1/p}}{n!}
\le \kappa\, (2p)^{4n}(1\vee \sigma)^{2n}.
\end{equation}
(ii) Let $\Gamma$ be an $\ell \times \ell$ projection matrix and $Y\sim N(0,I_\ell)$. There exists $\kappa>0$ such that for all multi-index $a \in \mathbb N^\ell$
\begin{equation}
\label{eq:Ha2}
{\E \left\{ H_{\otimes a}(\Gamma Y)^2  \right\}} \le \kappa (2\ell)^{4|a|} \, a!.
\end{equation}
\end{lemma}

\begin{proof}
(i) For any $n\ge 1$, $\E\{H_n(Y)^{2p}\}<\infty$ and, from \cite[Theorem 2.1(ii)]{larsson2002p}, we have as $n\to \infty$ 
\begin{equation*}
\|H_n\|_{2p}:= \E\{H_n(Y)^{2p}\}^{1/2p} = \frac{c_{2p}}{n^{1/4}} (2p-1)^{n/2}\sqrt{n!} \{1+O(1/n)\}
\end{equation*}
for some explicit constant $c_{2p}$. Therefore, there exists $\kappa>0$ such that the following rough upper bound holds for any $n\ge 0$,
\begin{equation}
\label{eq:normHn2p}
\|H_n\|_{2p} \le \kappa\,  (2p)^{n/2} \sqrt{n!}.
\end{equation}
The expansion of scaled Hermite polynomials \cite[Equation (4.6.33)]{ismail2005classical} states that for any $x\in\R$,
\begin{equation*}
H_n(\sigma x) = \sum_{k=0}^{\lfloor n/2 \rfloor} \sqrt{n!} (1-\sigma^2)^k \sigma^{n-2k} b_n(k) H_{n-2k}(x) \quad
\text{ with } \quad 
b_n(k) = \frac{(-1)^k \sqrt{n!}}{k! (n-2k)!}.
\end{equation*}
Using Hölder's inequality, we have the following (again rough) upper bound
\begin{equation*}
H_n(\sigma x)^{2p} \le (1\vee \sigma)^{2np} (n!)^p \, {(\lfloor n/2 \rfloor)^{2p-1}} \sum_{k=0}^{\lfloor n/2 \rfloor} b_n(k)^{2p} H_{n-2k}(x)^{2p}.
\end{equation*}
{Noticing that $(\lfloor n/2 \rfloor)^{2p-1} \leq n^{2p}$} and using \eqref{eq:normHn2p}, we obtain that for any $n\ge 1$,
\begin{align*}
\mathcal E:=\frac{\E \left\{ H_n(\sigma Y)^{2p} \right\}^{1/p}}{n!} &\le (1\vee \sigma)^{2n} n^2 
\left[ 
\sum_{k=0}^{\lfloor n/2 \rfloor} \left\{
b_n(k) \|H_{n-2k}\|_{2p}
\right\}^{2p}
\right]^{1/p} \\
& \le \kappa (1\vee \sigma)^{2n} n^2 
\left[ 
\sum_{k=0}^{\lfloor n/2 \rfloor} \left\{
\frac{\sqrt{n!}}{k! \sqrt{(n-2k)!}} (2p)^{\frac{n-2k}{2}}
\right\}^{2p}
\right]^{1/p} \\
& \le \kappa (1\vee \sigma)^{2n} n^2 (2p)^n 
\left[ 
\sum_{k=0}^{\lfloor n/2 \rfloor} \left\{
\frac{\sqrt{2k!}}{k!} {n \choose {2k}}^{1/2}
\right\}^{2p}
\right]^{1/p} .
\end{align*}
On the one hand, using known lower and upper bounds derived from Stirling formula for factorial numbers, we have for any $k=0,\dots,\lfloor n/2\rfloor$
\begin{align*}
\frac{\sqrt{2k!}}{k!}\le \left\{ \frac{(2k)^{2k+1}}{e^{2k-1}} \right\}^{1/2} \times \frac{e^{k-1}}{k^k} = \frac{2^{k+1/2}}{\sqrt{e}} \sqrt{k} \le 2^{n/2-1/2} \sqrt{n}.
\end{align*}
On the other hand, for any $k=0,\dots,\lfloor n/2\rfloor$ using an upper bound for the middle binomial coefficient
\begin{equation*}
{n \choose {2k}} \le {n \choose {\lfloor n/2 \rfloor}} \le \frac{4^n}{\sqrt{\pi n}}.
\end{equation*}
All in all, we have
\begin{equation*}
 \frac{\sqrt{2k!}}{k!} {{{n} \choose {k}}}^{1/2} \le 
 \frac{2^{n/2}}{\sqrt{2}} \sqrt n \frac{2^n}{(\pi n)^{1/4}} = \frac{2^{3n/2}}{\sqrt{2} \pi^{1/4}} n^{1/4}
\end{equation*}
whereby we deduce that
\begin{equation*}
\mathcal E \le \frac{\kappa}{2\sqrt{\pi}} (1\vee \sigma)^{2n}
(2p)^n 2^{3n} n^{3}
\end{equation*}
and finally the result.\\
(ii) First, the variance of the coordinate $(\Gamma Y)_i$ is $(\Gamma \Gamma^\top)_{ii}$. Since $\Gamma$ is a projection matrix, we have $(\Gamma \Gamma^\top)_{ii} = \Gamma_{ii}$ and $\Gamma_{ii} \leq 1$. Using Hölder's inequality and (i) with $\sigma \leq 1$, we have
\begin{equation*}
    \frac{1}{a!}{\E \left\{ H_{\otimes a}(\Gamma Y)^2  \right\}} \le 
    \prod_{i=1}^{\ell} 
    \frac{\E \left[
    H_{a_i} \left\{
    (\Gamma Y)_i
    \right\}^{2\ell}
    \right]^{1/\ell}}{a_i!} \leq 
    \kappa^{\ell} \; \prod_{i=1}^\ell 
    (2\ell)^{4 a_i} = \kappa^{\ell} \; (2\ell)^{4 |a|}.
\end{equation*}
\end{proof}

\subsection{Proof of Proposition~\ref{prop:chaos}}

\begin{proof}
(i) Since $t \mapsto F_\varepsilon\{\check X(t)\}=
|\det\{X^{\prime\prime}(t)\}| \,\iota_\L\{X^{\prime\prime}(t)\}
\delta_\varepsilon \{\|X^\prime(t)\|\}$ is continuous and has finite
second moment, ${\Phi}_{1,n,\varepsilon}$ has finite second
moment. So, we can consider the  Hermite expansion of
$F_\varepsilon\{\check X(t)\}$ which leads to the one of
$\cent{\Phi}_{1,n,\varepsilon}= \Phi_{1,n,\varepsilon} -
\E(\Phi_{1,n,\varepsilon})$ as 
\[
    \cent{\Phi}_{1,n,\varepsilon} = \sum_{q\ge 1} \sum_{\substack{a\in \mathbb N^D \\ |m|=q}} s_{q,\varepsilon} \quad \text{ with } s_{q,\varepsilon}=
    d_{a,\varepsilon} \int_{W_n} \phi_{1,n}(t)H_{\otimes a} \{\check Y(t)\}  \dd t
\]
and
\begin{align*}
d_{a,\varepsilon}&= \frac{1}{a!} \E 
\left[
H_{\otimes a} \{\check Y(0)\} F_\varepsilon\{\check X(0)\}
\right] = \frac{1}{a!} \E 
\left[
H_{\otimes a} \{\Sigma(0)^{-1/2} \check X(0)\} F_\varepsilon\{\check X(0)\}
\right].
\end{align*}
Since Hermite expansion is a particular case of chaos expansion, 
 $s_{q,\varepsilon}$ is reffered to as the $q$th chaos component of~$\cent{\Phi}_{1,n,\varepsilon}$.
We apply Lemma~\ref{lem:coeffhermite} (in the particular case where
$d_2=d_4=0$) to the vector $\check X(0)$ and
$\varphi(x_1,x_3)=H_{\otimes a}\{\Sigma(0)^{-1/2}
(x_1^\top,x_3^\top)^\top\} |\det(x_3)| \iota_\L(x_3)$ (recall, $x_3$
is a dummy variable corresponding to the matrix $X''(t)$). First, $\check X(0)$ is non-degenerate by~\ref{C:nondegeneracy}. Second, $\varphi$ is continuous in a neighborhood of $(0,x_3)$ for any $x_3$ and $\varphi(0,x_3)$ is bounded by a polynomial function of $x_3$. Therefore, as $\varepsilon\to 0$, $d_{a,\varepsilon} \to d_a$ for any $a\in \mathbb N^D$ where $d_a$ is given by
\[
    d_a=\frac{1}{a!}
f_{X^\prime(0)}(0) \; \times \; 
\E 
\left(
H_{\otimes a} \{\Sigma(0)^{-1/2} \check X_0(0) \} 
\,|\det\{X^{\prime\prime}(0)\}| \,
\iota_\L \{X^{\prime\prime}(0)\}  \, \middle| \, X^\prime(0) =0
\right) 
\]
which yields~\eqref{eq:da} by independence of $X^\prime(0)$ and
$X^{\prime\prime}(0)$ and since $\Var\{X^\prime(0)\}=\lambda_2
I_d$ where $\lambda_2$ is the second spectral moment. The following argument is an adaptation of~\cite[Proof of Proposition 1.3.]{estrade2016central}. Let 
\begin{equation*}
    \tilde{\Phi}_{1,n} = \sum_{q\ge 1} \sum_{\substack{a\in \mathbb N^D \\ |a|=q}}
    d_{a} \int_{W_n} \phi_{1,n}(t) H_{\otimes a} \{\check Y(t)\}  \dd t
\end{equation*}
and for a functional of $\Phi$ admitting an $L^2$-chaos expansion,
denote by $\pi^Q(\Phi)$ and $\pi_Q(\Phi)$ the projection onto the
first $Q$ chaos components  and the remainder ones. Using, the pointwise convergence of $d_{a,\varepsilon}$ and Fatou's lemma we have for any $Q\ge 1$
\begin{align*}
\E \{ \pi^Q(\tilde \Phi_{1,n})^2\} &= \sum_{q=1}^Q 
\E   \left(\left[
d_{a}\int_{W_n} \phi_{1,n}(t) H_{\otimes a} \{\check Y(t)\}  \dd t  
\right]^2\right)\\
&\le \sum_{q= 1}^{Q} \lim_{\varepsilon \to0}  \E  \left(\left[
d_{a,\varepsilon}\int_{W_n} \phi_{1,n}(t) H_{\otimes a} \{\check Y(t)\}  \dd t  
\right]^2\right)\\
&\le \E \{(\cent{\Phi}_{1,n,\varepsilon})^2\} <\infty
\end{align*}
whereby we deduce that $\E\{(\tilde{\Phi}_{1,n})^2\}<\infty$. Consider the decomposition
\begin{equation*}
    \cent{\Phi}_{1,n} - \tilde{\Phi}_{1,n} = 
    \pi_{Q}\left( \cent{\Phi}_{1,n} - \tilde{\Phi}_{1,n} \right) +
    \pi^{Q}\left( \cent{\Phi}_{1,n} - \cent{\Phi}_{1,n,\varepsilon} \right) + 
    \pi^{Q}\left( \cent{\Phi}_{1,n,\varepsilon} - \tilde{\Phi}_{1,n} \right).
\end{equation*}
First consider limits in the $L^2$ sense as $\varepsilon \to 0$ for fixed $Q$. By Proposition~\ref{prop:representation}(i), the second term tends to 0. By the pointwise convergence of $d_{a,\varepsilon}$ to $d_a$ and dominated convergence, the third term tends to 0. Then, as $Q\to \infty$ the first term tends to 0 since $\cent{\Phi}_{1,n} - \tilde{\Phi}_{1,n}$ has finite second moment. All in all, it gives $\cent{\Phi}_{1,n} = \tilde{\Phi}_{1,n}$ in the $L^2$ sense, which proves~\eqref{eq:chaosh1}.

Finally, if $\L$ is symmetric then for all $a$ such that $|a|$ is an odd number, 
\begin{equation*}
    H_{\otimes a} \{\Sigma(0)^{-1/2} \check X_0(0) \} 
    \,|\det\{X^{\prime\prime}(0)\}| \,
    \iota_\L \{X^{\prime\prime}(0)\} 
\end{equation*}
is an odd function of $X^{\prime\prime}(0)$. Since $X^{\prime\prime}(0)\stackrel{\mathcal{D}}{=} -X^{\prime\prime}(0)$, one easily concludes that $d_a = -d_a = 0$ in that case.\\

\noindent(ii) We proceed in the same way. Let 
\begin{align*}
F_{\varepsilon}\{\check X(\mathbf{t})\}&= 
|\det\{X^{\prime\prime}(t^{(1)})\}| \, \iota_\L\{X^{\prime\prime}(t^{(1)})\} \,
\delta_{\varepsilon} \{\|X^\prime(t^{(1)})\| \} \times \\
&\qquad |\det\{X^{\prime\prime}(t^{(2)})\}| \, \iota_\L\{X^{\prime\prime}(t^{(2)})\} \,
\delta_{\varepsilon} \{\|X^\prime(t^{(2)})\| \}.
\end{align*}
The function ${\mathbf{t}} \mapsto \E [F_{\varepsilon}\{\check X(\mathbf{t})\}^2]<\infty$ is bounded and continuous, so ${\Phi}_{2,n,\varepsilon}$ has finite second moment. Then, we can consider the  Hermite expansion of $F_{\varepsilon}\{\check X(\mathbf{t})\}$ which leads to the one of $\cent{\Phi}_{2,n,\varepsilon}= \Phi_{2,n,\varepsilon} - \E(\Phi_{2,n,\varepsilon})$ as
\begin{align*}
\cent{\Phi}_{2,n,\varepsilon}  =  \sum_{q\ge 1} \sum_{\substack{\mathbf a\in \mathbb N^{2D} \\ |\mathbf a|=q}}
\int_{(W_n)^2} d_{\mathbf a,\varepsilon} (\mathbf{t}) \phi_{2,n}(\mathbf t) H_{\otimes \mathbf a} \{\check Y(\mathbf{t})\} \dd \mathbf t.
\end{align*} 
By stationarity and isotropy, $d_{\mathbf a,\varepsilon} (\mathbf{t}) = d_{\mathbf a,\varepsilon} (0,re_1)=d_{\mathbf a,\varepsilon} (r)$ with $r=\|t^{(1)} - t^{(2)}\|$ and is given by
\begin{align*}
d_{\mathbf a,\varepsilon} (r) &= \frac{1}{\mathbf a!}\E
\left[
H_{\otimes \mathbf a}\{\check Y(0,re_1)\} F_{\varepsilon}\{\check X(0,re_1)\}
\right]       \\
&=  \frac{1}{\mathbf a!}\E
\left[
H_{\otimes \mathbf a}\{\bar \Sigma(r)^{-1/2}\check X(0,re_1)\} F_{\varepsilon}\{\check X(0,re_1)\}
\right].
\end{align*}
We apply Lemma~\ref{lem:coeffhermite} to the vector $\check X(0,re_1)$ and the function 
\begin{align*}
\varphi(x_1,x_2,x_3,x_4)&= H_{\otimes \mathbf a}\{\bar \Sigma(r)^{-1/2} (x_1^\top,\dots,x_4^\top)^\top \} \times |\det(x_3)| \iota_\L(x_3) |\det(x_4)| \iota_\L(x_4).
\end{align*}
First, $\check X(0,re_1)$ is non-degenerate by~\ref{C:nondegeneracy:hessian}. Second, $\varphi$ is a continuous function in a neighborhood of $(0,0,x_3,x_4)$ for any $x_3,x_4$ and such that $\varphi(0,0,x_3,x_4)$ is bounded by a polynomial function of $x_3$ and $x_4$. Therefore, for any $r>0$, as  $\varepsilon\to 0$, $d_{\mathbf a,\varepsilon}(r) \to d_{\mathbf a}(r)$ given by~\eqref{eq:dast}. 
Now, we use the same arguments as in (i). Using Fatou's lemma (applied twice), we  prove that $\E\{(\tilde{\Phi}_{2,n})^2\}<\infty$, which implies that $\cent{\Phi}_{2,n}- \tilde{\Phi}_{2,n}$ has finite variance. Using this, the pointwise convergence of $d_{\mathbf a,\varepsilon}(r)$ as $\varepsilon\to 0$ and Proposition~\ref{prop:representation}(ii), we deduce in the same way that $\cent{\Phi}_{2,n}= \tilde{\Phi}_{2,n}$ in the $L^2$ sense, which leads to~\eqref{eq:chaosh2} after noticing that the same argument as in~(i) still applies to determine the Hermite rank: if $\L$ is symmetric, $d_{\mathbf a}=0$ for any $\mathbf a\in \mathbb N^{2D}$ such that $|\mathbf a|$ is odd.

The continuity of $d_{\mathbf a}$ on $[\eta, R]$ follows the same
lines as the continuity of $g_{\L}$ in Theorem~\ref{thm:pcf}: it suffices to check that $r\mapsto \bar \Sigma(r)$ is continuous on $[\eta, R]$. Finally, to prove~\eqref{ineq:da}, we start by applying a Cauchy-Schwarz inequality, $\{\mathbf a! d_{\mathbf a}(r)\}^2 \le \{f_{V(r)}(0,0)\}^2 \times E_1 \times E_2$ where
\begin{align}
E_1 &=\E
\Bigg(
H_{\otimes \mathbf a}^2 \left\{ \bar \Sigma(r)^{-1/2} \check X_0(0,re_1)\right\} 
\mid \; X^{\prime}(0)=X^{\prime}(re_1)=0
\Bigg)\\   
E_2 &=\E
\Bigg(
\det\{X^{\prime\prime}(0)\}^2 \det\{X^{\prime\prime}(re_1)\}^2 
\mid \; X^{\prime}(0)=X^{\prime}(re_1)=0
\Bigg).
\end{align}
The term $\{f_{V(r)}(0,0)\}^2$ is uniformly bounded on $[\eta, R]$ since the distribution of $V(r)$ is continuous in $r$ and non-degenerate on $[\eta, R]$.
The term $E_2$ is uniformly bounded on $[\eta, R]$ since in particular the conditional distribution of $\{X^{\prime\prime}(0)^\top,X^{\prime\prime}(re_1)^\top  \}^\top$ given $X^\prime(0)=X^\prime(re_1)$ is Gaussian, with conditional mean and covariance continuous in $r$. For the term, $E_1$ we can apply Lemma~\ref{lem:proj}(i) by rearranging the terms with $X_1= \{X^{\prime\prime}(0)^\top,X^{\prime\prime}(re_1)^\top  \}^\top$ and $X_2=\{X^\prime(0)^\top,X^{\prime}(re_1)^\top \}^\top$, $d_1=2D-2d$, $d_2=2d$, and $Y_0 = \bar \Sigma(r)^{-1/2} \check X_0(0,re_1)$. Indeed, the assumption of Lemma~\ref{lem:proj}(i) corresponds to~\ref{C:nondegeneracy:hessian}. Therefore, $E_1=\E\left( H_{\otimes \mathbf a}^2 \left\{ \Gamma \check Y \right\} \right)$ for some projection matrix $\Gamma$ and where $Y\sim N(0,I_{2D})$.
Then, we apply Lemma~\ref{lem:hermite}(ii) whereby we deduce, for $\check Y\sim \mathcal N(0,I_{2D})$ that $E_1\times E_2 \le  \kappa (4D)^{4|\mathbf a|} \, \mathbf a!$ for some $\kappa>0$.
\end{proof}



\section{Proof of Lemma~\ref{lem:gamma} and~Theorem~\ref{thm:variance}} \label{app:var}

\subsection{Proof of Lemma~\ref{lem:gamma}} 

\begin{proof}
(i) Using Mehler's formula \cite[Equation~(8) p.3858]{estrade2016central}, the stationarity and isotropy of $\check Y$, there exists a constant $\kappa=\kappa(a,b,q)$ such that for $r=\|t-s\|$,
\[
\left|
\gamma_{a,b}(t-s)
\right|
\le \kappa   \; \left\| 
\E \left\{ \check Y(0) \check{Y}(re_1)^\top\right\}
 \right\|_{\max}^q.
\]
Now, since by definition $\check Y(t)=\Sigma(0)^{-1/2}\check X(t)$, we obtain
\[
\left|
\gamma_{a,b}(t-s)
\right|
\le \kappa  \;  \left\| \Sigma(0)^{-1/2} \right\|_{\max}^{2q} \; 
 \left\|\Cov\{ \check X(0) , \check X(re_1)\}\right\|_{\max}^q.
\]
The latter covariance involves terms of the form $\partial^\alpha c(re_1)$ for $\alpha\in \mathbb N^D$ with $|\alpha|=1,\dots,4$. Lemma~\ref{lem:isotropic} ensures that $\left\|\Cov\{ \check X(0) , \check X(re_1)\}\right\|_{\max}\le \Xi(r)$ which yields that $\gamma_{a,b}(t-s) \le \kappa \,\Xi(r)^q$.

\noindent (ii) For simplicity of notation, we write $\gamma_{\mathbf a,\mathbf b} (\mathbf{t} \ominus \mathbf{s})$ before proving that it is indeed a function of $\mathbf{t} \ominus \mathbf{s}$ only. Using Mehler's formula, we have,
\begin{align*}
    \gamma_{\mathbf a,\mathbf b} (\mathbf{t} \ominus \mathbf{s}) &= \sum_{\substack{d_{ij}\ge 0 \\\sum_{i} d_{ij}=a_j, \sum_j d_{ij}=b_i}} \mathbf a! \mathbf b ! \prod_{i,j=1}^{2D} 
    \frac{\E\left\{ \check Y(\mathbf{t})_{i} \check Y(\mathbf{s})_{j} \right\}^{d_{ij}}}{d_{ij}!} \\
    \E\left\{ \check Y(\mathbf{t})_{i} \check Y(\mathbf{s})_{j} \right\}^{d_{ij}}&=\sum_{k_1,l_1=1}^{2D} \dots \sum_{k_{d_{ij}},l_{d_{ij}}=1}^{2D}
    \prod_{m=1}^{d_{ij}} \Sigma(\mathbf{t})^{-1/2}_{i k_m}\Sigma(\mathbf{s})^{-1/2}_{j l_m}  
    \E \left\{ \check X(\mathbf{t})_{k_m} \check X(\mathbf{s})_{l_m}\right\}.
\end{align*}
By stationarity of $\bX$, it is indeed a function of $\mathbf{t} \ominus \mathbf{s}$. Now, using~\ref{C:integrability}, we have for any $i,j,k_m,l_m \in \{1,\dots,2D\}$ with $m=1,\dots,d_{ij}$, for some generic constant $\kappa$,
\begin{align}
\prod_{m=1}^{d_{ij}} | \Sigma(\mathbf{t})^{-1/2}_{ik_m}| &\le \kappa  \|\Sigma(\mathbf{t})^{-1/2}\|^{d_{ij}}_{\mathrm{max}} \nonumber \\
\prod_{m=1}^{d_{ij}} | \Sigma(\mathbf{s})^{-1/2}_{jl_m}| &\le \kappa  \|\Sigma(\mathbf{s})^{-1/2}\|^{d_{ij}}_{\mathrm{max}} \nonumber \\
\prod_{m=1}^{d_{ij}} \left| \E \left\{ \check X(\mathbf{t})_{k_m} \check X(\mathbf{s})_{l_m}\right\} \right| & \le \kappa \, \Xi(\|t^{(1)} - s^{(1)}\|)^{\alpha_1}\Xi(\|t^{(1)} - s^{(2)}\|)^{\alpha_2} \nonumber\\
& \hspace{1cm} 
\times \Xi(\|t^{(2)} - s^{(1)}\|)^{\alpha_3}\Xi(\|t^{(2)} - s^{(2)}\|)^{\alpha_4} \label{eq:tmpcov}
\end{align}
for some $\alpha \in \mathbb N^4$ such that $|\alpha|=d_{ij} \leq q$. 
The result is easily deduced since $d_{ij}\le q$ and $\Xi$ is a decreasing function.\\
(iii) The proof combines (i) and~(ii). We omit the details.
\end{proof}

\subsection{Proof of Theorem~\ref{thm:variance}-(i)}
\begin{proof}
By Proposition~\ref{prop:chaos}-(i), the $q$th {chaos} component of $\cent{\Phi}_{1,n}$  is given by
\begin{equation}\label{eq:def:sq:phi1}
    s_q(\phi_{1,n}):=\sum_{a \in \mathbb N^D,|a|=q} d_a \int_{W_n} \phi_{1,n}(t)H_{\otimes a}\{\check Y(t)\}\dd t
\end{equation}
and so $\Var (\cent{\Phi}_{1,n}) = \sum_{q\in \mathbb N_\L^*}
\Var\{s_q(\phi_{1,n})\}$. Let $q\ge 1$ (or 2 if $\L$ is
symmetric)
\begin{align}
\Var\{s_q(\phi_{1,n})\} &= \sum_{\substack{a,b \in \mathbb N^D \\ |a|=|b|=q}} d_a d_{b} \int_{W_n}\int_{W_n}
\phi_{1,n}(t) \phi_{1,n}(s) 
\E \left[
H_{\otimes a}\{\check Y(t)\} H_{\otimes b}\{\check Y(s)\}
\right] \dd s  \dd t\nonumber\\
&= \sum_{\substack{a,b \in \mathbb N^D \\ |a|=|b|=q}} d_a d_{b} \int_{W_n}\int_{W_n}
\phi_{1,n}(t) \phi_{1,n}(s) 
\gamma_{a,b}(t-s)  \dd s \dd t \label{eq:varsqh1}
\end{align}    
by stationarity of $\check Y$. Equation~\eqref{eq:gammaab1} and
condition~\ref{C:integrability} ensure that $\gamma_{a,b} \in
L^1(\R^d)$. 
 Using the notation $\phi_{1,n,W_n}$, the definition and properties of the convolution product (denoted by $*$) and Parseval's identity, we have
\begin{align*}
\Var\{s_q(\phi_{1,n})\} &=\sum_{\substack{a,b \in \mathbb N^D \\ |a|=|b|=q}} d_a d_{b}
\int_{\R^d} \phi_{1,n,W_n}(t) \, \phi_{1,n,W_n} * \gamma_{a,b}(t) \dd t \\
&= \sum_{\substack{a,b \in \mathbb N^D \\ |a|=|b|=q}} d_a d_{b} (2\pi)^{-d} \int_{\R^d}  \overline{\widehat{\phi}_{1,n,W_n}}(\omega) \, \widehat{\phi}_{1,n,W_n}(\omega)  \widehat{\gamma}_{a,b}(\omega) \dd \omega \\
&=n^{-d} \sum_{\substack{a,b \in \mathbb N^D \\ |a|=|b|=q}} d_a d_{b} (2\pi)^{-d} \int_{\R^d}  \left| \, \widehat{\phi}_{1,n,W_n}(\omega/n) \right|^2 \, 
\widehat{\gamma}_{a,b}(\omega/n)  \dd \omega.
\end{align*}
By assumption, the dominated convergence theorem can be applied to the integral term. Combined with Parseval's identity, we finally obtain
\begin{equation*}
 n^{d} \,  \Var\{s_q(\phi_{1,n})\} \to \|\phi_1\|^2_{L^2(\R^d)} \, \sum_{\substack{a,b \in \mathbb N^D \\ |a|=|b|=q}} d_a d_{b} \, \widehat{\gamma}_{a,b}(0).  
\end{equation*}
To rephrase the previous result, the assumption made on $\phi_{1,n}$ allows us to separate the influence of the function $\phi_1$ and the dependence of the random field $\bX$ captured by the term $\widehat \gamma_{a,b}(0)$. In other words, $\Var\{s_q(\phi_{1,n})\} \sim \|\phi_1\|^2_{L^2(\R^d)} \, \Var[s_q\{\mathbf 1(\cdot \in W_n)\}]$  as $n\to \infty$. Recall also that the choice $\phi_{1,n}(t) = \mathbf 1(t\in W_n)$ corresponds to the functional ${\Phi}_{1,n}=N_\L(W_n)$ (in this situation $\phi_1(t)=\mathbf 1(t\in[-1/2,1/2]^d)$ and $\|\phi_1\|_{L^2(\R^d)}=1$). We can therefore benefit from the proof of~\cite[Proposition 2.1]{estrade2016central} (see also~\cite[Theorem 2.1]{azais2024multivariate}) which considers exactly this functional and in particular proves that $\sup_n \sum_{q>Q} n^{-d} \, \Var[s_q\{\mathbf 1(\cdot \in W_n)\}] \to 0$ as $Q\to \infty$. Henceforth, we also have $\sup_n \sum_{q>Q} n^{-d} \, \Var\{s_q(\phi_{1,n})\} \to 0$ as $Q\to \infty$ which yields the result.
\end{proof}

\subsection{Proof of Theorem~\ref{thm:variance}-(ii)}
\label{sec:proof:thm:variance:ii}

\begin{proof}
For all $x, u, v, w\in \mathbb{R}^d$, let $\mathbf{t}_x = (w+v+x,w+v-u+x)$, $\mathbf{s}_x = (v+x,x)$ and $\mathbf{z}_x = \mathbf{t}_x \ominus \mathbf{s}_x $. By construction and stationarity, we observe that
$\gamma_{\mathbf a, \mathbf b} (\mathbf{z}_x) = \gamma_{\mathbf a, \mathbf b} (\mathbf{z}_0)$. From~\eqref{eq:gammaab2}, we have
\begin{align*}
|\gamma_{\mathbf a,\mathbf b}(\mathbf{z}_0)| &\le \kappa \, 
\psi(\|u\|)^{q} \psi(\|v\|)^{q} 
\Xi(\|w\|)^{\alpha_1}    \Xi(\|w+v\|)^{\alpha_2} \Xi(\|w-u\|)^{\alpha_3}\Xi(\|w+v-u\|)^{\alpha_4}
\end{align*}
where $\psi(\| u \|) =\|\bar\Sigma(\|u\|)^{-1/2}\|_{\max}$.
By condition~\ref{C:nondegeneracy:hessian} and continuity of $r \mapsto \bar\Sigma(r)$ on the compact set $[\eta,R]$, $\psi(\|u\|)^{q} \psi(\|v\|)^{q}$ 
is uniformly bounded. In addition, since, for any $w\in \R^d$ such that $\|w\|\ge 3R$,
\begin{equation*}
\min \left( \|w\|, \|w+v\|, \|w-u\|,\|w+v-u\|\right) \ge  \|w\|/3
\end{equation*}
there exists $\kappa$ (depending on $\mathbf a,\mathbf b$, $\eta$ and $R$) such that
\begin{align*}
|\gamma_{\mathbf a,\mathbf b}(\mathbf z_0)| & \leq  \kappa \,  \Xi(\|w\|/3)^q,
\end{align*}
which in turn implies~\eqref{eq:integK2}. Now, we continue with the
same approach as in (i). By Proposition~\ref{prop:chaos}-(ii), the
$q$th chaos  component of $\cent{\Phi}_{2,n}$ is given by 
\begin{equation}\label{eq:def:sq:phi2}
s_q(\phi_{2,n}):=\sum_{\mathbf a \in \mathbb N^{2D},|\mathbf a|=q} 
\int_{(W_n)^2} \phi_{2,n}(\mathbf{t}) d_{\mathbf a}(\mathbf{t}) H_{\otimes \mathbf a}\{\check Y(\mathbf{t})\}\dd \mathbf{t}
\end{equation}
and so $\Var (\cent{\Phi}_{2,n}) = \sum_{q\in \mathbb N_{\L}^*}
\Var\{s_q(\phi_{2,n})\}$. Let $q\ge 1$. Then
\begin{equation*}
\Var\{s_q(\phi_{2,n})\} = \sum_{\substack{\mathbf a, \mathbf b \in \mathbb N^{2D} \\ |\mathbf a|=|\mathbf b|=q}} \int_{(W_n)^2} \int_{(W_n)^2}
\phi_{2,n}(\mathbf{t}) \phi_{2,n}(\mathbf{s})
 d_{\mathbf a}(\mathbf{t})d_{\mathbf b}(\mathbf{s}) 
 \gamma_{\mathbf a,\mathbf b}(\mathbf{t} \ominus \mathbf{s}) \dd \mathbf{t} \dd \mathbf{s}.
\end{equation*}
Remark that the integrand is null except when $\mathbf t,\mathbf s \in  (W_n)^2 \cap D_{\eta,R}$. Consider the change of variables  $u=t^{(1)} - t^{(2)}, v = s^{(1)} - s^{(2)}, w = t^{(1)} - s^{(1)}$ and $x=s^{(2)}$ (for which we have $\gamma_{\mathbf{a},\mathbf b}(\mathbf{t}\ominus \mathbf s)=\gamma_{\mathbf{a},\mathbf b}(\mathbf z_x)=\gamma_{\mathbf{a},\mathbf b}(\mathbf z_0)$).
By the assumption~\eqref{eq:assphi2n} made on $\phi_{2,n}$, we have
\begin{align*}
\Var\{s_q(\phi_{2,n})\} &=  n^{-d}  
\sum_{
\substack{\mathbf a, \mathbf b \in \mathbb N^{2D} \\  
|\mathbf a|=|\mathbf b|=q}
}
\int_{\R^d} \int_{B(0,\eta,R)} \int_{B(0,\eta,R)} 
\phi_2(\|u\|)\phi_2(\|v\|) d_{\mathbf a}(\|u\|) d_{\mathbf b}(\|v\|)  \times \\
&\frac{|W_n| \times |W_n \cap (W_n)_{-v} \cap (W_n)_{-v-w} \cap (W_n)_{u-v-w}|}{|W_n \cap (W_n)_{-u}| \times |W_n \cap (W_n)_{-v}|} \;
\gamma_{\mathbf a, \mathbf b} (\mathbf{z}_0 )
\dd u \dd v \dd w. 
\end{align*}
Using~\eqref{eq:integK2} and the dominated convergence theorem we obtain that as $n\to \infty$,
\begin{align*}
n^d \Var\{s_q(\phi_{2,n})\} \to &
\sum_{
\substack{\mathbf a, \mathbf b \in \mathbb N^{2D} \\  
|\mathbf a|=|\mathbf b|=q}}
\int_{\R^d} \int_{B(0,\eta,R)} \int_{B(0,\eta,R)} 
\phi_2(\|u\|)\phi_2(\|v\|) d_{\mathbf a}(\|u\|) d_{\mathbf b}(\|v\|)  \times \\
& \qquad \qquad \gamma_{\mathbf a, \mathbf b} \{\mathbf{z}_0 \}
\dd u \dd v \dd w.
\end{align*}
Hence,~\eqref{eq:varPhi2} is proved if one proves that $\mathcal S_Q  : =\sup_n \sum_{q>Q} n^d \Var\{s_q(\phi_{2,n})\} \to 0$ as $Q\to \infty$. Without loss of generality, assume that $R\in \mathbb N$ and decompose $(W_n)_{\ominus R} = \cup_{i\in \mathcal I_n} \Delta_i$ where $\mathcal I_n =\mathbb Z^d \cap (W_n)_{\ominus R}$, $\Delta_i=\Delta_0+i$ and $\Delta_0=[-1/2,1/2]^d$. Using~\eqref{eq:assphi2n}, we can check that for $Q\ge  1$, 
\begin{align*}
\mathcal S_Q & = \sup_n  \; n^{-d} \,\Var \left(
\sum_{i\in \mathcal I_n} \sum_{q>Q} s_{i,q}
\right) 
\end{align*}
where 
\begin{equation*}
s_{i,q} = 
\sum_{\substack{ \mathbf a\in \mathbb N^{2D} \\ |\mathbf a|=q }}
\int_{\mathbf{\Delta}_{i,\eta,R}} \phi_2(\|t^{(1)} - t^{(2)}\|) 
d_{\mathbf a}(\mathbf t) 
H_{\otimes \mathbf a} \{\check Y(\mathbf{t})\} \dd \mathbf{t}
\end{equation*}
and $\mathbf{\Delta}_{i,\eta,R} = \{\mathbf t\in \R^d: t^{(1)} \in \Delta_i, t^{(2)} \in B\{t^{(1)},\eta,R\}\}$. Let $\cent{\psi}_{2,i}= \psi_{2,i} - \E (\psi_{2,i})$, where 
\begin{equation*}
\psi_{2,i} = \sum^{\neq}_{ \substack
    t^{(1)} \in \bY_\L\cap \Delta_i, \\ t^{(2)} \in \bY_{\L} \cap (\Delta_{i})_{\oplus R}
} \phi_2(\|t^{(1)}-t^{(2)}\|).
\end{equation*}
Note that conditions \ref{C:general}[5]-\ref{C:nondegeneracy}[4] ensure that $\E(\psi_{2,i}^2)<\infty$ for any $i$ and that using the Hermite expansion developed in Proposition~\ref{prop:chaos}(ii) $\cent{\psi}_{2,i} = \sum_{q\ge 1} s_{i,q}$. 
Let $\pi_Q(\cent{\psi}_{2,i})=\sum_{q>Q} s_{i,q}$ denotes the
projection operator onto all chaos {components of} order $q> Q$. We have
\begin{equation*}
\mathcal S_Q = \sup_n \; n^{-d} \, \Var 
\bigg\{ 
\sum_{i\in \mathcal I_n} \pi_Q ( \cent{\psi}_{2,i})
\bigg\} = \sup_n \; ( V_{1,Q} +  V_{2,Q} )
\end{equation*}
where we decompose the variance for some large $\ell >2R+\sqrt{d}$ (to be set later) as
\begin{align*}
V_{1,Q} &= 
n^{-d} 
\sum_{\substack{i,j \in \mathcal{I}_n \\ \|j -i\|\le 2R+\sqrt d +\ell}} 
\E\left\{ \pi_Q ( \cent{\psi}_{2,i})\pi_Q ( \cent{\psi}_{2,j}) \right\}\\
V_{2,Q} &= 
n^{-d} 
\sum_{\substack{i,j \in \mathcal{I}_n \\ \|j -i\|> 2R+\sqrt{d} +\ell}} 
\E\left\{ \pi_Q ( \cent{\psi}_{2,i})\pi_Q ( \cent{\psi}_{2,j}) \right\}.
\end{align*}
Since the variables $\cent{\psi}_{2,i}$ are identically distributed, there exists $\kappa$ independent of $n$ and $Q$ such that the first term satisfies (using Cauchy-Schwarz inequality) 
\begin{equation*}
V_{1,Q} \le \sqrt d \{2R+\sqrt d + \ell\}^d \frac{|\mathcal I_n|}{|W_n|} 
\E \left\{ \pi_Q(\cent{\psi}_{2,0})^2\right\} \le \kappa \E \left\{ \pi_Q(\cent{\psi}_{2,0})^2\right\} =o(1)
\end{equation*}
as $Q\to \infty$ since  $\E \{ (\cent{\psi}_{2,0})^2\}<\infty$. Now, let $i,j \in \mathcal I_n$ such that $\|j-i\|> 2R+\sqrt d + \ell$, we have 
\begin{equation*}
\mathcal E_{i,j} :=\E\left\{ \pi_Q ( \cent{\psi}_{2,i})\pi_Q ( \cent{\psi}_{2,j}) \right\} 
= \int_{\mathbf{\Delta}_{i,\eta,R}}
\int_{\mathbf{\Delta}_{j,\eta,R}} \sum_{q>Q} \E
\left[
F_{q,\mathbf{t}} \{ \check{Y}(\mathbf{t})\}  F_{q,\mathbf{s}} \{ \check{Y}(\mathbf{s})\}
\right] \dd \mathbf{s} \dd \mathbf{t}
\end{equation*}
where for any $\mathbf{t} \in W_n$ and $\check y \in \mathbb{R}^{2D}$,
\begin{equation*}
F_{q,\mathbf{t}} \{ \check y\} = \sum_{\substack{ \mathbf a \in \mathbb N^{2D} \\ |\mathbf a|=q}} \phi_2(\|t^{(1)}-t^{(2)}\|) d_{\mathbf a} (\mathbf{t}) H_{\otimes \mathbf a} \{ \check y\}.
\end{equation*}
Let $\Gamma = \Var [ \{\check{Y}(\mathbf{t})^\top , \check{Y}(\mathbf{s})^\top  \}^\top  ]$. The components of $\Gamma$ are of the form $\gamma_{\mathbf{a},\mathbf{b}}(\mathbf{t} \ominus \mathbf{s})$ with $|\mathbf{a}| = |\mathbf{b}| = 1$.
Using Equation~\eqref{eq:gammaab2}, we have
for any $\mathbf{t} \in \mathbf{\Delta}_{i,\eta,R}$ and $\mathbf{s} \in \mathbf{\Delta}_{j,\eta,R}$
$$\rho^\star:=\|I_{4D}- \Gamma\|_{\infty} \le 
\left\{
\sup_{\eta \leq r \leq R} \|\bar\Sigma(r)^{-1/2}\|_{\max}^2 
\right\}
\,  \Xi\{\delta(\mathbf{t}, \mathbf{s})\}.$$ 
Condition~\ref{C:nondegeneracy:hessian} and the continuity of $r \mapsto \bar\Sigma(r)$ imply the finiteness of the $\sup$ above. We can check that $\delta(\mathbf{t}, \mathbf{s}) \ge \|j-i\| - 2 \sup_{t \in (\Delta_0)_{\oplus R}}\|t\|\ge \|j-i\|-2R-\sqrt{d}$. So $\rho^\star \le  \kappa \Xi \{\|j-i\|-2R -\sqrt{d} \}$. Since $\Xi$ is a decreasing function which tends to zero, one can set $\ell$ large enough such that $\rho^\star \le (1/3-\eta)/4D$ for some $\eta>0$. Thus \cite[Corollary 2.1]{soulier2001moment}, which extends \cite[Lemma~1]{arcones1994limit}, can be applied which yields  
\begin{equation*}
\E   \left[
F_{q,\mathbf{t}} \{ \check{Y}(\mathbf{t})\}  F_{q,\mathbf{s}} \{ \check{Y}(\mathbf{s})\}
\right] \le   \E\{F_{q,\mathbf{t}}(\check Y)^2\}^{1/2} \E\{F_{q,\mathbf{s}}(\check Y)^2\}^{1/2} (2D \rho^\star)^q 
\end{equation*}
where $\check Y \sim \mathcal N(0,I_{2D})$. For any $\mathbf{t} \in \mathbb{R}^{2d}$,
\begin{align*}
\E\{F_{q,\mathbf{t}}(\check Y)^2\}  &=\sum_{\substack{ \mathbf a \in \mathbb N^{2D} \\ |\mathbf a|=q}} \phi_2^2(\|t^{(1)}-t^{(2)}\|) d_{\mathbf a}^2(\mathbf{t}) \, \mathbf a!.
\end{align*}
Using~\eqref{ineq:da} we obtain that there exists a constant $\kappa$ depending only on a compact set containing $\{\|t^{(1)} - t^{(2)}\| : \mathbf{t} \in \Delta_0 \times (\Delta_0)_{\oplus R} \}$ such that for any $i$ and $\mathbf{t} \in \Delta_i \times (\Delta_i)_{\oplus R}$,
\begin{align*}
\E\{F_{q,\mathbf{t}}(\check Y)^2\}  &\le \kappa \sum_{\substack{\mathbf a \in \mathbb N^{2D} \\ |\mathbf a|=q}} 1  \le \kappa q^{2D} (4D)^{4q}.
\end{align*}
Finally, for some generic constant $\kappa$ (which may vary from line to line) independent of $i,j,q,n$ , we have by denoting $c=2D \times \sup_{\eta\leq r\leq R} \|\bar\Sigma(r)^{-1/2}\|_{\max}^2 $,
\begin{align*}
V_{2,Q} &= 
n^{-d} 
\sum_{\substack{i,j \in \mathcal{I}_n \\ \|j -i\|> 2R+\sqrt d +\ell}} 
\mathcal E_{i,j} \\
& \leq \kappa n^{-d} 
\sum_{\substack{i,j \in \mathcal{I}_n \\ \|j -i\|> 2R+\sqrt d +\ell}} \sum_{q>Q} q^{2D} (4D)^{4q}\,c^q \,\Xi(\|j-i\|-2R-\sqrt d)^q \\
& \leq (4Dc)\kappa \;  n^{-d} \sum_{\substack{i,j \in \mathcal{I}_n \\ \|j -i\|> \ell}}  \Xi(\|j-i\|) \sum_{q>Q} q^{2D} \{(4D c) \,\Xi(\ell)\}^{q-1} \\
&\le \kappa 
\left\{ n^{-d}\int_{W_n}\int_{W_n \cap B(t^{(1)},\ell-1)^c}
\Xi(\|t^{(1)} - t^{(2)}\|) \dd t^{(2)} \dd t^{(1)}  \right\}
\left[
\sum_{q>Q} q^{2D} \{(4Dc) \Xi(\ell)\}^{q-1} 
\right]\\
&\le \kappa \left\{\int_{\ell-1}^{\infty} r^{d-1} \Xi(r)\dd r\right\} 
\left[ \sum_{q>Q} q^{2D} \{(4Dc) \Xi(\ell)\}^{q-1} \right]
\end{align*}
for $n$ large enough, where we use condition~\ref{C:integrability}, a polar coordinates change and dominated convergence theorem for the last two lines. Condition~\ref{C:integrability} ensures that the first term is bounded while by choosing $\ell$ large enough we can ensure that $u=(4Dc)\,\Xi(\ell)<1$. Thus, the series with terms $q^{2D} u^q$ converges, which yields that $\sup_n V_{2,Q}=o(1)$ as $Q\to \infty$ and ends the proof. 
\end{proof}

\subsection{Proof of Theorem~\ref{thm:variance}(iii)-(iv)}

\begin{proof}
(iii) Let $u,v,w\in \R^d$ and $\mathbf z_w =
(u+v+w)\ominus(v+w,w)$. Observe that $\gamma_{a,\mathbf b}(\mathbf z_w) = \gamma_{a,\mathbf b}(\mathbf z_0)$. Now, from~\eqref{eq:gammaab3}
\begin{equation*}
|\gamma_{a,\mathbf b}(\mathbf z_0)| \le \kappa \|\Sigma(0)^{-1/2}\|^q_{\max} \; 
\|\bar \Sigma(\|v\|)^{-1/2}\|^q_{\max} \;
\Xi[ \delta\{u+v, (v,0)\}].
\end{equation*}
So for any $v \in B(0,\eta,R)$ and $u\in \R^d$ such that $\|u\|\ge R$,
there exists $\kappa$ (depending on $a,\mathbf b, \eta, R$) such that
$|\gamma_{a,\mathbf b}(\mathbf z_0)| \le \kappa \, \Xi(\|u\|/2)^q$
which in turn implies~\eqref{eq:integCov}. Now, we continue with the
same approach as in (i)-(ii) and omit a few details. First,
\begin{align*}
\Cov( \cent{\Phi}_{1,n} , \cent{\Phi}_{2,n}) &= \sum_{q\ge 1} \Cov \left\{s_q(\phi_{1,n}), s_q(\phi_{2,n})  \right\} \\
\Cov \left\{s_q(\phi_{1,n}), s_q(\phi_{2,n})  \right\} &= \sum_{\substack{a\in \mathbb N^D, \mathbf b \in \mathbb N^{2D} \\ |a|=|\mathbf b|=q}}  d_a \int_{W_n} \int_{(W_n)^2} \phi_{1,n}(t) \phi_{2,n}(\mathbf s)d_{\mathbf{b}} (\mathbf s) 
\gamma_{a,\mathbf b} (t\ominus \mathbf s)  \dd \mathbf s \dd t.
\end{align*}
Second, using the assumptions made on $\phi_{1,n}$ and $\phi_{2,n}$ and the change of variables $u=t-s^{(1)}$, $v=s^{(1)}-s^{(2)}$ and $s^{(2)}=w$ (for which we have $\gamma_{a,\mathbf b}(t\ominus \mathbf s)=\gamma_{a,\mathbf b}(\mathbf z_w)=\gamma_{a,\mathbf b}(\mathbf z_0)$) we have as $n \to \infty$,
\begin{align*}
\Cov \Big\{s_q(\phi_{1,n}), &s_q(\phi_{2,n})  \Big\} \\
&\sim n^{-d} d_a 
\int_{\R^d}  \int_{ (W_n)^2 } \phi_1(t) \phi_{2,n}(\mathbf{s}) d_{\mathbf b}(\mathbf s) \gamma_{a, \mathbf b}(nt\ominus \mathbf s) \dd \mathbf s \dd t \\
&\sim n^{-2d} d_a \phi_1(0) \int_{\R^d}  \int_{ (W_n)^2 }\phi_{2,n}(\mathbf{s}) d_{\mathbf b}(\mathbf s) \gamma_{a, \mathbf b}(t\ominus \mathbf s) \dd \mathbf s \dd t \\
&= n^{-2d} d_a \phi_1(0)  \int_{\R^d} \int_{W_n} \int_{B(0,\eta,R)} \frac{\phi_2(\|v\|)}{|W_n \cap (W_n)_{-v}|} d_{\mathbf b}(\|v\|) \gamma_{a,\mathbf b} (\mathbf z_w) \dd v \dd w \dd u \\
& \sim n^{-d} d_a \phi_1(0) \int_{\R^d} \int_{B(0,\eta,R)} \phi_2(\|v\|)d_{\mathbf b}(\|v\|) \gamma_{a,\mathbf b} (\mathbf z_0) \dd v \dd u. 
\end{align*}
Hence, from~\eqref{eq:integCov}, the result is proved is one proves that as $Q\to \infty$, 
\begin{equation*}
    \mathcal S_Q  : =\sup_n \sum_{q>Q} n^d \Cov\{s_q(\phi_{1,n}),s_q(\phi_{2,n})\} \to 0
\end{equation*}
which,  from the proofs of (i)-(ii), follows from the Cauchy-Schwarz inequality.\\
(iv) Observe that $\phi_1:=\sum_{i=1}^{m_1}\beta_i \phi_{1,n}^1$ and $\phi_2:=\sum_{i=m_1+1}^{m_1+m_2}\beta_i \phi_{2,n}^2$ satisfy respectively~\eqref{eq:assphi1hat} and~\eqref{eq:assphi2n}. Thus,~\eqref{eq:linearcombination} ensues from Theorem~\ref{thm:variance}(i)-(iii). 
\end{proof}

We end this appendix by proving~\eqref{eq:ratiovar} for any stationary
spatial point process with intensity $\rho$ and pair correlation
function $g$, such that $g-1\in L^1(\R^d)$. Using the Campbell
theorem, properties of the Fourier transform and the dominated convergence theorem, we have
\begin{align*}
n^d \, \Var\left( \cent{\Phi}_{1,n} \right) &= n^d \left[
\rho \int \phi_{1,n,W_n}(t)^2 \dd t + \rho^2\int \int \phi_{1,n}(t)\phi_{1,n}(s) \{ g(t-s)-1\}\dd t \dd s
\right] \\
&= n^d (2\pi)^{-d} 
\left\{
\rho \int  |\hat \phi_{1,n,W_n}(\omega)|^2 \dd \omega + 
\rho^2 \int |\hat \phi_{1,n,W_n}(\omega)|^2 (\widehat{g-1})(\omega) \dd \omega 
\right\}\\
&=  (2\pi)^{-d} 
\left\{
\rho \int  |\hat\phi_{1,n,W_n}(\omega/n)|^2 \dd \omega + 
\rho^2 \int |\hat\phi_{1,n,W_n}(\omega/n)|^2 (\widehat{g-1})(\omega/n) \dd \omega 
\right\}\\
&\sim  \|\phi\|^2 \left\{
\rho + \rho^2 (\widehat{g-1})(0)
\right\}  \sim \|\phi\|^2 n^d \Var\{N_\L(W_n)\}.
\end{align*}


\section{Multiple Wiener-Itô integrals and proof of Theorem~\ref{thm:clt}}
\label{app:CLT}

To prove item (i) of Theorem~\ref{thm:clt}, it suffices that for any $\beta = (\beta_1,\beta_2)^\top \in \mathbb{R}^2$, $n^{d/2} \beta^\top \cent{\Phi}_{n}\; \stackrel{\mathcal{D}}{\longrightarrow} \mathcal{N} (0, \beta^\top \Sigma(\phi_1,\phi_2) \beta)$. Hence, from here onwards, let us fix some $\beta \in \mathbb{R}^2$.

The proof follows~\cite{estrade2016central}, see also~\cite{nicolaescu2017clt,azais2024multivariate}. In particular, the proof of the case $\beta_2=0$ is merely a transcription of their proof into our context. The main ingredient of the proof is \cite[Theorem 11.8.1.]{peccati2011wiener} which is adapted in our framework as Theorem~\ref{thm:cltMWI} below. The following subsection is dedicated to the preliminaries needed to state this result.

We remind that, on one hand, $D = d + d(d+1)/2$ is the dimension of $\check X(t)$ and $\check Y(t)$, and that $s_{q}(\phi_{1,n})$, defined in~\eqref{eq:def:sq:phi1} stands for the $q$th chaotic component of $\cent{\Phi}_{1,n}$.
On the other hand, the dimension of $\check X(\mathbf{t})$ and $\check Y(\mathbf{t})$ is $2D$, and $s_{q}(\phi_{2,n})$, defined in~\eqref{eq:def:sq:phi2}, stands for the $q$th chaos component of $\cent{\Phi}_{2,n}$. All in all, 
\begin{equation*}
    s_q = \beta_1 \sum_{\substack{a\in \mathbb N^D \\ |a|=q }} 
    d_a \int_{W_n} \phi_{1,n}(t) H_{\otimes a}\{\check Y(t)\} \dd t +
    \beta_2 \sum_{\substack{\mathbf{a}\in \mathbb N^{2D} \\ |\mathbf{a}|=q }} 
    \int_{(W_n)^2} \phi_{2,n}(\mathbf{t}) d_\mathbf{a}(\mathbf{t}) H_{\otimes \mathbf{a}}\{\check Y(\mathbf{t})\} \dd \mathbf{t}
\end{equation*}
stands for the $q$th chaos component of $\cent{\Phi}_{\beta, n} := \beta_1\cent{\Phi}_{1,n} + \beta_2\cent{\Phi}_{2,n}$.

\subsection{Multiple Wiener-Itô Integrals (MWI)}

Let $\mathcal{H}$ be the set of complex-valued Hermitian square integrable functions w.r.t. Lebesgue measure $\dd \omega$ in $\mathbb{R}^d$, that is, $\mathcal{H} = \{\psi: \mathbb{R}^d\to \mathbb{C}: \psi(-\omega) =\overline{\psi(\omega)}, \|\psi\|_{\mathcal{H}}^2=\int_{\R^d}|\psi(\omega)|^2 \dd\omega<\infty\}$. Let $\hat{B}$ be a complex Hermitian Brownian measure on $\R^d$. The (real-valued) Wiener-Itô integral w.r.t. $\hat{B}$, denoted by 
\begin{equation*}
I_1(\psi) = \int_{\R^d}\psi(\omega) \dd\hat{B}(\omega),
\end{equation*}
is an isometry from $\mathcal{H}$ into $L^2(\Omega)$. That is, for $\psi_1,\psi_2\in \mathcal{H}$, 
we have
\begin{equation} \label{eq:MWI:isometry}
 \E \big\{ I_1(\psi_1)I_1(\psi_2)\big\} = \int_{\R^d}\psi_1(\omega)\overline{\psi_2(\omega)}\dd\omega = 
\int_{\R^d}\psi_1(\omega)\psi_2(-\omega)\dd\omega 
 = \left< \psi_1,\psi_2 \right>_{\mathcal{H}}.
\end{equation}

\begin{proposition}\label{prop:Y:as:MWI}
Assume \ref{C:general}$[2]$ and \ref{C:nondegeneracy:hessian}. Each coordinate of $\check{Y}$ can be represented as a MWI.\\
(i) For all $t\in \mathbb{R}^d$, $i=1,\dots,D$, there exists an explicit Hermitian orthonormal kernel $\psi_{t,i}\in \mathcal{H}$ such that $\check Y_i(t) = I_1(\psi_{t,i})$. 

In particular, $\left< \psi_{t,i},\psi_{s,j} \right>_{\mathcal{H}} = \mathbb{E}\left\{ \check Y_i(t) \check Y_j(s) \right\}$ which we denote by $\gamma_{ij}(t-s)$.\\
(ii) For all $\mathbf{t}\in \mathbb{R}^{2d}$, $i=1,\dots,2D$, there exists an explicit Hermitian orthonormal kernel $\psi_{\mathbf{t},i}\in \mathcal{H}$ such that $\check Y_i(\mathbf{t}) = I_1(\psi_{\mathbf{t},i})$. 

In particular, $\left< \psi_{\mathbf{t},i},\psi_{\mathbf{s},j} \right>_{\mathcal{H}} = \mathbb{E}\left\{ \check Y_i(\mathbf{t}) \check Y_j(\mathbf{s}) \right\}$ which we denote by $\gamma_{ij}(\mathbf{t} \ominus \mathbf{s})$.\\
(iii) For all $t\in \mathbb{R}^d$ and $\mathbf{s} \in \mathbb{R}^{2d}$, $\left< \psi_{t,i},\psi_{\mathbf{s},j} \right>_{\mathcal{H}} = \mathbb{E}\left\{ \check Y_i(t) \check Y_j(\mathbf{s}) \right\}$ which we denote by $\gamma_{ij}(t \ominus \mathbf{s})$.

Furthermore, let $\eta>0$. Under \ref{C:integrability}, for all $t,s\in \mathbb{R}^d$ and $\mathbf{t}, \mathbf s\in D_{\eta,R}$, the following upper bounds hold (where $\kappa$ depends on $\eta$ in the last two equations and the notation $\delta$ is introduced in Lemma~\ref{lem:gamma}):
\begin{align}
    |\gamma_{ij}(t-s)| &\leq \kappa \Xi(\|t-s\|), \quad i,j\in \{1,\dots,D\}, \label{eq:gammaij1} \\
    |\gamma_{ij}(\mathbf{t} \ominus \mathbf{s})| &\leq \kappa \Xi\{\delta(\mathbf{t}, \mathbf{s})\}, \quad i,j\in \{1,\dots,2D\}, \label{eq:gammaij2}\\
    |\gamma_{ij} (t,\mathbf s )| 
    &\leq \kappa \Xi\{\delta(t, \mathbf{s})\}, \quad i\in \{1,\dots,D\}, j\in \{1,\dots,2D\}. \label{eq:gammaij3}
\end{align}
\end{proposition}

\begin{proof}
Item (i) is proved in \cite{azais2024multivariate} (recall that \ref{C:integrability} implies that $\bX$ admits a spectral density). The upper bound~\eqref{eq:gammaij1} comes from the fact that $\gamma_{ij}$ here corresponds to $\gamma_{ab}$ of Lemma~\ref{lem:gamma} with $a$ and $b$ being vectors full of zeros except for $a_i=1$ and $b_j = 1$. Items (ii)-(iii) and their respective upper bounds follow similarly.
\end{proof}

Proposition~\ref{prop:Y:as:MWI} can be used to represent the chaotic components of $\cent{\Phi}_{\beta,n}$ (see Proposition \ref{prop:chaos:as:MWI} below) but more notation is needed. First, we use $\tilde{D}$ to denote either $D$ or $2D$. For any two functions $\psi_1,\psi_2\in \mathcal{H}$, we denote by $\psi_1\otimes \psi_2$ the tensor product function defined on $(\mathbb{R}^d)^2$ by $(\psi_1\otimes \psi_2)(\omega_1,\omega_2) = \psi_1(\omega_1)\psi_2(\omega_2)$. In turn, for any $\psi_j\in\mathcal{H},j\leq \tilde{D}$ and $\tilde a \in \{1,\dots,\tilde{D}\}^q$, we denote the following function on $(\mathbb{R}^d)^{q}$
\begin{equation}\label{eq:definition:psi:k:fold}
    \psi_{\tilde{a}} = \psi_{\tilde a_1}\otimes\dots\otimes \psi_{\tilde a_q}.
\end{equation}
Finally, the $q$-fold multiple Wiener-Itô integral (MWI) w.r.t. $\hat{B}$ is defined by
\begin{equation*}
    I_q\big(\psi_{\tilde a}\big) = I_q\big(\psi_{\tilde a_1}\otimes\dots\otimes \psi_{\tilde a_q}\big) = \prod^{\tilde{D}}_{j=1} H_{a_j}\big\{I_1(\psi_{j})\big\},
\end{equation*}
where $a_j = \#\{i=1,\dots,q: \tilde a_i = j\}$. Let $\mathcal{H}_q$ denote the domain of $I_q$, namely, the set of those $\psi:(\mathbb{R}^d)^q\to \mathbb{C}$ such that 
$\psi(-\omega_1,\dots,-\omega_q) =\overline{\psi(\omega_1,\dots,\omega_q)}$ and 
$\|\psi\|^2_{\mathcal{H}_q} =\int_{\R^{dq}}| \psi(\omega_1,\dots,\omega_q)|^2\dd\omega_1\dots \dd\omega_q$ is finite. The latter integral defines the norm and (thus, also) the inner product in $\mathcal{H}_q$. 
We introduce for technical reasons the symmetrization of  $\psi\in \mathcal{H}_q$: 
\begin{equation} \label{d:sym}
 \operatorname{Sym}(\psi)(\omega_1,\dots,\omega_q) := \frac{1}{q!}\sum_{\pi\in {\cal S}_q} \psi\{\omega_{\pi(1)},\dots,\omega_{\pi(q)}\}
\end{equation}
where ${\cal S}_q$ is the group of permutations of $\{1,\dots,q\}$. For $\psi\in\mathcal{H}_q$, we have that $I_q(\psi) = I_q\{\operatorname{Sym}(\psi)\}$. Moreover, $I_q$ is an isometry from $\mathcal{H}_q^s:=\{\psi\in \mathcal{H}_q: \psi=\operatorname{Sym}(\psi)\}$, with the modified norm 
$\sqrt{q!}\|\cdot\|_{\mathcal{H}_q}$, onto its image, which is called
the $q$th Wiener chaos ${\cal K}_q$. Finally, we introduce the
contraction operator. For $r\leq p\wedge q$, denote by $\bar{\otimes}_r$ the $r$th contraction operator:
$\psi_1\in\mathcal{H}^s_p, \psi_2\in\mathcal{H}^s_q\mapsto \psi_1\bar{\otimes}_r \psi_2 \in \mathcal{H}_{p+q-2r}$ given by 
\begin{multline*}
 \psi_1\bar{\otimes}_r \psi_2 (\omega_1,\dots,\omega_{p+q-2r}) \\
 =\int_{(\R^{d})^r} \psi_1(z_1,\dots,z_r;\omega_1,\dots,\omega_{p-r}) 
  \psi_2(-z_1,\dots,-z_r;\omega_{p-r+1},\dots,\omega_{p+q-2r})
 \dd z_1\dots \dd z_r.
\end{multline*}
The next result is the main tool and is an adaptation of Theorem 11.8.1 in \cite{peccati2011wiener}. The third item is the most useful for our purpose.
\begin{theorem} \label{thm:cltMWI}
Let $Q\geq 1$. 
For $q\leq Q$, consider the sequence of kernels $(\psi_{q,n})_{n\geq 1}$ with $\psi_{q,n}\in\mathcal{H}^s_q$. 
Then, as $n\to\infty$, the following conditions are equivalent:
\begin{enumerate}
  \item $\big\{I_1(\psi_{1,n}),\dots,I_{Q}(\psi_{Q,n})\big\}$ converges in distribution towards a centered normal random vector with variance $\operatorname{diag}(\sigma^2_1,\dots,\sigma^2_Q)$.
  \item For each $q\leq Q$, $I_q(\psi_{q,n})$ converges in distribution towards a centered normal random variable with variance $\sigma^2_q$.
  \item For each $q\leq Q$ and $r=1,\dots,q-1$, the norms $\|\psi_{q,n}\bar{\otimes}_r \psi_{q,n}\|_{\mathcal{H}_2(q-r)}$ 
tend to $0$. 
\end{enumerate}
\end{theorem}
We are now in position to state the MWI representation of the chaos components.
For each $a\in \mathbb N^{D}$ such that $|a|= q$, let us denote
$\mathcal{I}_a = \{\tilde a\in \{1,\dots,D\}^q: \sum_{j=1}^{q}
\mathbf{1}_{\{i\}}(\tilde a_j) = a_i,\, \forall i\leq D\}$ and, for
all $\tilde a\in \mathcal{I}_a$, $ \tilde d_{\tilde a} = d_a /
\#(\mathcal{I}_a)$. For each $\mathbf{a} \in \mathbb{R}^{2D}$, denote
analogously $\mathcal{I}_{\mathbf{a}}$ and
$\tilde{d}_{\tilde{\mathbf{a}}}(\mathbf{t}) =
\tilde{d}_{\tilde{\mathbf{a}}}(\|t^{(1)} - t^{(2)}\|) =
d_\mathbf{a}(\|t^{(1)} - t^{(2)}\|) / \#(\mathcal{I}_\mathbf{a})$ for
all $\tilde{\mathbf{a}} \in \mathcal{I}_{\mathbf{a}}$. The proposition
below follows exactly from \cite[Proposition
3.5.]{azais2024multivariate} when $\beta_2=0$ and can be extended to
$\beta_2 \neq 0$ using the same kind of arguments {as for $\beta_2=0$.}
\begin{proposition}\label{prop:chaos:as:MWI}
Assume \ref{C:general}[5], \ref{C:nondegeneracy}[4] and \ref{C:nondegeneracy:hessian}. For all $q\geq 1$, $n^{d/2} \, s_{q} = I_q( \, g_{q,n})$ where
\begin{multline*}
g_{q,n} = \beta_1 \sum_{\tilde a \in \{1,\dots,D\}^q} \tilde d_{\tilde a} \int_{W_n} n^{d/2}\,\phi_{1,n}(t)\, \psi_{t,\tilde a} \dd t \\
+ \beta_2 \sum_{\tilde{\mathbf{a}} \in \{1,\dots,2D\}^q} \int_{(W_n)^2} n^{d/2}\,\phi_{2,n}(\mathbf{t})\, \tilde d_{\tilde{\mathbf{a}}}(\mathbf{t})\,  \psi_{\mathbf{t},\tilde{\mathbf{a}}} \dd \mathbf{t}.
\end{multline*}
In particular, $g_{q,n}$ is a function in $\mathcal{H}^s_q$.
Notice that $\psi_{t,\tilde{a}}$ and $\psi_{\mathbf{t},\tilde{\mathbf{a}}}$ are defined according to~\eqref{eq:definition:psi:k:fold}. For instance, $\psi_{t,\tilde a} = \psi_{t,\tilde a_1} \otimes \dots \otimes \psi_{t,\tilde a_q}$.
\end{proposition}

\subsection{Proof of Theorem~\ref{thm:clt}(i)}

\begin{proof}
From Theorem~\ref{thm:variance}, the result is proved if for any $Q\ge 1$, $n^{d/2} \sum_{q\leq Q} s_q$ tends to a centered normal distribution as $n\to \infty$. So, according to Theorem \ref{thm:cltMWI}, we are to prove that the norms of the contractions $g_{q,n}\bar{\otimes}_{r} g_{q,n}$, $r=1,\dots,q-1$ tend to zero as $n\to\infty$. Since $q$ is fixed, by linearity, let us concentrate on the norm of the contraction of two normalized integrals as those in the two sums of $g_{q,n}$. There are three types of terms: two pure terms (corresponding to the linear and bilinear statistics) and a mixed term. Let us start with the simpler one (which already appears in \cite{estrade2016central} for instance).

\paragraph{Step 1.}
By linearity, the coefficients $\tilde{d}_{\tilde{a}}$ play no role in the convergence and so it is sufficient, for $\tilde a,\tilde b \in \{1,\dots,D\}^{q}$, to upper bound 
\begin{equation*}
    \mathcal{E} := \bigg\| \Big\{\int_{W_n} n^{d/2}\phi_{1,n}(t)\, \psi_{t,\tilde a} \dd t\Big\} \bar{\otimes}_r 
\Big\{\int_{W_n} n^{d/2}\phi_{1,n}(s)\, \psi_{s,\tilde b} \dd s\Big\}
\bigg\|_{\mathcal{H}_{2(q-r)}}^2.
\end{equation*}
Using Proposition \ref{prop:Y:as:MWI}, and the isometry \eqref{eq:MWI:isometry}, we have for any $\tilde{a},\tilde{b} \in \{1,\dots,D\}^q$,
\begin{equation*}
    \psi_{t,\tilde a} \bar{\otimes}_r 
    \psi_{s,\tilde b} = \prod^{r}_{k=1}\gamma_{\tilde a_k \tilde b_k}(t-s)\cdot \left(\bigotimes^{q}_{k=r+1}\psi_{t,\tilde a_k}\right)\otimes \left(\bigotimes^{q}_{k=r+1}\psi_{s,\tilde b_k}\right).
\end{equation*} 
Moreover, using the definition of $\|\cdot\|_{\mathcal{H}}$ and Equation \eqref{eq:MWI:isometry} once again, we have for any $i=1,\dots,D$
\begin{equation*}
    \left\| \int_{W_n} \psi_{t,i} \,\dd t \right\|_{\mathcal{H}}^2 = \int_{(W_n)^2} \left< \psi_{t,i},\psi_{t^\prime,i} \right>_{\mathcal{H}} \dd t \dd t^\prime = \int_{(W_n)^2} \gamma_{ii}(t-t^\prime) \dd t \dd t^\prime.
\end{equation*}
Following the same kind of computation, one ends up with
\begin{multline*}
        \mathcal{E} = n^{2d} \int_{W_n}\int_{W_n}\int_{W_n}\int_{W_n} \phi_{1,n}(t)\phi_{1,n}(s)\phi_{1,n}(t^\prime)\phi_{1,n}(s^\prime) \times \\
        \prod^{r}_{k=1}\gamma_{\tilde a_k \tilde b_k}(t-s)
        \prod^{r}_{k=1}\gamma_{\tilde a_k \tilde b_k}(t^\prime-s^\prime) 
        \prod^{q}_{k=r+1}\gamma_{\tilde a_k \tilde a_k }(t-t^\prime)
        \prod^{q}_{k=r+1}\gamma_{\tilde b_k \tilde b_k}(s-s^\prime)
        \ \dd t \dd s \dd t^\prime \dd s^\prime.
\end{multline*}

By~\eqref{eq:assphi1hat}, we have $\sup_{t\in W_n}|\phi_{1,n}(t)|=\|\phi_{1,n,W_n}\|_\infty=O(n^{-d})$. Hence, from Proposition~\ref{prop:Y:as:MWI}, we bound the previous squared norm, namely $\mathcal E$, following~\cite{estrade2016central,azais2024multivariate} as
\begin{align}
\mathcal E  &\leq \kappa  n^{-2d} \int_{(W_n)^{4}} \Xi(\|t-s\|)^r \Xi(\|t^\prime-s^\prime\|)^r \Xi(\|t-t^\prime\|)^{q-r}\times \Xi(\|s-s^\prime\|)^{q-r} \dd t \dd s \dd t^\prime \dd s^\prime \label{eq:bound:E:order:1}\\ 
&\leq \kappa n^{-d} \int_{(\mathbb{R}^d)^3} \Xi(\|v_1\|)^r \, \Xi(\|v_2\|)^r \, \Xi(\|v_3\|)^{q-r}  \dd v_1 \dd v_2 \dd v_3 =O(n^{-d})\nonumber
\end{align}
where the last inequality comes from: 1) the isometric change of
variables $v_1 = t-s$, $v_2 = t^\prime-s^\prime$, $v_3 = t-t^\prime$,
$v_4 = s^\prime$, 2) the fact that $\Xi$ is bounded so that the
integral with respect to $v_4$ is of order $n^{d}$, and 3) enlarging
the domain to the whole of $\mathbb{R}^d$.

\paragraph{Step 2.}
By linearity, using the fact that $\tilde{d}_{\tilde{\mathbf{a}}}$ is bounded on $(W_n)^2 \cap D_{\eta,R}$, it is sufficient, for $\tilde{\mathbf{a}},\tilde{\mathbf{b}} \in \{1,\dots,2D\}^{q}$, to upper bound 
\begin{equation*}
    \mathcal{E} := \bigg\| \Big\{\int_{(W_n)^2} n^{d/2}\,\phi_{2,n}(\mathbf{t})\, \psi_{\mathbf{t},\tilde{\mathbf{a}}} \dd \mathbf{t}\Big\} \bar{\otimes}_r 
    \Big\{\int_{(W_n)^2} n^{d/2}\,\phi_{2,n}(\mathbf{s})\,  \psi_{\mathbf{s},\tilde{\mathbf{b}}} \dd \mathbf{s}\Big\}
\bigg\|_{\mathcal{H}_{2(q-r)}}^2.
\end{equation*}
Following the same kind of computation as for the previous step, one ends up with
\begin{multline*}
    \mathcal{E} = n^{2d} \int_{((W_n)^2)^{4}} \phi_{2,n}(\mathbf{t})\phi_{2,n}(\mathbf{s})\phi_{2,n}(\mathbf{t}^\prime)\phi_{2,n}(\mathbf{s}^\prime) \times \\
    \prod^{r}_{k=1}\gamma_{\tilde{\mathbf{a}}_k \tilde{\mathbf{b}}_k}(\mathbf{t} \ominus \mathbf{s})
    \prod^{r}_{k=1}\gamma_{\tilde{\mathbf{a}}_k \tilde{\mathbf{b}}_k}(\mathbf{t}^\prime \ominus \mathbf{s}^\prime)\prod^{q}_{k=r+1} \gamma_{\tilde{\mathbf{a}}_k \tilde{\mathbf{a}}_k }(\mathbf{t} \ominus \mathbf{t}^\prime)
    \prod^{q}_{k=r+1}\gamma_{\tilde{\mathbf{b}}_k \tilde{\mathbf{b}}_k}(\mathbf{s} \ominus \mathbf{s}^\prime)
    \ \dd \mathbf{t} \dd \mathbf{s} \dd \mathbf{t}^\prime \dd \mathbf{s}^\prime.
\end{multline*}
Since the integrand is null except for $\mathbf{t}, \mathbf{s} \in (W_n)^2 \cap D_{\eta,R}$, we have $\gamma_{ij}(\mathbf{t} \ominus \mathbf{s}) \leq \kappa \Xi\{\delta(\mathbf{t}, \mathbf{s})\}$ by Proposition~\ref{prop:Y:as:MWI} and $\delta(\mathbf{t}, \mathbf{s}) \geq \| t^{(1)} - s^{(1)} \| - 2R-\sqrt d$ so that we are able to control $\gamma_{ij}(\mathbf{t} \ominus \mathbf{s})$ by a function that depends only on $t^{(1)}, s^{(1)} \in \mathbb{R}^d$ like in the previous section. Furthermore, $\phi_{2,n}(\mathbf{t}) \leq | (W_n)_{\ominus R} |^{-1} \leq 2 n^{-d}$ as soon as $n$ is large enough.
All in all, we end up with an integral similar to the one in the first two lines of~\eqref{eq:bound:E:order:1} except that $\|t - s\|$ is replaced by $\|t^{(1)}-s^{(1)}\| -2R-\sqrt d$ and so on. Since we are only interested in integrability conditions, the offset $- 2R-\sqrt d$ has no influence and so we conclude that $\mathcal{E} = O(n^{-d})$ like in the previous step.

\paragraph{Step 3.}
Once again, the coefficients play no role so it is sufficient, for $\tilde{a}\in \{1,\dots,D\}^q$ and $\tilde{\mathbf{b}} \in \{1,\dots,2D\}^{q}$, to upper bound 
\begin{equation*}
    \mathcal{E} := \bigg\| \Big\{\int_{W_n} n^{d/2}\,\phi_{1,n}(t)\, \psi_{t,\tilde{a}} \dd t\Big\} \bar{\otimes}_r 
    \Big\{\int_{(W_n)^2} n^{d/2}\,\phi_{2,n}(\mathbf{s})\,  \psi_{\mathbf{s},\tilde{\mathbf{b}}} \dd \mathbf{s}\Big\}
\bigg\|_{\mathcal{H}_{2(q-r)}}^2.
\end{equation*}
Following the same kind of computation as for the two previous steps, one ends up with
\begin{multline*}
    \mathcal{E} = n^{2d} \int_{(W_n)^2 \times ((W_n)^2)^{2}} \phi_{1,n}(t)\phi_{2,n}(\mathbf{s})\phi_{1,n}(t^\prime)\phi_{2,n}(\mathbf{s}^\prime) \times \\
    \prod^{r}_{k=1}\gamma_{\tilde a_k \tilde{\mathbf{b}}_k}(t \ominus \mathbf{s})
    \prod^{r}_{k=1}\gamma_{\tilde a_k \tilde{\mathbf{b}}_k}(t^\prime \ominus \mathbf{s}^\prime)\prod^{q}_{k=r+1} \gamma_{\tilde a_k \tilde a_k }(t - t^\prime)
    \prod^{q}_{k=r+1}\gamma_{\tilde{\mathbf{b}}_k \tilde{\mathbf{b}}_k}(\mathbf{s} \ominus \mathbf{s}^\prime)
    \ \dd t \dd t^\prime \dd \mathbf{s}  \dd \mathbf{s}^\prime.
\end{multline*}
Following the arguments of the previous step, we conclude that $\mathcal{E} = O(n^{-d})$ once again.

In conclusion, all the contributions of the contractions $g_{q,n}\bar{\otimes}_{r} g_{q,n}$ tend to zero as $n\to \infty$ and the CLT follows for $\cent{\Phi}_{\beta, n}$.
\end{proof}

\subsection{Proof of Theorem~\ref{thm:clt}(ii)} \label{app:CLTii}

\begin{proof}
Let $\check \zeta_\L:= \{ \cent{\Phi}_{1,n}(\phi_1) ,
\cent{\Phi}_{2,n}(\phi_2^1),\dots,\cent{\Phi}_{2,n}(\phi_2^m) \}^\top$
with $\phi_1(t)=\mathbf 1(t \in [-1/2,1/2]^d)$ and $\phi_2^j = \mathbf
1(t \in [\eta, r_j])$ for $j=1,\dots,m$. Note that any linear
combinations of linear (resp.\ billinear) statistics each
satisfying~\eqref{eq:assphi1hat} (resp.\ \eqref{eq:assphi2n}) still
satisfies \eqref{eq:assphi1hat} (resp.\
\eqref{eq:assphi2n}). Therefore, the asymptotic normality of $\check{\zeta}_\L$ easily ensues from Theorem~\ref{thm:clt}(i). And, using
Theorem~\ref{thm:variance}(iv), we deduce that $n^{d/2} \check{\zeta}_\L \to \mathcal N(0,\check\Sigma)$ where $\check \Sigma$ is the
symmetric $(m+1) \times (m+1)$ matrix defined for $j>i$ by
\begin{equation*}
(\check\Sigma)_{ij} :=
\left\{
\begin{array}{ll}
\mathcal C(\phi_1,\phi_2^{j-1})& \text{if }i=1,j>1\\
\mathcal C(\phi_2^{i-1},\phi_{2}^{j-1}) & \text{otherwise}
\end{array}
\right. 
\end{equation*}
with again $\mathcal C(\phi_i,\phi_i) = \mathcal V(\phi_i)$ for $i=1,2$. Observe that $\hat \rho_\L-\rho_\L = \cent{\Phi}_{1,n}(\phi_1)$ and that for any $j=1,\dots,m$ by definition of the modified Ripley's $K$-function,
\begin{equation}\label{eq:behaviorK}
\hat K_{\eta,\L}(r_j) - K_{\eta,\L}(r_j) = \alpha_n \cent{\Phi}_{2,n}(\phi_2^j) + \beta_{n,j} \cent{\Phi}_{1,n}(\phi_1)
\end{equation}
with $\alpha_n= \hat \rho_\L^2/\rho_\L^2$ and $\beta_{n,j}={K_{\eta,\L}(r_j)}(\hat \rho_\L+\rho_\L)/\rho_\L^2$.
Theorem~\ref{thm:variance} implies in particular that, in probability, $\alpha_n\to 1$ and $\beta_{n,j}\to \beta_j=2K_{\eta,\L}(r_j)$ as $n\to \infty$. Therefore, Slutsky's lemma ensures that as $n\to \infty$, the limit distribution of $n^{d/2} \zeta_\L$ is the same as the limit of $n^{d/2} h(\check \zeta_\L)$ where $h :\R^{m+1}\to \R^{m+1}$ is given by
$$
h(u)=h(u_1,u_2,\dots, u_{m+1}) =  (u_1,u_2+\beta_1 u_1,\dots,u_{m+1} + \beta_m u_1).
$$
Note that $h(0)=0 \in \R^{m+1}$, that $h$ is differentiable at $0$ and
the Jacobian matrix at 0 {is the symmetric  $(m+1) \times (m+1)$ matrix} defined for $j\ge i$ by 
$$
J: = \left\{
\begin{array}{ll}
\beta_{i} & \text{if } i=1 \text{ and } j>1 \\
1 & \text{if } i=j \\
0 & \text{otherwise}.
\end{array}    
\right.
$$
We conclude by using the multivariate delta method to show that as $n\to \infty$,
\begin{equation*}
n^{d/2} \zeta_\L   \stackrel{\mathcal D}{\to} \mathcal N\left\{ 0, \Sigma(\rho_\L,K_{\eta,\L})  \right\}
\quad \text{ with  } \quad \Sigma(\rho_\L,K_{\eta,\L}) = J \check\Sigma J^\top.
\end{equation*}
\end{proof}


\bibliographystyle{imsart-number} 
\bibliography{refs}

\end{document}